\documentclass[11pt]{amsart}

\usepackage{amscd,amssymb,amsopn,amsmath,amsthm,mathrsfs,graphics,amsfonts,enumerate,verbatim,calc
}
\usepackage[all]{xypic}
\usepackage{url}

\usepackage[OT2,OT1]{fontenc}
\newcommand\cyr{%
\renewcommand\rmdefault{wncyr}%
\renewcommand\sfdefault{wncyss}%
\renewcommand\encodingdefault{OT2}%
\normalfont
\selectfont}
\DeclareTextFontCommand{\textcyr}{\cyr} 

\usepackage{amssymb,amsmath}

\DeclareFontFamily{OT1}{rsfs}{}
\DeclareFontShape{OT1}{rsfs}{n}{it}{<-> rsfs10}{}
\DeclareMathAlphabet{\mathscr}{OT1}{rsfs}{n}{it}

\numberwithin{equation}{section}
\newtheorem{theorem}{Theorem}[section]
\newtheorem{lemma}[theorem]{Lemma}
\newtheorem{proposition}[theorem]{Proposition}
\newtheorem{corollary}[theorem]{Corollary}

\newtheorem{question}{Question}
\newtheorem{Problem}{Problem}

\newtheorem{Maintheorem}{Main Theorem}

\theoremstyle{definition}
\newtheorem{definition}[theorem]{Definition}
\newtheorem{remark}[theorem]{Remark}
\newtheorem{Discussion}[theorem]{Discussion}
\theoremstyle{remark}
\newtheorem{example}[theorem]{Example}
\newtheorem*{acknowledgement}{Acknowledgement}

\newcommand{\im}{\operatorname{Im}}
\renewcommand{\ker}{\operatorname{Ker}}

\newcommand{\id}{\operatorname{id}}

\newcommand{\Hom}{\operatorname{Hom}}

\newcommand{\coker}{\operatorname{Coker}}

\newcommand{\Spa}{\operatorname{Spa}}
\newcommand{\Frac}{\operatorname{Frac}}
\newcommand{\Val}{\operatorname{Val}}

\newcommand{\fm}{\frak{m}}
\newcommand{\fp}{\frak{p}}
\newcommand{\fq}{\frak{q}}
\newcommand{\fa}{\frak{a}}

\newcommand{\fn}{\frak{n}}

\begin{document}
\title[A variant of perfectoid Abhyankar's lemma and almost Cohen-Macaulay algebras]
{A variant of perfectoid Abhyankar's lemma and almost Cohen-Macaulay algebras}
\dedicatory{Dedicated to our teachers, Professor Kazuhiro Fujiwara and Professor Paul Roberts.}

\author[K.Nakazato]{Kei Nakazato}
\address{Graduate School of Mathematics, Nagoya University, 
Nagoya 464-8602, Japan}
\email{keinakazato31@gmail.com}

\author[K.Shimomoto]{Kazuma Shimomoto}
\address{Department of Mathematics, Tokyo Institute of Technology, 2-12-1 Ookayama, Meguro, Tokyo 152-8551, Japan}
\email{shimomotokazuma@gmail.com}

\thanks{2020 {\em Mathematics Subject Classification\/}: 13A18, 13A35, 13B05, 13D22, 13J10, 14G22, 14G45}

\keywords{Almost purity theorem, big Cohen-Macaulay algebra, \'etale extension, perfectoid space}

%\subjclass{13}
%\subjclass[2000]{Primary 13-XX}
%\subjclass[2000]{Primary ; Secondary}

%\date{\today \, (\printtime)}
%\date{\today}

\begin{abstract} 
In this paper, we prove that a complete Noetherian local domain of mixed characteristic $p>0$ with perfect residue field has an integral extension that is an integrally closed, almost Cohen-Macaulay domain such that the Frobenius map is surjective modulo $p$. This result is seen as a mixed characteristic analogue of the fact that the perfect closure of a complete local domain in positive characteristic is almost Cohen-Macaulay. To this aim, we carry out a detailed study of decompletion of perfectoid rings and establish the Witt-perfect (decompleted) version of Andr\'e's perfectoid Abhyankar's lemma and Riemann's extension theorem.
\end{abstract}

\maketitle

\setcounter{tocdepth}{3}
\tableofcontents

\section{Introduction}

In the present article, rings are assumed to be commutative with a unity. Recently, Yves Andr\'e established \textit{Perfectoid Abhyankar's Lemma} in \cite{An1} as a conceptual generalization of \textit{Almost Purity Theorem}; see \cite[Theorem 7.9]{Sch12}. This result is stated for perfectoid algebras over a perfectoid field, which are defined to be certain $p$-adically complete and separated rings. Using his results, Andr\'e proved the existence of big Cohen-Macaulay algebras in mixed characteristic in \cite{An2}. More precisely, he constructed a certain \textit{almost} Cohen-Macaulay algebra using perfectoids. We are inspired by this result and led to the following commutative algebra question, which is raised in \cite{R08} and \cite{R10} implicitly.

\begin{question}[Roberts]
\label{PRoberts}
Let $(R,\fm)$ be a complete Noetherian local domain of arbitrary characteristic with its absolute integral closure $R^+$. Then does there exist an $R$-algebra $B$ such that $B$ is an almost Cohen-Macaulay $R$-algebra and $R \subset B \subset R^+$?
\end{question}

Essentially, Question \ref{PRoberts} asks for a possibility to find a \textit{relatively small} almost Cohen-Macaulay algebra. The structure of this article is twofold. We begin with giving an answer to Question \ref{PRoberts} and then discuss necessary perfectoid techniques.

\subsection{Main results on commutative algebra}

Andr\'e proved that any complete Noetherian local domain maps to a big Cohen-Macaulay algebra and using his result, it was proved that such a big Cohen-Macaulay algebra could be refined to be an integral perfectoid big Cohen-Macaulay algebra in \cite{Sh18}. See Definition \ref{BigMac} and Definition \ref{DefAlmost} for big (almost) Cohen-Macaulay algebras. Question \ref{PRoberts} was stated in a characteristic-free manner. Let us point out that if $\dim R \le 2$, then $R^+$ is a big Cohen-Macaulay algebra in an arbitrary characteristic. This is easily seen by using Serre's normality criterion. Recall that if $R$ has prime characteristic $p>0$, then $R^+$ is a big Cohen-Macaulay $R$-algebra. This result was proved by Hochster and Huneke and their proof is quite involved; see \cite{H07}, \cite{HH92}, \cite{HH95}, \cite{HL07}, \cite{Q16} and \cite{SS12} for related results as well as \cite{HM17}, \cite{HM18}, \cite{MS18}, \cite{Mu21} and \cite{Mu22} for applications to tight closure, multiplier/test ideals and singularities on algebraic varieties. There is another important work on the purity of Brauer groups using perfectoids; see \cite{CK19}. It seems to be an open question whether $R^+$ is almost Cohen-Macaulay when $R$ has equal-characteristic zero. If $R$ has mixed characteristic of dimension $3$, Heitmann proved that $R^+$ is a $(p)^{\frac{1}{p^\infty}}$-almost Cohen-Macaulay $R$-algebra in \cite{H02}. Our main concern, inspired also by the recent result of Heitmann and Ma \cite{HM18}, is to extend Heitmann's result to the higher dimensional case, thus giving a positive answer to Roberts' question in mixed characteristic; see Theorem \ref{AlmostCMalg}.

\begin{Maintheorem}
\label{AlmostBigCM}
Let $(R,\fm)$ be a complete Noetherian local domain of mixed characteristic $p>0$ with perfect residue field $k$. Let $p,x_2,\ldots,x_d$ be a system of parameters and let $R^+$ be the absolute integral closure of $R$. Then there exists an $R$-algebra $T$ together with a nonzero element $g \in R$ such that the following hold:
\begin{enumerate}
\item
$T$ admits compatible systems of $p$-power roots $p^{\frac{1}{p^n}}, g^{\frac{1}{p^n}} \in T$ for all $n>0$.

\item
The Frobenius endomorphism $Frob:T/(p) \to T/(p)$ is surjective.

\item
$T$ is a $(pg)^{\frac{1}{p^\infty}}$-almost Cohen-Macaulay normal domain with respect to $p,x_2,\ldots,x_d$ and $R \subset T \subset R^+$.

\item
The $p$-adic completion $\widehat{T}$ is integral perfectoid.

\item
$R[\frac{1}{pg}] \to T[\frac{1}{pg}]$ is an ind-\'etale extension. In other words, $T[\frac{1}{pg}]$ is a filtered colimit of finite \'etale $R[\frac{1}{pg}]$-algebras contained in $T[\frac{1}{pg}]$.
\end{enumerate}
\end{Maintheorem}

In other words, one can find an almost Cohen-Macaulay, normal domain whose $p$-adic completion is integral perfectoid between $R$ and its absolute integral closure. Using Hochster's partial algebra modification and tilting, one can construct an integral perfectoid big Cohen-Macaulay $R$-algebra over $T$; see \cite{Sh18} for details. In a sense, Main Theorem \ref{AlmostBigCM} is regarded as a weak analogue of the mixed characteristic version of a result by Hochster and Huneke. The proof of our result does not seem to come by merely considering decompleted versions of the construction by Heitmann and Ma in \cite{HM18}, due to the difficulty of studying $(pg)^{\frac{1}{p^\infty}}$-almost (or $(g)^{\frac{1}{p^\infty}}$-almost) mathematics under $p$-adic completion. For instance, we do not know if it is possible to decomplete Andr\'e's Riemann's extension theorem (Theorem \ref{Hebbarkeits1}), because it is hard to analyze how the functor $g^{-\frac{1}{p^\infty}}(~)$ and the $p$-adic completion are related to each other.\footnote{The theory of prismatic cohomology \cite{BS22} deals with $(pg)^{\frac{1}{p^\infty}}$-almost mathematics. However, as far as the authors are aware of, we are still lacking in a powerful theory of the decompletion in the framework of $(pg)^{\frac{1}{p^\infty}}$-almost mathematics.} This is the main reason one is required to redo the decompletion of Andr\'e's results in \cite{An1} and \cite{An2}.

Finally, Bhatt recently proved that the absolute integral closure of a complete local domain $(R,\fm)$ of mixed characteristic has the property that $R^+/p^nR^+$ is a balanced big Cohen-Macaulay $R/p^nR$-algebra for any $n>0$; see \cite{Bh20} in which the perfectoidization functor introduced in \cite{BS22} is used as an essential tool. It will be interesting to know how his methods and results are compared to ours; at present, the authors have no clue. However, it is worth pointing out the following fact.
\begin{enumerate}
\item[$\bullet$]
The almost Cohen-Macaulay algebra $T$ constructed in Main Theorem \ref{AlmostBigCM} is integral over the Noetherian local domain $(R,\fm)$ and much smaller than the absolute integral closure $R^+$.
\end{enumerate}

In a sense, $T$ is close to being a Noetherian ring. We mention some potential applications of Main Theorem \ref{AlmostBigCM}.

\begin{enumerate}
\item
Connections with the singularities studied in \cite{MS19} by exploiting the ind-\'etaleness of $R[\frac{1}{pg}] \to T[\frac{1}{pg}]$.

\item
A refined study of the main results on the closure operations in mixed characteristic as developed in \cite{ZJ20}.

\item
An explicit construction of a big Cohen-Macaulay module from the $R$-algebra $T$; see Corollary \ref{BigCMModule}.
\end{enumerate}

\subsection{Main results on the decompletion of perfectoids and Riemann's extension theorem}

To prove Main Theorem \ref{AlmostBigCM}, we need to relax the $p$-adic completeness from Perfectoid Abhyankar's Lemma and incorporate the so-called \textit{Witt-perfect} condition, which is introduced by Davis and Kedlaya in \cite{DK14}. Roughly speaking, a Witt-perfect (or $p$-Witt-perfect) algebra is a $p$-torsion free ring $A$ whose $p$-adic completion becomes an integral perfectoid ring. Indeed, Davis and Kedlaya succeeded in proving the almost purity theorem for Witt-perfect rings. The present article is a sequel to authors' previous work \cite{NS19}, in which the authors were able to give a conceptual proof of the almost purity theorem by Davis-Kedlaya by analyzing the integral structure of Tate rings under completion. The advantage of working with Witt-perfect rings is that it allows one to take an infinite integral extension over a certain $p$-adically complete ring to construct an almost Cohen-Macaulay algebra. The resulting algebra is not $p$-adically complete, but its $p$-adic completion is integral perfectoid. Let us state the main result; see Theorem \ref{PerAbhyankar} and Proposition \ref{preprop}.

\begin{Maintheorem}
\label{main}
Let $A$ be a $p$-torsion free algebra over a $p$-adically separated $p$-torsion free Witt-perfect valuation domain $V$ of rank $1$ admitting a compatible system of $p$-power roots $p^{\frac{1}{p^n}} \in V$, together with a regular element $g \in A$ admitting a compatible system of $p$-power roots $g^{\frac{1}{p^n}} \in A$. Suppose that the following conditions hold.
\begin{enumerate}
\item
$A$ is a $p$-adically Zariskian and normal ring.

\item
$A$ is a $(pg)^{\frac{1}{p^\infty}}$-almost Witt-perfect ring.

\item
$A$ is torsion free and integral over a Noetherian normal domain $R$ such that $g \in R$ and the height of the ideal $(p,g) \subset R$ is $2$.
\\

Let us put
$$
g^{-\frac{1}{p^{\infty}}}A:=\Big\{a \in A[\frac{1}{g}]~\Big|~g^{\frac{1}{p^n}} a \in A,~\forall n>0\Big\},
$$
which is an $A$-subalgebra of $A[\frac{1}{g}]$. Let $A[\frac{1}{pg}] \hookrightarrow B'$ be a finite \'etale extension, and 
denote by $B$ the integral closure of $g^{-\frac{1}{p^{\infty}}}A$
in $B'$. Then the following statements hold:

\begin{enumerate}
\item
The Frobenius endomorphism $Frob:B/(p) \to B/(p)$ is $(pg)^{\frac{1}{p^\infty}}$-almost surjective and it induces an injection $B/(p^{\frac{1}{p}}) \hookrightarrow B/(p)$.

\item
The induced map $A/(p^m) \to B/(p^m)$ is $(pg)^{\frac{1}{p^\infty}}$-almost finite \'etale for all $m>0$.
\end{enumerate}
\end{enumerate}
\end{Maintheorem}

In the original version of Perfectoid Abhyankar's Lemma as proved in \cite{An1} and \cite{An2}, it is assumed that $A$ is an \textit{integral perfectoid ring}, which is necessarily $p$-adically complete and separated. In Main Theorem \ref{main}, this assumption is weakened to $p$-adic Zariskianness.  A detailed study of \textit{almost Witt-perfect rings} appears in the paper \cite{NS19}; see also Definition \ref{AlmostWittRing} below. The functor $A \mapsto g^{-\frac{1}{p^\infty}}A$ is called a \textit{functor of almost elements}, which is fundamental in almost ring theory. The idea of the proof of Main Theorem \ref{main} is to transport Andr\'e's original proof to our situation. Here is a summary of ingredients toward the proof:

\begin{enumerate}
\item[$\bullet$]
The almost purity theorem for Witt-perfect rings.

\item[$\bullet$]
Descent to Galois extensions of commutative rings.

\item[$\bullet$]
Riemann's extension theorem (Hebbarkeitssatz).

\item[$\bullet$]
Description of the integral structures of affinoid Tate rings via continuous valuations.

\item[$\bullet$]
Comparison of integral closure and complete integral closure.
\end{enumerate}

The almost purity theorem for Witt-perfect rings is attributed to Davis and Kedlaya; see \cite{DK14} and \cite{DK15}. A systematic approach to this important result was carried out in authors' paper \cite{NS19}. 
The almost purity theorem yields the assertion of Main Theorem \ref{main} in the case when $g=1$. To extend it to the general situations, we need a ring theoretic analogue of \textit{Riemann's Extension Theorem}. Its perfectoid version has been proved by Scholze in \cite{Sch15}, and Andr\'e used it in the proof of Perfectoid Abhyankar's Lemma in \cite{An1}. We establish two types of decompleted variant of it, which are at the core of the technical part of this paper. The first one, which we call \emph{Zariskian Riemann's extension theorem}, is the following result; see Theorem \ref{RiemannExtadic}. We should remark that it is independent of the theory of perfectoid rings.

\begin{Maintheorem}
\label{MainZRET}
Let $A$ be a ring with a regular element $t$ that is $t$-adically Zariskian and integral over a Noetherian ring. Let $g\in A$ be a regular element. 
Let $A^{j}$ be the Tate ring associated to $\big(A[\frac{t^j}{g}], (t)\big)$ for every integer $j>0$ (see Definition \ref{TateRingDef} for Tate rings). Then we have an isomorphism of rings 
$$
A^+_{A[\frac{1}{tg}]} \xrightarrow{\cong} \varprojlim_{j>0}A^{j\circ}, 
$$ 
where the transition map $A^{j+1\circ}\to A^{j\circ}$ is the natural one, and $A^+_{A[\frac{1}{tg}]}$ is the integral closure of $A$ in $A[\frac{1}{tg}]$.
\end{Maintheorem}

For proving Main Theorem \ref{MainZRET}, a preliminary result Corollary \ref{Cont2} is crucial. 
Recall that an integrally closed domain $A$ is the intersection of all valuation domains that lie between $A$ and the field of fractions; see \cite[Proposition 6.8.14]{SH06} for the proof of this assertion from classical valuation theory. 
Corollary \ref{Cont2} is viewed as a variant of this result for affinoid Tate rings. 
The assumption that $A$ is \emph{$t$-adically Zariskian} and \emph{integral over a Noetherian ring} is necessary in order to find valuation rings of rank $1$ for the proof to work (it is interesting to know to what extent one can relax these assumptions). 
Main Theorem \ref{MainZRET} is also relevant to a standard technique used in non-archimedean geometry. 
Indeed, our proof for Proposition \ref{Cont1} is inspired by Huber's description of integral structures of affinoid rings via continuous valuations (\cite[Lemma 3.3]{Hu93}). 
Moreover, one can formulate Corollary \ref{Cont2} as a variant of Fujiwara-Kato's theorem (cf.\ {\cite[\textbf{II}, Theorem 8.1.11 and Theorem 8.2.19]{FK18}}) in rigid geometry; see Corollary \ref{Cont3} and Corollary \ref{1912132251}. 

The second variant, which we call \emph{Witt-perfect Riemann's extension theorem}, is stated as the assertions (c) and (d) of the following result; see Proposition \ref{pg-approx} and Theorem \ref{RiemannExt} as well as the explanation of appearing notation.

\begin{Maintheorem}
\label{main01}
Let $A$ be a $p$-torsion free algebra over a $p$-adically separated $p$-torsion free Witt-perfect valuation domain $V$ of rank $1$ admitting a compatible system of $p$-power roots  $p^{\frac{1}{p^n}} \in V$, together with a regular element $g \in A$ admitting a compatible system of $p$-power roots $g^{\frac{1}{p^n}} \in A$. Denote by $\widehat{(~)}$ the $p$-adic completion and suppose that the following conditions hold. 
\begin{enumerate}

\item
$A$ is a $(pg)^{\frac{1}{p^\infty}}$-almost Witt-perfect ring and completely integrally closed in $A[\frac{1}{p}]$.

\item
$(p,g)$ is a $(p)^{\frac{1}{p^\infty}}$-almost regular sequence on $A$ (which merely says that $g$ is a $(p)^{\frac{1}{p^\infty}}$-almost regular element on
$A/(p)$). 
\\

Then the following statements hold.
\begin{enumerate}
\item
The inclusion map: 
\begin{equation}
\label{Scholzeapprox2}
A\big[\big(\frac{p^j}{g}\big)^{\frac{1}{p^\infty}}\big] \hookrightarrow A^{j \circ}
\end{equation}
 is a $(p)^{\frac{1}{p^\infty}}$-almost isomorphism. 
\item
There is an $A[\frac{p^{j}}{g}]$-algebra isomorphism:
$$
\widehat{A^{j\circ}} \xrightarrow{\cong} \mathcal{A}^{j\circ}.
$$
Moreover, $A^{j\circ}$ is Witt-perfect.

\item
We have the following identification of rings: 
\begin{equation}
\label{fiberproduct1}
\varprojlim_{j>0}A^{j\circ}=A[\frac{1}{pg}]\times_{\mathcal{A}[\frac{1}{g}]}g^{-\frac{1}{p^\infty}}\mathcal{A}^{\circ}. 
\end{equation}
\item
There is an injective $A$-algebra map:
$$
\widehat{\varprojlim_{j>0} A^{j\circ}} \hookrightarrow \varprojlim_{j>0} \widehat{A^{j\circ}},
$$
whose cokernel is $(g)^{\frac{1}{p^\infty}}$-almost zero.
\end{enumerate}
\end{enumerate}
\end{Maintheorem}

The almost isomorphism in the assertion (d) is at the heart of the theorem; notice that in general, inverse limits and taking completion do not commute. Our proof for the assertions (c) and (d) relies on the already-known Riemann's extension theorem for perfectoid algebras. 
Thus we need to describe the relationship between rational localizations of Tate rings associated to a Witt-perfect ring and the corresponding integral perfectoid ring. The assertions (a) and (b) are consequences of a fine study on it. The $(p)^{\frac{1}{p^\infty}}$-almost regularity assumption on the sequence $(p,g)$ ensures that $g$ is $(p)^{\frac{1}{p^\infty}}$-almost regular on the $p$-adic completion $\widehat{A}$; this is due to Lemma \ref{adiccompletion}. Another reason for assuming $(p)^{\frac{1}{p^\infty}}$-almost regularity rather than $(pg)^{\frac{1}{p^\infty}}$-almost regularity is due to Lemma \ref{Aj-completion}. See also Proposition \ref{RegularRatLoc} as an intermediary step. We remark that there are no common assumptions of both Main Theorem \ref{MainZRET} and Main Theorem \ref{main01} on the ring $A$. Let us summarize the content of each section of the present paper.

In \S \ref{SecNotConv}, we give generalities on almost rings, almost modules and topological rings. We also recall the definitions of perfectoid algebras and their almost analogues whose detailed studies are given in the authors' paper \cite{NS19}.

In \S \ref{SecPreLem}, some basic results are proved on complete integral closure and its behavior under completion. We stress that the use of ``Beauville-Laszlo's lemma" is indispensable for getting meaningful results. This section is intended to give some justification/clarification on the difference between integral and complete integral closures.

In \S \ref{FEtTate}, we study some behavior of finite \'etale extensions of Tate rings under rational localization. This section is regarded as a complement to \cite{NS19}, and so also includes a brief review of several results in that paper. 

In \S \ref{SecPAbhy}, we establish the decompleted variants of Riemann's extension theorem as well as Perfectoid Abhyankar's lemma over almost perfectoid rings. As this section contains quite technical discussions, the reader can skip the details on first reading.

In \S \ref{SecAppWPA}, we give the main applications of the results obtained in the previous sections. The main theorem asserts that one can construct an almost Cohen-Macaulay normal domain between the original complete local domain and its absolute integral closure.

In \S \ref{AppendixA}, we prove auxiliary facts on integrality as well as almost integrality via rigid analytic methods, following the book \cite{FK18}.

In \S \ref{AppendixB}, we give a complete account of the proof of Andr\'e's Riemann's extension theorem. To this aim, we also give a proof of the almost vanishing theorem on derived limits, which is discussed in \cite{An1}. We hope that this appendix will be helpful for the reader to understand key results in Andr\'e's original approach.

In \S \ref{AlmostG-GaloisExt}, we give a brief account on (almost) Galois extensions of commutative rings. These are already treated in Andr\'e's paper \cite{An1} and we omit the proofs.

$\bf{Caution}$: In this paper, we take both integral closure and complete integral closure for a given ring extension. This distinction is not essential in our setting in view of Proposition \ref{propIC=CIC}. However, we opt to formulate the results (mostly) in complete integral closure, because we believe that correct statements of the possible generalizations of our main results without integrality over a Noetherian ring should be given in terms of complete integral closure. The reader is warned that complete integral closure is coined as \textit{total integral closure} in the lecture notes \cite{Bh17}. We collect notation used in the proof of Theorem \ref{PerAbhyankar} in Definition \ref{DefAj} (see also Remark \ref{DefAi2}).

The almost version of perfectoid or Witt-perfect rings often appears in the following discussions. To the best of authors' knowledge, the first appearance of almost perfectoid rings came from Andr\'e's work on Perfectoid Abhyankar's Lemma. The reason is that $(pg)^{\frac{1}{p^\infty}}$-almost mathematics is essential in Andr\'e's work, in which case we can only say that the cokernel of the Frobenius map is $(pg)^{\frac{1}{p^\infty}}$-surjective. The reader will notice that the base ring in Theorem \ref{PerAbhyankar} is required to be almost Witt-perfect in order for the proof to work. In a future's occasion, we hope to clarify a real distinction between perfectoid and almost perfectoid rings.

\section{Notation and conventions}

%Let $A$ be a ring with an element $t \in A$. Then we will denote by $\widehat{A}$ the $t$-adic completion of $A$. In most cases that we encounter, $t$ will be either a regular element or $t=p$, a fixed prime number. 

We say that a commutative ring $A$ is \textit{normal}, if the localization $A_{\fp}$ is an integrally closed domain in its field of fractions for every prime ideal $\fp \subset A$. For ring maps $A \to C$ and $B \to C$, we write $A \times_C B$ for the fiber product. The \textit{completion} of a module is always taken to be complete and separated. 
%We also consider non-adic completion when studying Banach rings. In \S \ref{non-adicBanachcompletion} we review its definition and some related results. In \S \ref{AlmostG-GaloisExt} we review some facts on (almost) Galois extension. 

\label{SecNotConv}
\subsection{Almost ring theory}
We use language of almost ring theory. The most comprehensive references are \cite{GR03} and \cite{GR18}, where the latter discusses applications of almost ring theory to algebraic geometry and commutative ring theory. Notably, it includes an extension of the Direct Summand Conjecture to the setting of log-regular rings. Throughout this article, for an integral domain $A$, let $\Frac(A)$ denote the field of fractions of $A$. A \textit{basic setup} is a pair $(A,I)$, where $A$ is a ring and $I$ is its ideal such that $I^2=I$.\footnote{As in \cite{GR03}, we assume that $I \otimes_A I$ is flat. For the applications, we only consider the case where $I$ is the filtered colimit of principal ideals; see \cite[Proposition 2.1.7]{GR03}.} An $A$-module $M$ is \textit{I-almost zero} (or simply \textit{almost zero}) if $I M=0$. Let $f:M \to N$ be an $A$-module map. Then we say that $f$ is \textit{I-almost injective} (resp. \textit{I-almost surjective} if the kernel (resp. cokernel) of $f$ is annihilated by $I$. Moreover, we say that $f$ is an \textit{I-almost isomorphism} (or simply an \textit{almost isomorphism}) if both kernel and cokernel of $f$ are annihilated by $I$. Let us define an important class of basic setup $(K,I)$ as follows: Let $K$ be a perfectoid field of characteristic 0 with a non-archimedean norm $|\cdot|:K \to \mathbf{R}_{\ge 0}$. Fix an element $\varpi \in K$ such that $|p| \le |\varpi|<1$ and $I:=\bigcup_{n>0} \varpi^{\frac{1}{p^n}}K^\circ$ (such an element $\varpi$ exists and plays a fundamental role in perfectoid geometry). Set $K^\circ:=\{x \in K~|~|x| \le 1\}$ and $K^{\circ \circ}:=\{x \in K~|~|x| < 1\}$. Then $K^\circ$ is a complete valuation domain of rank $1$ with field of fractions $K$ and the pair $(K^\circ,I)$ is a basic setup. 

Let $(A,I)$ be a basic setup. Then the category of almost $A$-modules or $A^a$-modules $A^a-\bf{Mod}$, is the quotient category of $A$-modules $A-\bf{Mod}$ by the Serre subcategory of $I$-almost zero modules. So this defines the localization functor $(~)^a:A-\mathbf{Mod} \to A^a-\mathbf{Mod}$. This functor admits a right adjoint and a left adjoint functors respectively:
$$
(~)_*:A^a-\mathbf{Mod} \to A-\mathbf{Mod}~\mbox{and}~(~)_!:A^a-\mathbf{Mod} \to A-\mathbf{Mod}.
$$
These are defined by $M_*:=\Hom_A(I,M_0)$ with $M_0^a=M$ and $M_!:=I \otimes_A M_*$. See \cite[Proposition 2.2.14 and Proposition 2.2.23]{GR03} for these functors. So we have the following fact: The functor $(~)_*$ commutes with limits and $(~)_!$ commutes with colimits. Finally, the functor $(~)^a$ commutes with both colimits and limits. In particular, an explicit description of $M_*$ will be helpful. Henceforth, we abusively write $M_*$ for $(M^a)_*$ for an $A$-module $M$. The notation $
M \xrightarrow{\approx} N$ will be used throughout to indicate that there is an $A$-homomorphism $M \to N$ that is an $I$-almost isomorphism. An isomorphism in the category $A^a-\mathbf{Mod}$ will be denoted by $M \approx N$.\footnote{This symbol is used when there is not necessarily an honest homomorphism between $M$ and $N$.} For technical details, we refer the reader to \cite{GR03}.

Let us recall an explicit description of $M_{*}$. 

\begin{lemma}
\label{almostelement}
Let $M$ be a module over a ring $A$ and let $\varpi \in A$ be an element such that $A$ admits a compatible system of $p$-power roots $\varpi^{\frac{1}{p^n}} \in A$ for $n \ge 0$. Set $I=\bigcup_{n>0} \varpi_n A$ with $\varpi_n:=\varpi^{\frac{1}{p^n}}$ and suppose that $\varpi$ is regular on both $A$ and $M$.
Then the following statements hold:
\begin{enumerate}
\item
$(A,I)$ is a basic setup.

\item
There is an equality:
$$
M_*=\Big\{b \in M[\frac{1}{\varpi}]~\Big|~\varpi_n b \in M~ \mbox{for all}~n>0\Big\}.
$$
Moreover, the natural map $M \to M_*$ is an $I$-almost isomorphism. If $M$ is an $A$-algebra, then $M_*$ has an $A$-algebra structure and the natural map $M \to M_*$ is an $A$-algebra map.
\end{enumerate}
\end{lemma}

\begin{proof}
The presentation for $M_*$ is found in \cite[Lemma 5.3]{Sch12} over a perfectoid field and the proof there works under our setting without any modifications. If $M$ is an $A$-algebra, then the above presentation will endow $M_*$ with an $A$-algebra structure. In other words, $M_*$ is naturally an $A$-subring of $M[\frac{1}{\varpi}]$. 
\end{proof}

In the situation of the lemma, we often write $M_*$ as $\bigcap_{n>0} \varpi^{-\frac{1}{p^n}}M$ or $\varpi^{-\frac{1}{p^\infty}}M$ to indicate that what basic setup of almost ring theory we are talking. Next we observe that $I$-almost isomorphy is preserved under pullbacks in the category of (actual) $A$-algebras.

\begin{lemma}\label{PullbackRings}
Let $(A, I)$ be a basic setup. 
Let $f: R\to S$ be an $A$-algebra homomorphism that admits a commutative diagram of $A$-algebras: 
\[\xymatrix{
R\ar[dr]_{\varphi_{R}}\ar[rr]^{f}&& S\ar[dl]^{\varphi_S}\\
&T&.
}\]
Let $\psi: T'\to T$ be an $A$-algebra homomorphism. Then the following assertions hold. 
\begin{enumerate}
\item
If $f$ is $I$-almost injective (i.e.\ $\ker (f)$ is annihilated by $I$), then so is the base extension $\id_{T'}\times_{T}f: T'\times_{T}R\to T'\times_{T}S$. 
\item
If $f$ is $I$-almost surjective (i.e.\ $\coker (f)$ is annihilated by $I$), then so is the base extension $\id_{T'}\times_{T}f: T'\times_{T}R\to T'\times_{T}S$. 
\end{enumerate}
\end{lemma}
\begin{proof}
We use the explicit description of fiber products: 
$T'\times_{T}R=\{(t',r) \in T'\times R~|~\psi(t')=\varphi_{R}(r)\}$ and $T'\times_{T}S=\{(t',s) \in T'\times S~|~\psi(t')=\varphi_{S}(s)\}$. 

(1): Pick an element $(t',r)\in T'\times_{T}R$ with $\id_{T'}\times_{T}f((t',r))=0$. Then $t'=\id_{T'}(t')=0$. 
Moreover, since $f(r)=0$, $xs=0$ for every $x\in I$ by assumption. 
Hence $x(t',r)=0$ for every $x\in I$, which yields the assertion. 

(2): Pick an element $(t',s)\in T'\times_{T}S$. Then by assumption, for every $x\in I$, there exists some $r_{x}\in R$ such that $f(r_{x})=xs$. 
Thus we obtain an element $(xt', r_{x})\in T'\times_{T}R$ whose image in $T'\times_{T}S$ is $x(t', s)$, as desired. 
\end{proof}

The following lemma claims that almost isomorphy is preserved under adic completion.

\begin{lemma}\label{almostcomp}
Let $(A, I)$ be a basic setup. Let $f: M\to N$ be an $I$-almost isomorphism between $A$-modules. Let $J\subset A$ be an ideal, and let $\widehat{M}$ and $\widehat{N}$ be the $J$-adic completions.  
Then the $A$-module map $\widehat{f}: \widehat{M}\to \widehat{N}$ induced by $f$ is also an $I$-almost isomorphism. 
\end{lemma}

\begin{proof}
The assertion is equivalent to the assertion that the map $(\widehat{f})^{a}: (\widehat{M})^{a}\to (\widehat{N})^{a}$ in $A^a-\mathbf{Mod}$ is an isomorphism. 
Since the functor $(~)^a$ commutes with limits, $(\widehat{f})^{a}$ is canonically isomorphic to $\varprojlim_{n>0}f^{a}_{n}: \varprojlim_{n>0}(M/J^{n}M)^{a}\to \varprojlim_{n>0}(N/J^{n}N)^{a}$ where 
$f_{n}: M/J^{n}M\to N/J^{n}N$ is the $A$-module map induced by $f$ for every $n> 0$. 
Thus it suffices to show that $f_{n}$ is an $I$-almost isomorphism for every $n> 0$. It can be easily seen that $f_{n}$ is $I$-almost surjective because $f$ is so. 
Let us verify that $f_{n}$ is $I$-almost injective. 
First, we have $\ker (f_{n})=f^{-1}(J^{n}N)/J^{n}M$. Moreover, for an arbitrary $\epsilon\in I$, 
$$
\epsilon f^{-1}(J^{n}N)\subset f^{-1}(J^{n}\im (f))=J^{n}M+\ker (f)
$$ 
because $f$ is $I$-almost surjective. 
Thus, since $f$ is $I$-almost injective, for an arbitrary $\epsilon'\in I$ we have 
$$
\epsilon'\epsilon f^{-1}(J^{n}N)\subset J^{n}M+\epsilon' \ker (f)=J^{n}M. 
$$
Therefore, $\epsilon'\epsilon\ker (f_{n})=0$. Since $I^{2}=I$, it implies that $\epsilon'' \ker (f_{n})=0$ for every $\epsilon''\in I$. 
Hence the assertion follows. 
\end{proof}

\subsection{Integrality and almost integrality}\label{InAI}
Here we list several closure operations of rings that will be used frequently.

\begin{definition}
Let $R \subset S$ be a ring extension. 

\begin{enumerate}
\item
An element $s \in S$ is \textit{integral} over $R$, if $\sum_{n=0}^\infty R \cdot s^n$ is a finitely generated $R$-submodule of $S$. The set of all elements denoted as $T$ of $S$ that are integral over $R$ forms a subring of $S$. If $R=T$, then $R$ is called \textit{integrally closed} in $S$. We denote by $R^+_S$ the integral closure of $R$ in $S$.

\item
An element $s \in S$ is \textit{almost integral} over $R$, if $\sum_{n=0}^\infty R \cdot s^n$ is contained in a finitely generated $R$-submodule of $S$. The set of all elements denoted as $T$ of $S$ that are almost integral over $R$ forms a subring of $S$, which is called the \textit{complete integral closure} of $R$ in $S$. We denote this ring by $R^*_S$. If $R=T$, then $R$ is called \textit{completely integrally closed} in $S$. 
\end{enumerate}
\end{definition}

This definition can be extended to any ring map $R \to S$ in a natural way, as follows. Let $R$ be a ring, let $S$ be an $R$-algebra and let $s\in S$ be an element. Then  we say that $s$ is \textit{integral} (resp.\ \textit{almost integral}) over $R$, if $s$ is integral (resp.\ almost integral) over the image of $R$ in $S$. 
We should remark that ``almost integrality'' does not mean ``integrality in almost ring theory'' in a strict sense, but there is an interesting connection between these two notions; see \cite[Lemma 5.3]{NS19}. 

From the definition, it is immediate to see that if $R$ is a Noetherian domain and $S$ is the field of fractions of $R$, then $R$ is integrally closed if and only if it is completely integrally closed. There are subtle points that we must be careful about on complete integral closure. The complete integral closure $T$ of $R$ is not necessarily completely integrally closed in $S$ and such an example was constructed by W. Heinzer \cite{He69}. Let $R \subset S \subset T$ be ring extensions. Let $b \in S$ be an element. Assume that $b$ is almost integral over $R$ when $b$ is regarded as an element of $T$. Then it does not necessarily mean that $b$ is almost integral over $R$ when $b$ is regarded as an element of $S$; see \cite{GH66} for such an example. 

We also recall the notion of \textit{absolute integral closure} due to Artin \cite{Ar71}.

\begin{definition}[Absolute integral closure]\label{AICdef}
Let $A$ be an integral domain. Then the \textit{absolute integral closure} of $A$ denoted by $A^+$, is defined to be the integral closure of $A$ in a fixed algebraic closure of $\Frac(A)$.
\end{definition}

\subsection{Semivaluation and adic spectra}
We need some basic language from Huber's \textit{continuous valuations} and \textit{adic spectra}; see \cite{Hu93} and \cite{Hu94}.
Continuous valuations are a special class of \textit{semivaluations} (see Definition \ref{semivaluation1} below) that satisfy a certain topological condition. 

\begin{definition}[Semivaluation]
\label{semivaluation1}
Let $A$ be a ring and let $|\cdot|:A \to \Gamma \cup \{0\}$ be a map for a totally ordered abelian group $\Gamma$ with group unit $1$ and we let $0< \gamma$ for arbitrary $\gamma \in \Gamma$. Then $|\cdot|$ is called a \textit{semivaluation}, if $|0|=0$, $|1|=1$, $|xy|=|x||y|$ and $|x+y| \le \max \{|x|,|y|\}$ for $x,y \in A$.
\end{definition}

\begin{definition}[Continuous valuation]
Let $A$ be a topological ring. Then a semivaluation $|\cdot|:A \to \Gamma \cup \{0\}$ is \textit{continuous} if $|\cdot|^{-1}(\Gamma_{<\gamma} \cup \{0\})$ is an open subset of $A$ for any $\gamma \in \Gamma$, where $\Gamma_{<\gamma}:=\{\alpha \in \Gamma~|~\alpha<\gamma\}$.
\end{definition}

The name ``semivaluation'' refers to the fact that $A$ need not be an integral domain. 
However, following the usage employed in \cite{Hu93}, we will stick to the word ``continuous valuation" rather than ``continuous semivaluation" for brevity. 

In this paper, we mainly consider continuous valuations on \textit{Tate rings}.

\begin{definition}
\label{TateRingDef}
Let $A$ be a topological ring.
\begin{enumerate}
\item
$A$ is called \textit{Tate}, if there is an open subring $A_0 \subset A$ together with an element $t \in A_0$ such that the topology on $A_0$ induced from $A$ is $t$-adic and $t$ becomes a unit in $A$. $A_0$ is called a \textit{ring of definition} and $t$ is called a \textit{pseudouniformizer}.

\item
Let $A_0$ be a ring and $t \in A_0$ is a regular element. Then the \textit{Tate ring associated to} $(A_0,(t))$\footnote{$(t)$ denotes the principal ideal of $A_{0}$ generated by $t$. Notice that the ring $A_{0}[\frac{1}{t}]$ and $t$-adic topology on $A_{0}$ are independent of the choice of a generator of the ideal $tA_{0}$ because $t\in A_{0}$ is regular. } is the ring 
$A:=A_{0}[\frac{1}{t}]$ equipped with the linear topology such that $\{t^{n}A_0\}_{n\geq 1}$ forms a fundamental system of open neighborhoods of $0\in A$ (it is a unique Tate ring containing $A_{0}$ such that $A_{0}$ is a ring of definition and $t$ is a pseudouniformizer; see \cite[Lemma 2.11]{NS19}). 
\end{enumerate}
\end{definition}

For a Tate ring $A$, we denote by $A^\circ \subset A$ the subset consisting of powerbounded elements of $A$ and by $A^{\circ\circ} \subset A$ the subset consisting of topologically nilpotent elements of $A$. It is easy to verify that $A^{\circ\circ} \subset A^\circ \subset A$, $A^\circ$ is a subring of $A$ and $A^{\circ\circ}$ is an ideal of $A^\circ$. 
The pair $(A,A^+)$ is called an \textit{affinoid Tate ring}, if $A^+ \subset A$ 
is an open and integrally closed subring contained in $A^{\circ}$.\footnote{This $A^{+}$ should not be confused with the same symbol representing the absolute integral closure in Definition \ref{AICdef}.} 
Let $\Spa(A,A^+)$ denote the set of continuous valuations $|\cdot|$ on an affinoid Tate ring $(A,A^+)$ satisfying an additional condition $|A^+|\le 1$ modulo a natural equivalence relation.

Let us pick an element $|\cdot| \in \Spa(A,A^+)$, and set $\fp:=\{x\in A^{+}~|~|x|=0\}$. 
Then by Lemma \ref{semivaluation2} below, $\fp$ is a prime ideal, and 
$|\cdot|$ defines a valuation ring $V_{|\cdot|} \subset \Frac(A^+/\fp)$. This valuation ring is \textit{microbial} attached to $|\cdot|$ in view of \cite[Proposition 7.3.7]{Bh17}. For microbial valuation rings, we refer the reader to \cite{Hu96}.

\begin{lemma}
\label{semivaluation2}
Let $|\cdot|:A \to \Gamma \cup \{0\}$ be a semivaluation. Then $\fp:=\{x\in A~|~|x|=0\}$ is a prime ideal of $A$. 
Moreover, $|\cdot|$ uniquely extends to a valuation $|\cdot|_{\fp}:\Frac(A/\fp) \to \Gamma \cup \{0\}$, and $V_{|\cdot|}:=\{x \in \Frac(A/\mathfrak{p})~|~|x|_{\fp}\le 1\}$ is a valuation ring with its field of fractions $\Frac(A/\fp)$.
\end{lemma}

\begin{proof}
This is an easy exercise, using the properties stated in Definition \ref{semivaluation1}.
\end{proof}

The prime ideal $\fp$ in Lemma \ref{semivaluation2} is called the \textit{support} of the semivaluation $|\cdot|$.

\subsection{Perfectoid algebras}
Let us recall the notion of perfectoid algebras over a perfectoid field as defined in \cite{Sch12}. 
These are a special class of Banach algebras; see \cite[\S 2.4]{NS19} for the definition of Banach rings in this context and how they are related to Tate rings.

\begin{definition}[Perfectoid $K$-algebra]
Fix a perfectoid field $K$ and let $\mathcal{A}$ be a Banach $K$-algebra. Then we say that $\mathcal{A}$ is a \textit{perfectoid K-algebra}, if the following conditions hold:
\begin{enumerate}
\item
The set of powerbounded elements $\mathcal{A}^\circ \subset \mathcal{A}$ is open and bounded.

\item
The Frobenius endomorphism on $\mathcal{A}^\circ/(p)$ is surjective.
\end{enumerate}
\end{definition}

We will recall the almost variant of perfectoid algebras; see \cite{An1}.

\begin{definition}[Almost perfectoid $K$-algebra]
\label{Defalmostperf}
Fix a perfectoid field $K$ and let $\mathcal{A}$ be a Banach $K$-algebra with a basic setup $(\mathcal{A}^\circ,I)$. Then we say that $\mathcal{A}$ is \textit{I-almost perfectoid}, if the following conditions hold:
\begin{enumerate}
\item
The set of powerbounded elements $\mathcal{A}^\circ \subset \mathcal{A}$ is open and bounded.

\item
The Frobenius endomorphism $Frob:\mathcal{A}^\circ/(p) \to \mathcal{A}^\circ/(p)$ is $I$-almost surjective.
\end{enumerate}
\end{definition}

\begin{example}
Let $\mathcal{A}$ be a perfectoid $K$-algebra with a nonzero nonunit element $t \in K^\circ$ admitting a compatible system of $p$-power roots $\{t^{\frac{1}{p^n}}\}_{n>0}$. Fix any regular element $g \in \mathcal{A}^\circ$ that admits a compatible system $\{g^{\frac{1}{p^n}}\}_{n>0}$. Let $I:=\bigcup_{n>0} (tg)^{\frac{1}{p^n}}$. Then the pair $(\mathcal{A}^\circ,I)$ gives a basic setup, which is a prototypical example that is encountered in this article.
\end{example}

\section{Preliminary lemmas}
\label{SecPreLem}

\subsection{Some properties of complete integral closure}\label{Cicuc}

Here we investigate several properties of complete integral closure.  
First, we study how it behaves under separated completion. 
Thus we start with recalling the following lemma, which is a key for the main results of \cite{BL95}; see also \cite[Tag 0BNR]{Stacks} for a proof and related results.

\begin{lemma}[Beauville-Laszlo]
\label{Beauville-Laszlo}
Let $A$ be a ring with a regular element $t\in A$ and let $\widehat{A}$ be the $t$-adic completion. Then $t$ is a regular element in $\widehat{A}$ and one has the commutative diagram:
$$
\begin{CD}
A @>>> \widehat{A} \\
@VVV @VVV \\
A[\frac{1}{t}] @>>> \widehat{A}[\frac{1}{t}]
\end{CD}
$$
that is cartesian. In other words, we have $A \cong A[\frac{1}{t}] \times_{\widehat{A}[\frac{1}{t}]} \widehat{A}$.
\end{lemma}

\begin{corollary}\label{VariantBL}
Let $A$ be a Tate ring, and let $\mathcal{A}$ be the separated completion of $A$. 
Then the natural map $\psi: A\to \mathcal{A}$ restricts to $\psi^\circ: A^\circ\to \mathcal{A}^{\circ}$, and the commutative diagram: 
\[\xymatrix{
A^{\circ}\ar[r]^{\psi^\circ}\ar[d]&\mathcal{A}^{\circ}\ar[d]\\
A\ar[r]^\psi&\mathcal{A}
}\]
is cartesian. 
\end{corollary}
\begin{proof}
Let $(A_{0}, (t))$ be a pair of definition of $A$, and $\widehat{A_{0}}$ the $t$-adic completion of $A_{0}$. Then $\mathcal{A}$ is the Tate ring associated to $(\widehat{A_0}, (t))$. 
Hence the first assertion is clear. To check the second assertion, pick $a\in A$ such that $\psi(a)\in \mathcal{A}^{\circ}$. Then there exists some $l>0$ such that $\psi(t^la^{n})=t^{l}\psi(a)^{n}\in \widehat{A_{0}}$ for every $n\geq 0$. 
Hence $t^{l}a^{n}\in A_{0}$ for every $n\geq 0$ by Lemma \ref{Beauville-Laszlo}. Therefore, $a\in A^{\circ}$ as desired. 
\end{proof}

The following lemma is quite useful and often used in basic theory of perfectoid spaces. We take a copy from Bhatt's lecture notes \cite{Bh17}.

\begin{lemma}
\label{p-adicnormal}
Let $A$ be a ring with a regular element $t \in A$ and let $\widehat{A}$ be the $t$-adic completion of $A$. Fix a prime number $p>0$. Then the following assertions hold.

\begin{enumerate}
\item
Suppose that $A$ is integrally closed in $A[\frac{1}{t}]$. Then $\widehat{A}$ is integrally closed in $\widehat{A}[\frac{1}{t}]$. If moreover $A$ admits a compatible system of $p$-power roots $\{t^{\frac{1}{p^n}}\}_{n>0}$, then $t^{-\frac{1}{p^\infty}} A$ is integrally closed in $A[\frac{1}{t}]$.

\item
Suppose that $A$ is completely integrally closed in $A[\frac{1}{t}]$. Then $\widehat{A}$ is completely integrally closed in $\widehat{A}[\frac{1}{t}]$. If moreover $A$ admits a compatible system of $p$-power roots $\{t^{\frac{1}{p^n}}\}_{n>0}$, then $t^{-\frac{1}{p^\infty}} A$ is completely integrally closed in $A[\frac{1}{t}]$.
\end{enumerate}
\end{lemma}

\begin{proof}
We refer the reader to \cite[Lemma 5.1.1, Lemma 5.1.2 and Lemma 5.1.3]{Bh17}, and \cite[Lemma 2.7]{NS19}. Here we point out that Lemma \ref{Beauville-Laszlo} plays a role in the proofs.
\end{proof}

The following lemma is easy to prove, but plays an important role in our arguments.

\begin{lemma}
\label{completeintegralmap}
Let $A$ be a ring, and let $t \in A$ be a regular element. Fix a prime number $p>0$. Suppose that $A$ admits a compatible system of $p$-power roots $\{t^{\frac{1}{p^n}}\}_{n>0}$. Then for any $A[\frac{1}{t}]$-algebra $B$, we have $A^{*}_{B} = (t^{-\frac{1}{p^\infty}}A)^{*}_{B}$. 
\end{lemma}

\begin{proof}
It suffices to show that $(t^{-\frac{1}{p^\infty}}A)^{*}_{B}\subset A^{*}_{B}$. 
Let $A'$ be the image of $A$ in $B$. 
Then the image of $t^{-\frac{1}{p^{\infty}}}A$ in $B$ is contained in $t^{-\frac{1}{p^\infty}}A'$. Thus we have $(t^{-\frac{1}{p^\infty}}A)^{*}_{B}\subset (t^{-\frac{1}{p^\infty}}A')^{*}_{B}$ and $A^*_{B}=(A')^{*}_{B}$. 
Hence we may assume that $A=A'$, that is,  $A[\frac{1}{t}]$ is a subring of $B$.  
Pick $x\in(t^{-\frac{1}{p^\infty}}A)^{*}_{B}$. 
Then, $\sum^\infty_{n=0}t^{-\frac{1}{p^{\infty}}}A\cdot x^{n}$ is contained in a finitely generated $t^{-\frac{1}{p^{\infty}}}A$-submodule of $B$, and thus so is $\sum_{n=0}^\infty A\cdot x^{n}$. 
Hence $t\sum^\infty_{n=0}A \cdot x^{n}$ is contained in a finitely generated $A$-submodule $M$ of $B$. 
Let $b_{1},\ldots, b_{r}$ be a system of generators of $M$ over $A$. 
Then, since $t$ is invertible in $B$, we obtain a finitely generated $A$-submodule $\sum^{r}_{i=1}A\cdot\frac{b_{i}}{t}$ of $B$, which contains $\sum^\infty_{n=0}A\cdot x^{n}$, as desired. 
\end{proof}

The following corollary is an immediate consequence of Lemma \ref{completeintegralmap}.

\begin{corollary}
\label{lem07031}
Let $A$ be a ring with a regular element $t \in A$ such that $A$ is completely integrally closed in $A[\frac{1}{t}]$. Fix a prime number $p>0$. Suppose that $A$ admits a compatible system of $p$-power roots $\{t^{\frac{1}{p^n}}\}_{n>0}$. Then we have $t^{-\frac{1}{p^\infty}} A=A$ (in particular, $t^{-\frac{1}{p^\infty}} A$ is completely integrally closed in $A[\frac{1}{t}]$). 
\end{corollary}

%\begin{proof}
%Since clearly $A\subset t^{-\frac{1}{p^\infty}} A$, it suffices to show the reverse inclusion. Pick an element $b\in t^{-\frac{1}{p^\infty}}A$. Then for every $n>0$, there exists $a_{n}\in A$ such that $t^{\frac{1}{p^{n}}}b=a_n$ and therefore, $b^n=t^{-1}(t^{\frac{1}{p^{n}}})^{p^{n}-n}a_{n}^{n}\in A[\frac{1}{t}]$. Here notice that $(t^{\frac{1}{p^{n}}})^{p^{n}-n}a_{n}^{n} \in A$. Thus, one obtains $tb^{n}\in A$ for every $n>0$ and we have $b\in A$, because $A$ is completely integrally closed in $A[\frac{1}{t}]$. 
%\end{proof}

Now let us discuss complete integral closedness of inverse limits.

\begin{lemma}
\label{lem07032}
Let $A$ be a ring with an element $t\in A$, let $\Lambda$ be a directed poset, and let $\{A_{\lambda}\}_{\lambda\in \Lambda}$ an inverse system of $A$-algebras. Suppose that each $A_{\lambda}$ is $t$-torsion free and completely integrally closed in $A_{\lambda}[\frac{1}{t}]$. Then $\varprojlim_{\lambda} A_{\lambda}$ is a $t$-torsion free $A$-algebra and completely integrally closed in $(\varprojlim_{\lambda} A_{\lambda})[\frac{1}{t}]$. 
\end{lemma}

\begin{proof}
Clearly, $\varprojlim_{\lambda} A_{\lambda}$ is a $t$-torsion free $A$-algebra. Pick an element $b\in (\varprojlim_{\lambda} A_{\lambda})[\frac{1}{t}]$ which is almost integral over $\varprojlim_{\lambda} A_{\lambda}$. Then there exists some $m>0$ such that $t^{m}b^{n}\in\varprojlim_{\lambda} A_{\lambda}$ for every $n>0$. Take $d>0$ and $a=(a_{\lambda})\in \varprojlim_{\lambda} A_{\lambda}$ for which $t^{d}b=a$. Then for every $n>0$, it follows that $t^{dn+m}b^{n}=t^{m}a^{n}$, which implies $t^{m}a^{n}\in t^{dn}(\varprojlim_{\lambda} A_{\lambda})$. Thus for each $\lambda\in \Lambda$, the element $\frac{a_{\lambda}}{t^{d}}\in A_{\lambda}[\frac{1}{t}]$ satisfies $t^{m}(\frac{a_{\lambda}}{t^{d}})^{n}\in A_{\lambda}$ for every $n$. Since $A_{\lambda}$ is completely integrally closed in $A_{\lambda}[\frac{1}{t}]$, one finds that $a_{\lambda}\in t^{d}A_{\lambda}\ (\forall\lambda\in\Lambda)$ and thus $a\in t^{d}(\varprojlim_{\lambda} A_{\lambda})$. Hence $b=\frac{a}{t^d} \in \varprojlim_{\lambda}A_{\lambda}$, as desired. 
\end{proof}

In the situation of Lemma \ref{almostelement}, complete integral closedness is preserved under $(~)_*$.

\begin{lemma}
\label{Lem0724}
Let $A \hookrightarrow B$ be a ring extension such that $A$ is completely integrally closed in $B$. Suppose that $A$ has an element $t$ such that $B$ is $t$-torsion free and $A$ admits a compatible system of $p$-power roots $\{t^{\frac{1}{p^{n}}}\}_{n>0}$. Then $t^{-\frac{1}{p^{\infty}}}A$ is completely integrally closed in $t^{-\frac{1}{p^{\infty}}}B$. 
\end{lemma}

\begin{proof}
Pick an element $c\in t^{-\frac{1}{p^{\infty}}}B$ which is almost integral over $t^{-\frac{1}{p^{\infty}}}A$. We would like to show that $t^{\frac{1}{p^{k}}}c\ \in A$ for every $k>0$. Since $A$ is completely integrally closed in $B$, it suffices to check that each $t^{\frac{1}{p^{k}}}c\in B$ is almost integral over $A$. Now by assumption, 
$\sum^{\infty}_{n=0}t^{-\frac{1}{p^{\infty}}}A\cdot c^{n}$ is contained in a finitely generated $t^{-\frac{1}{p^{\infty}}}A$-submodule of $t^{-\frac{1}{p^{\infty}}}B$. Hence $t^{\frac{1}{p^{k}}}(\sum^{\infty}_{n=0}t^{-\frac{1}{p^{\infty}}}A\cdot c^{n})$ is contained in a finitely generated $A$-submodule of $B$ for every $k>0$. Meanwhile, it follows that   
$$
\sum^\infty_{n=0} A\cdot (t^{\frac{1}{p^{k}}}c)^n
\subset
t^{\frac{1}{p^{k}}}\biggl(\sum^{\infty}_{n=0}A\cdot c^{n}\biggr)
\subset
t^{\frac{1}{p^{k}}}\biggl(\sum^{\infty}_{n=0}t^{-\frac{1}{p^{\infty}}}A\cdot c^{n}\biggr).
$$
Therefore, $t^{\frac{1}{p^{k}}}c\in B$ is almost integral over $A$, as desired. 
\end{proof}

Next we consider several types of ring extensions. 
The following lemmas are mainly used in \S\ref{SecAppWPA}. 

\begin{lemma}
\label{completenormal}
The following assertions hold.
\begin{enumerate}
\item
Let $R$ be a Noetherian integrally closed domain with its absolute integral closure $R^+$ and assume that $A$ is a ring such that $R \subset A \subset R^+$. Then $A$ is integrally closed in $\Frac(A)$ if and only if $A$ is completely integrally closed in $\Frac(A)$.

\item
Let $R \subset S \subset T$ be ring extensions. Assume that $R$ is completely integrally closed in $T$. Then $R$ is also completely integrally closed in $S$.
\end{enumerate}
\end{lemma}

\begin{proof}
$(1)$: The proof is found in the proof of \cite[Theorem 5.9]{Sh15}, whose statement is given only for Noetherian normal rings of characteristic $p>0$. However, the argument there remains valid for Noetherian normal rings of arbitrary characteristic (see also Proposition \ref{propIC=CIC}).

$(2)$: For $s \in S$, assume that $\sum_{n=0}^\infty R \cdot s^n$ is contained in a finitely generated $R$-submodule of $S$. Then this property remains true when regarded as an $R$-submodule of $T$. So we have $s \in R$ by our assumption.
\end{proof}

\begin{lemma}
\label{normallocaldomain}
Let $A$ be a normal domain with field of fractions $\Frac(A)$ and assume that $\Frac(A) \hookrightarrow B$ is an integral extension such that $B$ is reduced. Denote by $C:=A^+_B$ the integral closure of $A$ in $B$. Then $C_\fp$ is a normal domain for any prime ideal $\fp$ of $C$.
\end{lemma}

\begin{proof}
Notice that $B$ can be written as the filtered colimit of finite integral subextensions $\Frac(A) \to B' \to B$. Without loss of generality, we may assume and do that $\Frac(A) \to B$ is a finite integral extension. Since $\Frac(A)$ is a field, $B$ is a reduced Artinian ring, so that we can write $B=\Pi_{i=1}^m L_i$ with $L_i$ being a field. Since $A \to C$ is torsion free and integral, we see that $\Frac(A) \otimes_A C$ is the total ring of fractions of $C$, which is just $B$. In other words, $C$ has finitely many minimal prime ideals, because so does $B$. Then by \cite[Tag 030C]{Stacks},
$C$ is a finite product of normal domains, which shows that $C_\fp$ is a normal domain for any prime ideal $\fp \subset C$.
\end{proof}

\subsection{Almost regular sequences and the ring of bounded functions}
\label{subsecARS}
The notion of \emph{almost regular sequences} often shows up in the main content of this article.

\begin{definition}[Almost regular sequence]
Let $(A,I)$ be a basic setup with a sequence of elements $x_1,\ldots,x_n$ in $A$ and let $M$ be an $A$-module. Then we say that $x_1,\ldots,x_n$ is an \textit{$I$-almost regular sequence in M} if
$$
b \cdot \big((x_1.\ldots,x_i):_M x_{i+1}\big) \subset (x_1,\ldots,x_i)M
$$
for any $b \in I$ and $i=0,\ldots,n-1$. In particular, we say that an element $x \in M$ is \textit{I-almost regular} if the kernel of the multiplication map $M \xrightarrow{x} M$ is annihilated by $I$.
\end{definition}

First of all, we record the following fundamental lemma. 

\begin{lemma}
\label{adiccompletion}
Let $(A,I)$ be a basic setup and assume that $a,b$ is an $I$-almost regular sequence on $A$. Let $\widehat{A}$ denote the $a$-adic completion of $A$. Then $a$ and $b$ are $I$-almost regular elements of $\widehat{A}$.
\end{lemma}

\begin{proof}
First we prove that $a$ is $I$-almost regular on $\widehat{A}$. For any $k>0$, the multiplication map $A \xrightarrow{a} A$ induces an $I$-almost injective map $A/(a^k) \xrightarrow{a} A/(a^{k+1})$. This forms a commutative diagram:
$$
\begin{CD}
A/(a^{k+1}) @>a>>A/(a^{k+2}) \\
@VVV @VVV \\
A/(a^{k}) @>a>>A/(a^{k+1}). \\
\end{CD}
$$
Taking inverse limits along the vertical directions respectively, $\widehat{A} \xrightarrow{a} \widehat{A}$ is $I$-almost injective. 

Next we prove that $b$ is $I$-almost regular. Let $t \in \widehat{A}$ be such that $bt=0$. Then one obviously has $bt \in a^n \widehat{A}$ for all $n>0$. Since $b$ is $I$-almost regular on $A/(a^n) \cong \widehat{A}/(a^n)$, it follows that $\epsilon t \in \bigcap_{n>0} a^n \widehat{A}=0$ for any $\epsilon \in I$.
\end{proof}

\begin{example}
We give a counterexample to Lemma \ref{adiccompletion} without almost regularity condition. Let us consider the subring:
$$
R:=\mathbb{Z}\big[\frac{x}{p},\frac{x}{p^2},\ldots\big] \subset \mathbb{Q}[x].
$$
Then it is clear that $R$ is a domain. However, after taking the $p$-adic completion $\widehat{R}$, since $x \in p^nR$, $x$ becomes zero in $\widehat{R}$. Therefore, $p$ is a regular element in $\widehat{R}$, while $x$ is not so. 
\\
\end{example}

Let $A$ be a ring with elements $f,g \in A$. Then we can consider the \textit{ring of bounded functions defined by a sequence $f^n,g$}, denoted by $A[\frac{f^n}{g}]$ as a subring of $A[\frac{1}{g}]$. In other words, we define
$$
A[\frac{f^n}{g}]:=\big(A[T]/(gT-f^n)\big)\big/\fa,
$$
where $\fa:=\bigcup_{m>0} (0: g^m)$ as an ideal of $A[T]/(gT-f^n)$. In the technical part of this paper, the following problem plays a central role.

\begin{Problem}[Algebraic formulation of Riemann's extension problem]
\label{RiemannExtProblem}
Study the ring-theoretic structure of the intersection
$$
\bigcap_{n>0} \biggl(A[\frac{f^n}{g}]\biggr)^{*}_{A[\frac{1}{fg}]}
$$
taken inside $A[\frac{1}{fg}]$. 
\end{Problem}

In his remarkable paper \cite{Sch15}, Scholze studied the perfectoid version of Problem \ref{RiemannExtProblem}, with an application to the construction of Galois representations using torsion classes in the cohomology of certain symmetric spaces. We need the following fact on the description of the ring of bounded functions as a Rees algebra. In the case that the relevant sequence is regular, the proof is found in \cite[Tag 0BIQ]{Stacks}.

\begin{lemma}
\label{blowupring1}
Let $(A,I)$ be a basic setup and let $f, g \in A$ be elements such that $f,g$ forms an $I$-almost regular sequence. Then the ring map
$$
A[T]/(gT-f) \to A[\frac{f}{g}];~T \mapsto \dfrac{f}{g}
$$
is an $I$-almost isomorphism.
\end{lemma}

\begin{proof}
It suffices to prove that $A[T]/(gT-f)$ is $I$-almost $g$-torsion free. So suppose that
\begin{equation}
\label{ReesRing12}
g^ka \in (gT-f)
\end{equation}
for some $a \in A[T]$ and some $k \in \mathbb{N}$. Let $\epsilon \in I$ be any element. Then we want to show that $\epsilon a \in (gT-f)$. Now we can find $b \in A[T]$ such that $g^ka=(gT-f)b$ and write this equation as
\begin{equation}
\label{Reesalg1}
bf=g(bT-g^{k-1}a).
\end{equation}
By the almost regularity of $f,g$ on $A[T]$, for any $\epsilon_{n_1} \in I$, it follows that $\epsilon_{n_1}(bT-g^{k-1}a)=fc$ for some $c \in A[T]$. Substituting this back into $(\ref{Reesalg1})$, we have $\epsilon_{n_1} bf=gfc$. As $f$ is $I$-almost regular, we have $\epsilon_{n_1}\epsilon_{n_2} b=\epsilon_{n_2}gc$ for any $\epsilon_{n_2} \in I$. Since $\epsilon_{n_1}\epsilon_{n_2}g^{k-1}a=\epsilon_{n_1}\epsilon_{n_2} bT-\epsilon_{n_2} fc$, we obtain $\epsilon_{n_1}\epsilon_{n_2}g^{k-1}a=\epsilon_{n_2}c(gT-f)$. Since $\epsilon_{n_1},\epsilon_{n_2} \in I$ are arbitrary and $I=I^2$, we find that 
$\epsilon g^{k-1}a \in (gT-f)$ for any $\epsilon \in I$. Arguing inductively on $k$ in view of $(\ref{ReesRing12})$, it follows that $\epsilon a \in (gT-f)$ for any $\epsilon \in I$, as desired.
\end{proof}

We prove a result which compares the ring of bounded functions under completion. This result will play a crucial role later.

\begin{proposition}
\label{RegularRatLoc}
Let the notation and the hypotheses be as in Lemma \ref{blowupring1}. Then for any $n \ge 0$, there is an $I$-almost isomorphism of rings:
$$
\widehat{A[\dfrac{f^n}{g}]} \xrightarrow{\approx} \widehat{\widehat{A}[\frac{f^n}{g}]},
$$
where $\widehat{(~)}$ is the $f$-adic completion.
\end{proposition}

\begin{proof}
In view of Lemma \ref{blowupring1}, we need to show that the natural map
\begin{equation}
\label{blowupring2}
\widehat{A[T]/(gT-f)} \to \widehat{\widehat{A}[T]/(gT-f)}
\end{equation}
is bijective. It suffices to show that for any $n>0$, $(\ref{blowupring2})$ is bijective after dividing out by the ideal generated by $f^n$ on both sides. So we get
$$
A[T]/(gT-f,f^n) \to \widehat{A}[T]/(gT-f,f^n),
$$
which is isomorphic to
$$
A/(f^n) \otimes_A A[T]/(gT-f) \to \widehat{A}/(f^n) \otimes_A A[T]/(gT-f). 
$$
Since $A/(f^n) \cong \widehat{A}/(f^n)$, we are done.
\end{proof}

\section{Finite \'etale extensions of Tate rings}\label{FEtTate}
Here we study some behavior of finite \'etale extensions of Tate rings under rational localization.

\subsection{Finite \'etale extensions of preuniform Tate rings}
Let us begin with a review on some results in \cite{NS19} used later. 
First we recall a canonical structure as a Tate ring that is induced on a module-finite algebra over a Tate ring. See \cite[Lemma 2.19 and Lemma 2.20]{NS19} for the next lemma.

\begin{lemma}
\label{tatefinex}
Let $A$ be a Tate ring and let $B$ be a module-finite $A$-algebra. 
Take a ring of definition $A_0\subset A$, a pseudouniformizer $t\in A_0$ and a finite generating set $S$ of $B$ over $A$. Let $M_0\subset B$ be the $A_0$-submodule generated by $S$. Equip $B$ with the linear topology defined by $\{t^nM_0\}_{n>0}$. 
Then the following assertions hold. 
\begin{enumerate}
\item
The topology on $B$ is independent of the choices of $A_0$, $t$ and $S$. 
\item
$B$ is a Tate ring with the following property:  
\begin{itemize}
\item{for every ring of definition $A_0$ and every pseudouniformizer $t\in A_0$ of $A$, there exists a ring of definition $B_0$ of $B$  that is an integral $A_0$-subalgebra of $B$ with finitely many generators and $t\in B_0$ is a pseudouniformizer of $B$. }
\end{itemize}
\end{enumerate}
\end{lemma}

Next we recall the definition of (pre)uniformity of Tate rings. 
\begin{definition}
Let $A$ be a Tate ring. We say that $A$ is \emph{preuniform} if $A^{\circ}$ is a ring of definition of $A$. 
Moreover, we say that $A$ is \emph{uniform} if $A$ is preuniform and complete and separated. 
\end{definition}

Permanence of (pre)uniformity is one of the most remarkable features of finite \'etale extensions of Tate rings. See \cite[Corollary 4.8 (1), (4), and (5)]{NS19} for the next proposition.

\begin{proposition}
\label{cor1547205}
Let $A$ be a preuniform Tate ring. Let $f: A\to B$ be a finite \'etale ring map. Equip $B$ with the canonical structure as a Tate ring (cf.\ Lemma \ref{tatefinex}). 
Then the following assertions hold. 
\begin{enumerate}
\item
$B$ is also preuniform. In particular, for any ring of definition $A_0$ of $A$, $(A_0)^+_B$ and $(A_0)^*_B$ are rings of definition of $B$.  
\item
If $f$ is injective, then $f^{-1}(B^{\circ})=A^{\circ}$. 
\item
If $A$ is uniform, then so is $B$. 
\end{enumerate}
\end{proposition}

When one considers separated completion of finite \'etale extensions of Tate rings, two types of extension of complete Tate rings appear. 
The following statement assures that they are isomorphic under the preuniformity assumption. See \cite[Corollary 4.10]{NS19} for the next proposition.

\begin{proposition}
\label{IsomCompPowerBd}
Let $A$ be a preuniform Tate ring with a pseudouniformizer $t$. 
Let $\mathcal{A}$ be the separated completion of $A$. Let $B$ be a finite \'etale $A$-algebra and denote by $\mathcal{B}$ the finite \'etale $\mathcal{A}$-algebra $B\otimes_A\mathcal{A}$. Equip $B$ and $\mathcal{B}$ with the canonical structure as a Tate ring (cf.\ Lemma \ref{tatefinex}) respectively. 
Let $\widehat{B^\circ}$ be the $t$-adic completion of $B^\circ$. 
Then the natural $A$-algebra homomorphism $\varphi: \mathcal{B}\rightarrow \widehat{B^\circ}[\frac{1}{t}]$ is an isomorphism which induces an isomorphism 
$\mathcal{B}^\circ\xrightarrow{\cong}\widehat{B^\circ}$. 
\end{proposition}

In addition, we record a complement to \S\ref{Cicuc}. 
Preuniform Tate rings fit into Galois theory of rings. 

\begin{lemma}
\label{GalPwrBdd}
Let $A$ be a Tate ring and let $A \hookrightarrow B$ be a finite Galois extension with Galois group $G$. Equip $B$ with the canonical structure as a Tate ring as in Lemma \ref{tatefinex}. Then the action of $G$ preserves $B^\circ$. Moreover, if further $A$ is preuniform, then $(B^\circ)^G=A^\circ$. 
\end{lemma}

\begin{proof}
Let $A_0$ be a ring of definition of $A$ and let  $t\in A_0$ be a pseudouniformizer of $A$. As in the proof of \cite[Lemma 2.20]{NS19}, we can take a ring of definition $B_0$ of $B$ that is finitely generated as an $A_0$-module and satisfies $B=B_0[\frac{1}{t}]$. Pick $b\in B^\circ$ and $\sigma\in G$ arbitrarily. Then there is some $l>0$ such that $t^lb^n\in B_0$ and therefore, $t^l\sigma(b)^n\in \sigma(B_0)$ for every $n>0$. Meanwhile, since $\sigma(B_0)$ is also finitely generated as an $A_0$-module, we have $t^{l'}\sigma(B_0)\subset B_0$ for some $l'>0$. Hence $\sigma(b)$ is also almost integral over $B_0$. Thus, the action of $G$ preserves $B^\circ$. If further $A$ is preuniform, then we have 
$$
(B^\circ)^G=B^G\cap B^\circ=A\cap B^\circ=A^\circ
$$
by Proposition \ref{cor1547205} (2), as wanted. 
\end{proof}

\subsection{Rational functions associated to regular sequences}
We specialize the above results to study the rings of rational functions associated to regular sequences. 
\\

$\bf{Notation}$: Let $A$ be Tate ring, $(A_0, (t))$ a pair of definition of $A$, and  $f,g \in A_{0}$ regular elements. 
Then we define a Tate ring $A(\frac{f}{g})$ as the Tate ring associated to $(A_{0}[\frac{f}{g}], (t))$ (see Definition \ref{TateRingDef} (2)), and also define $A\{\frac{f}{g}\}$ as the separated completion of $A(\frac{f}{g})$.
These Tate rings are independent of the choice of a pair of definition $(A_{0}, (t))$. 
\\

One can clarify the relationship between $A\{\frac{f}{g}\}$ and $\widehat{A}\{\frac{f}{g}\}$.

\begin{lemma}
\label{CompRatTate}
Let $A$ be Tate ring, let $(A_0, (t))$ a pair of definition of $A$, and let $f,g \in A_{0}$ be regular elements. 
Let $\mathcal{A}$ be the separated completion of ${A}$. 
Suppose that $(f,g)$ forms a regular sequence in $A_{0}$. 
Then $\mathcal{A}\{\frac{f}{g}\}$ is the separated completions of $A(\frac{f}{g})$. 
\end{lemma}

\begin{proof}
For any $A_{0}$-algebra $B$, we let $\widehat{B}$ denote the $t$-adic completion of $B$. 
By definition, $A(\frac{f}{g})$ has a pair of definition $(A_{0}[\frac{f}{g}], (t))$. Moreover, the completion $\widehat{A_{0}[\frac{f}{g}]}$ is canonically isomorphic to $\widehat{\widehat{A_{0}}[\frac{f}{g}]}$ 
by Proposition \ref{RegularRatLoc} because $(f,g)$ forms a regular sequence. 
Since $(\widehat{\widehat{A_{0}}[\frac{f}{g}]}, (t))$ is a pair of definition of $\mathcal{A}\{\frac{f}{g}\}$, the assertion follows. 
\end{proof}

Let us inspect topological features of finite \'etale algebras over the rings of rational functions.

\begin{proposition}
\label{RiemannFinEt}
Let $A$ be a Tate ring, and let $(A_{0}, (t))$ be a pair of definition of $A$. Let  $g\in A_{0}$ be a regular element. 
Suppose that $A(\frac{t}{g})$ is preuniform. 
Let $B'$ be a finite \'etale $A[\frac{1}{g}]$-algebra. 
Let $B_{0}$ be the integral closure of $A_{0}$ in $B'$, and let $B$ be the Tate ring associated to $(B_{0},(t))$. 
Let $B'_{t/g}$ be the finite \'etale $A(\frac{t}{g})$-algebra $B'$ equipped with the canonical structure as a Tate ring (cf.\ Lemma \ref{tatefinex}). 
Then the following assertions hold. 
\begin{enumerate}
\item
$B'_{t/g}$ is preuniform. 
\item 
$B'_{t/g}=B(\frac{t}{g})$ as topological rings. 
In particular, $B(\frac{t}{g})$ is preuniform.

\end{enumerate}
\end{proposition}
\begin{proof}
By Proposition \ref{cor1547205} (1), $B'_{t/g}$ is preuniform, and $\big((A_{0}[\frac{t}{g}])^+_{B'}, (t)\big)$ is a pair of definition of $B'_{t/g}$. 
Moreover, $\big(B_0[\frac{t}{g}]\big)^+_{B'}=\big(A_{0}[\frac{t}{g}]\big)^+_{B'}$ because $B_{0}=(A_{0})^{+}_{B'}$. 
Thus we find that $B_{0}[\frac{t}{g}]$ is bounded in $B'_{t/g}$. 
Hence it suffices to show that $B_{0}[\frac{t}{g}]\subset B'_{t/g}$ is open. 
Note that $B'_{t/g}=(A_{0}[\frac{1}{tg}])^{+}_{B'}=(A_{0})^{+}_{B'}[\frac{1}{tg}]=B_{0}[\frac{1}{tg}]$ as rings. 
By Lemma \ref{tatefinex} (2), one can take a ring of definition $B'_{0}$ of $B'_{t/g}$ which is module finite over $A_0[\frac{t}{g}]$. 
Then, since $A_0[\frac{t}{g}]\subset B_{0}[\frac{t}{g}]$ and $(B_{0}[\frac{t}{g}])[\frac{1}{t}]=B'_{t/g}$ as rings, there exists some $l>0$ such that $t^{l}B'_{0}\subset B_{0}[\frac{t}{g}]$. 
Hence $B_{0}[\frac{t}{g}]$ is open in $B'_{t/g}$, as desired. 
\end{proof}

\begin{proposition}\label{rationaluniformity}
Keep the notations and assumptions as in Proposition \ref{RiemannFinEt}. 
Suppose further that $(t,g)$ forms a regular sequence in $A_{0}$ and $B_{0}$. 
Let $\mathcal{A}$ and $\mathcal{B}$ be the separated completions of $A$ and $B$, respectively. 
Then the following assertions hold. 
\begin{enumerate}
\item
$\mathcal{A}\{\frac{t}{g}\}$ and $\mathcal{B}\{\frac{t}{g}\}$ are the separated completions of $A(\frac{t}{g})$ and $B(\frac{t}{g})$, respectively. 
\item
$\mathcal{A}\{\frac{t}{g}\}$ and $\mathcal{B}\{\frac{t}{g}\}$ are uniform. 
\item
Set $\mathcal{B}'_{t/g}:=B'\otimes_{A[\frac{1}{g}]}\mathcal{A}\{\frac{t}{g}\}$, and equip $\mathcal{B}'_{t/g}$ with the canonical structure as a Tate ring that is module-finite over $\mathcal{A}
\{\frac{t}{g}\}$ (cf.\ Lemma \ref{tatefinex}). 
Then we have isomorphisms of topological rings 
\begin{equation}\label{IsomsTateBRat}
\mathcal{B}'_{t/g}\cong \widehat{B'_{t/g}}\cong \widehat{B(\frac{t}{g})}\cong\mathcal{B}\{\frac{t}{g}\}
\end{equation}
(where $\widehat{B'_{t/g}}$ and $\widehat{B(\frac{t}{g})}$ denote the separated completions). In particular, $\mathcal{B}'_{t/g}$ is uniform. 
\end{enumerate}
\end{proposition}
\begin{proof}
The assertion (1) follows from Lemma \ref{CompRatTate}. 
Thus, since $A(\frac{t}{g})$ and $B(\frac{t}{g})$ are preuniform by Proposition \ref{RiemannFinEt}, the assertion (2) follows from \cite[Proposition 2.4 (1)]{NS19}. 
Let us prove (3). 
By the assertion (1) and Proposition \ref{cor1547205} (3), $\mathcal{B}'_{t/g}$ is uniform. 
Moreover, by Corollary \ref{IsomCompPowerBd}, the natural $A[\frac{1}{g}]$-algebra homomorphism 
\begin{equation}\label{naturalAg-hom}
\mathcal{B}'_{t/g}\to \widehat{B'^{\circ}_{t/g}}[\frac{1}{t}]
\end{equation}
is an isomorphism which restricts to an isomorphism of rings ${\mathcal{B}'}^{\circ}_{t/g}\xrightarrow{\cong}\widehat{B'^{\circ}_{t/g}}$. 
Since ${\mathcal{B}'}^{\circ}_{t/g}$ and $B'^{\circ}_{t/g}$ are rings of definitions of ${\mathcal{B}'}_{t/g}$ and $B'_{t/g}$ respectively, (\ref{naturalAg-hom}) gives an isomorphism of topological rings ${\mathcal{B}'}_{t/g}\xrightarrow{\cong} \widehat{B'_{t/g}}$. 
The other isomorphism in (\ref{IsomsTateBRat}) follows from the assertion (1) and Proposition \ref{RiemannFinEt} (2). 
\end{proof}

For example, the assumption of regularity in Proposition \ref{rationaluniformity} is realized in the following situation.

\begin{lemma}
\label{PrepForPerAbhy3}
Keep the notations and assumptions as in Proposition \ref{RiemannFinEt}. 
Suppose that $A_{0}$ is a normal ring, and there exists a ring map $f: R\to A_0$ with the following properties. 
\begin{enumerate}
\item
$R$ is a Noetherian N-2 normal domain.\footnote{A domain $R$ is said to be N-2 if for any finite extension of fields $\textnormal{Frac}(R)\subset L$, the ring extension $R\subset R^{+}_{L}$ is module-finite.}
\item
There exist $t_0\in f^{-1}(t)$ and $g_{0}\in f^{-1}(g)$ such that the height of the ideal $(t_{0},g_{0})\subset R$ is $2$. 
\item
$f$ is integral. 
\item
$f(R\setminus \{0\})$ consists of regular elements of $A_{0}$. 
\end{enumerate}
Then, $(t, g)$ forms a regular sequence on $A_{0}$ and $B_{0}$. 
\end{lemma}

To prove Lemma \ref{PrepForPerAbhy3}, we need a more fundamental lemma.

\begin{lemma}
\label{normalext}
Let $A$ be a normal ring that is torsion free and integral over a Noetherian N-2 normal domain $R$. Assume that $(x,y)$ is an ideal of $R$ such that the height of $(x,y)$ is $2$. Then $x,y$ forms a regular sequence on $A$.
\end{lemma}

\begin{proof}
If $A$ is the zero ring, then any sequence in $A$ forms a regular sequence. Thus we assume that $A$ is not the zero ring below. Notice that then the map $R\to A$ is injective because it is torsion free and $R$ is a domain. 

Since $R$ is a Noetherian normal domain, every associated prime of the quotient ring $R/(x)$ is minimal by Serre's normality criterion \cite[Theorem 4.5.3]{SH06}. Thus, $\overline{y} \in R/(x)$ is a regular element by the assumption that the height of $(x,y)$ is $2$.\footnote{The normality of $R$ is necessary. See \cite[Example 2.2.6]{SH06} for an example constructed by Nagata.} That is, $x,y$ is a regular sequence on $R$. Let $\Frac(R)$ be the field of fractions of $R$. Then since $R \to A$ is integral and torsion free, it follows that $\Frac(R) \otimes_R A$ is the total ring of fractions of $A$. Let $\Frac(R) \to B$ be a finite-dimensional subextension of $\Frac(R) \otimes_R A$ and denote by $R^+_B$ be the integral closure of $R$ in $B$. Since $B$ is a reduced finite $\Frac(R)$-algebra, it is a finite product of fields. Hence, $R \to R^+_B$ is module-finite and $R^+_B$ is a Noetherian normal ring because $R$ is N-2. We can write $R^+_B= \prod_{i=1}^m R_i$, where $R \to R_i$ is module-finite and $R_i$ is a normal domain. In this case, we see that the height of $(x,y)R_i$ is $2$ for each $i$ and thus, $x, y$ is a regular sequence on $R^+_B$. Since $A$ is a normal ring, it can be written as a colimit of such $R^+_B$. So we conclude that $x,y$ is regular on $A$.
\end{proof}

Now we can deduce Lemma \ref{PrepForPerAbhy3} easily.

\begin{proof}[Proof of Lemma \ref{PrepForPerAbhy3}]
Since $A_{0}$ is normal, so is $A_{0}[\frac{1}{tg}]$ ($=A[\frac{1}{g}]$). 
Hence the \'etale $A[\frac{1}{g}]$-algebra $B'$ is normal, which implies that $B_{0}$ is normal because it is integrally closed in $B'$. 
Since the map $A[\frac{1}{g}]\to B'$ is flat, it sends any regular element of $A_{0}$ to a regular element of $B_{0}$. 
Thus, since $f(R\setminus \{0\})\subset A_{0}$ consists of regular elements, the composite map 
\begin{equation}\label{compRAB}
R\xrightarrow{f} A_{0}\to B_{0}
\end{equation}
defines the structure as a torsion free integral $R$-algebra on $B_{0}$. 
Therefore, we can apply Lemma \ref{normalext} to deduce the assertion. 
\end{proof}

\section{A variant of perfectoid Abhyankar's lemma for almost Witt-perfect rings}
\label{SecPAbhy}

Let $p>0$ be a prime number. For the sake of reader's convenience, we recall the definition of \textit{Witt-perfect rings} due to Davis and Kedlaya; see \cite{DK14} and \cite{DK15}.

\begin{definition}[Witt-perfect ring]
For a prime number $p>0$, we say that a ring $A$ is \textit{p-Witt-perfect} (simply \textit{Witt-perfect}), if the Witt-Frobenius map $ \mathbf{F}:\mathbf{W}_{p^n}(A) \to \mathbf{W}_{p^{n-1}}(A)$ is surjective for all $n \ge 2$. 
\end{definition}

The most part of this article will deal with $p$-torsion free Witt-perfect rings. Let us recall the almost version of the Witt-perfect condition as introduced in \cite{NS19}.

\begin{definition}[Almost Witt-perfect ring]
\label{AlmostWittRing}
Let $A$ be a $p$-torsion free ring with an element $\varpi \in A$ admitting a compatible system of $p$-power roots $\varpi^{\frac{1}{p^n}} \in A$. Then we say that $A$ is \textit{$(\varpi)^{\frac{1}{p^\infty}}$-almost Witt-perfect}, if the following conditions are satisfied. 
\begin{enumerate}
\item
The Frobenius endomorphism on $A/(p)$ is $(\varpi)^{\frac{1}{p^\infty}}$-almost surjective.
\item
For every $a\in A$ and every $n>0$, there is an element $b \in A$ such that $b^p \equiv p\varpi^\frac{1}{p^n}a \pmod{p^2}$. 
\end{enumerate}
\end{definition}

For applications, we often consider the case that $\varpi \in A$ is a regular element and $A$ is (completely) integrally closed in $A[\frac{1}{p}]$.\footnote{Without the integral closed (or more generally, $p$-root closed) condition, the kernel of the Frobenius map on $A/(p)$ can be complicated.} If one takes $\varpi=1$, then it is shown that $(\varpi)^{\frac{1}{p^\infty}}$-almost Witt perfectness coincides with the Witt-perfectness; see \cite{NS19} for details. Let us recall the following fact; see \cite[Proposition 3.20]{NS19}.

\begin{proposition}
\label{PerfWitt}
Let $V$ be a $p$-adically separated $p$-torsion free valuation ring and let $A$ be a $p$-torsion free $V[T^\frac{1}{p^\infty}]$-algebra. Set $\varpi^\frac{1}{p^n}:=T^\frac{1}{p^n}\cdot 1\in A$ for every $n\geq 0$ and denote by $\widehat{V}$ and $\widehat{A}$ the $p$-adic completions of $V$ and $A$, respectively. Then the following conditions are equivalent. 

\begin{itemize}
\item[$(a)$]
$V$ is a Witt-perfect valuation ring of rank $1$ and $A$ is $(\varpi)^\frac{1}{p^\infty}$-almost Witt-perfect and integrally closed (resp.\ completely integrally closed) in $A[\frac{1}{p}]$. 
\item[$(b)$]
There exist a perfectoid field $K$ and a $(\varpi)^{\frac{1}{p^\infty}}$-almost perfectoid $K\langle T^\frac{1}{p^\infty}\rangle$-algebra $\mathcal{A}$ with the following properties: 
\begin{itemize}
\item[$\bullet$]
$K$ is a Banach ring associated to $(\widehat{V}, (p))$, the norm on $K$ is multiplicative and $K^\circ=\widehat{V}$;
\item[$\bullet$]
$\mathcal{A}$ is a Banach ring associated to $(\widehat{A},(p))$ and $\widehat{A}$ is open and integrally closed in $\mathcal{A}$ (resp.\ $\mathcal{A}^\circ=\widehat{A}$);
\item[$\bullet$]
the bounded ring map of Banach rings $K\langle T^\frac{1}{p^\infty}\rangle\to \mathcal{A}$ is induced by the ring map $V[T^\frac{1}{p^\infty}]\to A$.  
\end{itemize}
\end{itemize}
\end{proposition}

\begin{remark}
\begin{enumerate}
\item
In Proposition \ref{PerfWitt}, one is allowed to map $T$ to $pg \in A$, in which case $A$ is a $(pg)^{\frac{1}{p^\infty}}$-almost Witt-perfect ring for some $g \in A$. We will consider almost Witt-perfect rings of this type.

\item
The advantage of working with (almost) Witt-perfect rings is in the fact that one need not impose $p$-adic completeness condition on a ring. Let $A:=W(k)[[x_2,\ldots,x_d]]$ be the power series algebra over the ring of Witt vectors of a perfect field $k$ of characteristic $p>0$. Then 
$$
A_{\infty}:=\bigcup_{n>0} W(k)[p^{\frac{1}{p^n}}][[x_2^{\frac{1}{p^n}},\ldots,x_d^{\frac{1}{p^n}}]]
$$
is a Witt-perfect algebra that is an integrally closed domain and integral, faithfully flat over $A$. The $p$-adic completion $\widehat{A_\infty}$ of $A_{\infty}$ is integral perfectoid. While $A \to \widehat{A_\infty}$ remains flat by \cite[Theorem 1.1]{Ye18}, it is not integral. The ring $A_\infty$ will be used essentially in the construction of almost Cohen-Macaulay algebras later. A similar construction for complete ramified regular local rings appears in \cite[Proposition 4.9]{Sh16}.

Another important example of a Witt-perfect ring is given by an arbitrary absolutely integrally closed domain $A$, where $A$ is a faithfully flat $\mathbb{Z}_p$-algebra. The $p$-adic completion $\widehat{A}$ is an integral perfectoid algebra over $\widehat{\mathbb{Z}_p[p^{\frac{1}{p^\infty}}]}$. Indeed, as $A$ is absolutely integrally closed in its field of fractions, it contains $\mathbb{Z}_p^+$. Hence $\widehat{A}$ is a $\widehat{\mathbb{Z}_p^+}$-algebra.
\end{enumerate}
\end{remark}

\subsection{Variants of Riemann's extension theorems}
In the context of commutative ring theory, \textit{Riemann's extension theorem} often means a kind of theorem that gives a satisfactory answer to Problem \ref{RiemannExtProblem} in \S \ref{subsecARS}. 
Such a theorem for perfectoid algebras is a key result to the proof of the Direct Summand Conjecture and its derived variant; see \cite{An2} and \cite{Bh18}. 
In this subsection, we establish two types of decompleted variant of the perfectoid Riemann's extension theorem; see Theorem \ref{RiemannExtadic} and Theorem \ref{RiemannExt}.

\subsubsection{Zariskian Riemann's extension theorem}

We start with recalling the definition of adically Zariskian rings.

\begin{definition}
Let $A$ be ring with an ideal $I \subset A$. Then we say that $A$ is \textit{I-adically Zariskian} if $I$ is contained in every maximal ideal of $A$. 
\end{definition}

Note that Zariskianness is preserved under integral ring maps by \cite[\S 9, Lemma 2]{M86}. In other words, for an $I$-adically Zariskian ring $A$, any integral $A$-algebra $B$ is $IB$-adically Zariskian.

We then introduce an important notion using semivaluations (cf.\ Definition \ref{semivaluation1}). 
In the following definition, for a semivaluation $|\cdot|$, we denote by $V_{|\cdot|}$ the valuation ring described in Definition \ref{semivaluation2}. 

\begin{definition}
Let $D \subset C$ be a ring extension and let us set
$$
\Val(C,D):=\Big\{|\cdot|~\Big|~|\cdot|~\mbox{is a semivaluation on}~C~\mbox{such that}~|D| \le 1~\mbox{and}~V_{|\cdot|}~\mbox{has dimension}~\le 1\Big\}\Big/ \sim,
$$
where $\sim$ is generated by natural equivalence classes of semivaluations.
\end{definition}

Let us prove the following algebraic result.

\begin{proposition}
\label{Cont1}
Let $(C,D)$ be a pair of rings such that $C$ is the localization of $D$ with respect to some multiplicative set consisting of regular elements. Suppose that $D$ is an integral extension of a Noetherian ring $R$. Fix a (possibly empty) subset $\mathcal{S} \subset D$ that consists of only regular elements. Then one has
$$
D^+_C=\Big\{x \in C~\Big|~|x| \le 1~\mbox{for any}~|\cdot| \in \Val(C,D)~\mbox{such that}~|g| \ne 0~\mbox{for every}~g \in \mathcal{S}\Big\},
$$
where $D^+_C$ is the integral closure of $D$ in $C$. In particular, $D^+_C$ does not depend on the choice of $\mathcal{S}$.
\end{proposition}

\begin{proof}
Since the containment $\subset$ is clear by the definition of $\Val(C,D)$, let us prove the reverse containment $\supset$. Let $y \in C$ be such that $|y| \le 1$, where $|\cdot| \in \Val(C,D)$ satisfies $|g| \ne 0$ for every $g \in \mathcal{S}$. Let $D[\frac{1}{y}]$ be the subring of the localization $C[\frac{1}{y}]$ which is generated by $y^{-1}=\frac{1}{y}$ over $D$.\footnote{Notice that $D[\frac{1}{y}]$ is \textit{not} the localization of $D$ with respect to the multiplicative system $\{y^n\}_{n \ge 0}$.} Consider the ring extension $D[\frac{1}{y}] \subset C[\frac{1}{y}]$. First suppose that $y^{-1}$ is a unit in $D[\frac{1}{y}]$. Then we can write
$$
y=\frac{a_0}{y^{n-1}}+\frac{a_1}{y^{n-2}}+\cdots+a_{n-1}
$$
for $a_i \in D$. Then we have $y^n-a_{n-1}y^{n-1}-\cdots-a_0=0$. Hence $y \in C$ is integral over $D$ and $y \in D^+_C$.

To derive a contradiction, suppose that $y^{-1} \in D[\frac{1}{y}]$ is not a unit. 
We may assume that $y$ is not nilpotent. Choose a prime ideal $\fm \subset D[\frac{1}{y}]$ such that $y^{-1} \in \fm$. Let $\fp \subset D[\frac{1}{y}]$ be a minimal prime ideal satisfying $\fp \subset \fm$. On the other hand, $R[\frac{1}{y}] \subset D[\frac{1}{y}]$ is an integral extension and $R[\frac{1}{y}]$ is Noetherian by Hilbert's Basis Theorem. 

Then, one can find a valuation ring $D[\frac{1}{y}]/\fp \subset V \subset \Frac(D[\frac{1}{y}]/\fp)$ such that the center (the maximal ideal) of $V$ contains $y^{-1}$ and the Krull dimension of $V$ is $1$: More concretely, one can construct $V$ in the following way. Let $\fn:=\fm \cap R[\frac{1}{y}]$ and $\fq:=\fp \cap R[\frac{1}{y}]$. Then we have a Noetherian subdomain $R[\frac{1}{y}]/\fq \subset \Frac(R[\frac{1}{y}]/\fq)$. By \cite[Theorem 6.3.2 and Theorem 6.3.3]{SH06}, there is a Noetherian valuation ring $V_{\fn}$ such that $R[\frac{1}{y}]/\fq \subset V_{\fn} \subset \Frac(R[\frac{1}{y}]/\fq)$ and the center of $V_{\fn}$ contains $\overline{\fn} \subset R[\frac{1}{y}]/\fq$. We have the commutative diagram:
$$
\begin{CD}
\Frac(R[\frac{1}{y}]/\fq) @>>> \Frac(D[\frac{1}{y}]/\fp) \\
@AAA @AAA \\
V_{\fn} @>>> V \\
\end{CD}
$$
where $V$ is defined as the localization of the integral closure of $V_{\fn}$ in $\Frac(D[\frac{1}{y}]/\fp)$ (this integral closure is a so-called \textit{Pr\"ufer domain}) at the maximal ideal containing $\overline{\fm}$. So $V$ is a valuation ring of Krull dimension $1$ and we have the composite map $D \to D[\frac{1}{y}] \to V$. Let $|\cdot|_V$ denote the corresponding valuation. Moreover, $\mathcal{S} \subset D$ consists of regular elements and $C[\frac{1}{y}]$ is the localization of $D$, so the image of elements in $\mathcal{S}$ remains regular elements in $C[\frac{1}{y}]$ and thus in the subring $D[\frac{1}{y}]$. As $\fp$ is a minimal prime ideal of $D[\frac{1}{y}]$, $g \notin \fp$  for every $g \in \mathcal{S}$. So we find that $|g| _V\ne 0$ and in particular, this implies that $D \to D[\frac{1}{y}] \to V$ extends to the map $C \to C[\frac{1}{y}] \to \Frac(V)$ and the semivaluation on $(C,D)$ induced by $|\cdot|_V$ gives a point $|\cdot|_C \in \Val(C,D)$.

By our assumption, we have $|y|_C \le 1$. Since $y^{-1} \in V$ is in the center, we know $|y^{-1}|_C< 1$. However, these facts are not compatible with $|y|_C|y^{-1}|_C=|yy^{-1}|_C=1$ and thus, $y^{-1} \in D[\frac{1}{y}]$ must be a unit, as desired.
\end{proof}

The above proposition has the following implication: Keep in mind that $A^+$ stands for an open integrally closed subring in a Tate ring $A$.

\begin{corollary}
\label{Cont2}
Let $(A,A^+)$ be an affinoid Tate ring with a fixed pseudouniformizer $t \in A^+$ such that $A^+$ is $t$-adically Zariskian and $A^+$ is integral over a Noetherian ring. For a regular element $g \in A^+$, let us set $(C,D):=(A[\frac{1}{g}],A^+)$. Then we have
\begin{equation}
\label{integralequal}
D^+_C=\Big\{x \in C~\Big|~|x| \le 1;~\forall~|\cdot| \in \Val(C,D)~\mbox{such that}~|t|<1\Big\}.
\end{equation}
Finally, let $\Val(C,D)_{|t|<1}$ be the set of all elements $|\cdot| \in \Val(C,D)$ for which $0<|t|<1$. Then the natural map $(A,A^+) \to (C,D)$ induces an injection  $\Val(C,D)_{|t|<1} \hookrightarrow \Spa(A,A^+)$.
\end{corollary}

\begin{proof}
Keep the notation as in the proof of Proposition \ref{Cont1}. The point is that one can choose the valuation domain $V$ so as to satisfy the required property. So assume that $y \in A[\frac{1}{g}]$ satisfies $|y| \le 1$ for all $|\cdot| \in \Val(A[\frac{1}{g}],A^+)$ and $y^{-1} \in A^+[\frac{1}{y}]$ is not a unit. Then we can find a maximal ideal $\fm \subset A^+[\frac{1}{y}]$ such that $y^{-1} \in \fm$, which gives the surjection $A^+ \twoheadrightarrow A^+[\frac{1}{y}]/\fm$ and let $\fn \subset A^+$ be its kernel. Then $\fn$ is a maximal ideal of $A^+$. The element $t \in A^+$ is in the Jacobson radical by assumption, so we have $t \in \fn$. There is a chain of prime ideals $\fp \subset \fm \subset A^+[\frac{1}{y}]$ such that $\fp$ is minimal and $t,y^{-1} \in \fm$. Then, we have the associated valuation ring $(V,|\cdot|_V)$ and the map $A^+[\frac{1}{y}]/\fp \hookrightarrow V$. It follows from the above construction that $|t|_V<1$, establishing $(\ref{integralequal})$. As $t$ maps into the maximal ideal of the rank $1$ valuation ring $V$, it follows from \cite[Proposition 7.3.7]{Bh17} that $|\cdot|_V$ pulled back to $A^+$ gives a point of $\Spa(A,A^+)$. Finally, the injectivity of the claimed map is clear from the construction.
\end{proof}

Corollary \ref{Cont2} can be formulated also in terms of adic geometry as in Corollary \ref{Cont3} below, where the notion of 
\textit{maximal separated quotient} plays an essential role. 
For generalities on topological spaces and maximal separated quotients, we refer the reader to \cite[Chapter 0, 2.3(c)]{FK18} and \cite[Definition 2.4.8]{KL15}. See also \cite{Mun14} for the construction and properties of maximal separated quotients in general topology. A detailed construction of the maximal separated quotient of an adic space via rank $1$ valuations is found in \cite[Proposition 7.4.13]{Bh17}.

\begin{corollary}
\label{Cont3}
Let $(A, A^+)$ be an affinoid Tate ring and let $(A_0, (t))$ be a pair of definition of $A$. 
Let $s\in A_0$ be an element such that $t\in sA_0$. Let $X=\textnormal{Spa}(A, A^+)$ and let $U$ be the subspace of $X$:
$$
U:=\Big\{x\in X~\Big|~|s|_{\widetilde{x}}< 1\ ~\mbox{for the maximal generization}~\widetilde{x}~\mbox{of}~x\Big\}.
$$
Suppose that $A_0$ is $s$-adically Zariskian and integral over a Noetherian ring. 
Then we have
$$
A^+=A^\circ=(A_0)^+_A=\Big\{a \in A~\Big|~|a|_x\leq 1~\mbox{for any}~x \in [U]\Big\}, 
$$
where $[U]$ denotes the maximal separated quotient of $U$. 
\end{corollary}

\begin{proof}
Since we have the containments 
$$
(A_0)^+_A\subset A^+\subset A^\circ\subset\Big\{a \in A~\Big|~|a|_x\leq 1 ~\mbox{for any}~x \in [U]\Big\}
$$ 
(the third inclusion holds because $|\cdot|_x$ is of rank $1$), it suffices to show that
\begin{equation}
\label{12072118}
(A_0)^+_A=\Big\{a \in A~\Big|~|a|_x\leq 1 ~\mbox{for any}~x \in [U]\Big\}. 
\end{equation}
By assumption, there exists $g\in A_0$ such that $t=sg$. Let $B$ be the Tate ring associated to $(A_0,(s))$ and $B^+:=(A_0)^+_B$. 
Then we have $A=B[\frac{1}{g}]$, $(A_0)^+_A=(B^+)^+_{B[\frac{1}{g}]}$ and 
\begin{equation}
\label{integralequal2}
(B^+)^+_{B[\frac{1}{g}]}=\Big\{b \in B[\tfrac{1}{g}]~\Big|~|b| \le 1~\mbox{for any}~|\cdot| \in \Val(B[\tfrac{1}{g}],B^+)_{|s|<1}\Big\}
\end{equation}
by Corollary \ref{Cont2}. Let us deduce (\ref{12072118}) from $(\ref{integralequal2})$ by constructing a canonical bijection 
$$
\Val(B[\tfrac{1}{g}],B^+)_{|s|<1}\xrightarrow{\cong}[U]. 
$$ 
Any point $|\cdot|\in\Val(B[\frac{1}{g}],B^+)_{|s|<1}$ satisfies that $|a|\leq 1$ for any $a\in A_0$ and $|t|=|sg|<1$. Thus, since $|\cdot|$ is of rank $1$, $|\cdot|$ gives a continuous semivaluation on $A$ such that $|A^\circ|\leq 1$. Hence we have a canonical injection
\begin{equation}\label{1005WedN}
\Val(B[\tfrac{1}{g}],B^+)_{|s|<1}\hookrightarrow[U].
\end{equation}
Moreover, $B^+\subset A^+$, and $|s|_{x}\neq 0$ for every $x\in [U]$ because $s\in A$ is invertible. 
Hence (\ref{1005WedN}) is also surjective, as desired. 
\end{proof}

Indeed, the following immediate corollary is already documented in a treatise on rigid geometry.

\begin{corollary}[cf.\ {\cite[\textbf{II}, Theorem 8.1.11 and 8.2.19]{FK18}}]
\label{1912132251}
Let $A$ be a complete and separated Tate ring. Suppose that $A$ has a ring of definition $A_0$ that is Noetherian. 
Set $X=\Spa\big(A, (A_0)^+_A\big)$. 
Then we have 
$$
(A_0)^+_A=\big\{a\in A~\big|~|a|_x\leq 1\ \textnormal{for any}\ x\in [X]\big\}. 
$$
\end{corollary}

Let us discuss an application of the above results. Let $A$ be a ring with regular elements $t,g$. For every $j>0$, we let $A^{j}$ denote the Tate ring associated to $(A[\frac{t^{j}}{g}], (t))$. Then the set of ${A}$-algebras $\{{A}^j\}_{j>0}$ naturally forms an inverse system, where ${A}^{j+1} \to {A}^{j}$ is the isomorphism $A[\frac{t^{j+1}}{g}][\frac{1}{t}]\xrightarrow{\cong} A[\frac{t^j}{g}][\frac{1}{t}]$ defined by the rule 
\begin{equation}
\label{Aj-inverse}
\frac{t^{j+1}}{g}\mapsto t\cdot\frac{t^{j}}{g}
\end{equation}
and compatible with the isomorphisms $A[\frac{t^j}{g}][\frac{1}{t}]\xrightarrow{\cong}A[\frac{1}{t g}]$ $(j>0)$. 
Then ${A}^{j+1} \to {A}^{j}$ is a continuous ring map between Tate rings, so that it induces $A^{j+1\circ} \to A^{j\circ}$. 
%
%\begin{lemma}\label{preuniheredit}
%$A^{j}$ is preuniform. 
%\end{lemma}
%\begin{proof}
%$\mathcal{A}^{j}=\widehat{A^{j}}$. Since $\mathcal{A}^{j}$ is uniform, $A^{j}$ is preuniform by \cite[Proposition 2.4]{NS19}. 
%\end{proof}
%The second isomorphism follows from the injectivity of $\mathcal{A}^{j\circ} \to \mathcal{A}^j$. 
There is the following commutative diagram:
$$
\begin{CD}
A @>>> A^{j+1\circ} \\
@| @VVV \\
A @>>> A^{j\circ}. \\
\end{CD}
$$
Now we can prove the following type of extension theorem, which  is fitting into the framework of Zariskian geometry; see \cite{T18} for more details.

\begin{theorem}[Zariskian Riemann's extension theorem]
\label{RiemannExtadic}
Let $A$ be a ring with a regular element $t$ that is $t$-adically Zariskian and integral over a Noetherian ring. Let $g\in A$ be a regular element. 
Let $A^{j}$ be the Tate ring associated to $\big(A[\frac{t^j}{g}], (t)\big)$ for every integer $j>0$. Then we have an isomorphism of rings 
$$
A^+_{A[\frac{1}{tg}]} \xrightarrow{\cong} \varprojlim_{j>0}A^{j\circ}.
$$
\end{theorem}

\begin{proof}
By assumption, we have a canonical ring isomorphism $\varphi_j: A[\frac{1}{tg}]\xrightarrow{\cong}A^{j}$ for each $j>0$. 
By restricting $\varphi_j$ to $A^+_{A[\frac{1}{tg}]}$, 
we obtain the ring map $\varphi^+_j: A^+_{A[\frac{1}{tg}]}\to A^{j\circ}$. Then 
$\{\varphi_j\}_{j>0}$ and $\{\varphi^+_j\}_{j>0}$ induce the commutative diagram of ring maps

\begin{equation}
\begin{CD}\label
{eq3412}
A^+_{A[\frac{1}{tg}]}@>{\varphi^+}>> \varprojlim_{j>0}A^{j\circ} \\
@VVV @VVV \\
A[\frac{1}{tg}]@>\cong>\varphi>\ \varprojlim_{j>0}A^{j}
\end{CD}
\end{equation}
where $\varphi$ is an isomorphism and the vertical maps are injective. 
Thus it suffices to prove that (\ref{eq3412}) is cartesian. Pick $c\in A[\frac{1}{tg}]$ such that $\varphi_j(c)\in A^{j\circ}$ for every $j>0$. 
Let us show that $c$ lies in $A^+_{A[\frac{1}{tg}]}$ by applying Corollary \ref{Cont3}. For this, we consider the $(tg)$-adic topology: 
let $A_{(tg)}$ be the Tate ring associated to $\big(A, (tg)\big)$ (notice that each $A^j$ is also the Tate ring associated to $\big(A[\frac{t^j}{g}], (tg)\big)$). Let $X_{(tg)}=\Spa(A_{(tg)}, A^+_{A_{(tg)}})$, $X_{j}=\Spa(A^{j}, A^{j\circ})$ for each $j>0$, and let $U$ be the subspace 
$$
U=\Big\{x\in X_{(tg)}~\Big|~|t|_{\widetilde{x}}< 1\ ~\mbox{for the maximal generization}~\widetilde{x}~\mbox{of}~x\Big\}
$$
of $X_{(tg)}$. 
Then the underlying ring of $A_{(tg)}$ is equal to $A[\frac{1}{tg}]$, and we have 
$$
A_{A_{(tg)}}^+=\Big\{a \in A_{(tg)}~\Big|~|a|_x\leq 1\ \textnormal{for all}\ x \in [U]\Big\}
$$
by Corollary \ref{Cont3}. 
On the other hand, since
$$
A^{j\circ}=\Big\{a\in A^{j}~\Big|~|a|_{x_j}\leq 1\ \textnormal{for all }x_j\in [X_{j}]\Big\}
$$
by Proposition \ref{propIC=CIC}, we have $|\varphi_j(c)|_{x_j}\leq 1$ for all $j>0$ and all $x_j\in [X_{j}]$. Now since $\varphi_j$ gives a continuous map $\big(A_{(tg)},A^+_{A_{(tg)}}\big) \to (A^j,A^{j\circ})$, (\ref{eq3412}) induces the continuous map $\varinjlim_{j>0}[X_j]\to [X_{(tg)}]$, which factors through $[U]$ because $t\in A^{j}$ is topologically nilpotent. Thus we are reduced to showing that the resulting map $f: \varinjlim_{j>0}[X_{j}]\to [U]$ is surjective. 

Pick $x\in [U]$ and let $|\cdot|_x: {A}_{(tg)}\to \mathbb{R}_{\geq 0}$ be a corresponding semivaluation. Let us find some $j_0>0$ such that the composite 
$$
|\cdot|_{x, j_0}: {A}^{j_0}\to{A}_{(tg)}\xrightarrow{|\mspace{3mu}\cdot\mspace{3mu}|_x}\mathbb{R}_{\geq 0}
$$
gives a point $x_{j_0}\in [X_{j_0}]$ for which $f([x_{j_0}])=x$. 
Since $|t|_x<1$ and $|\cdot|_x$ is of rank $1$, there exists some $j_0>0$ such that $|\frac{t^{j_0}}{g}|_x<1$. 
Then we have $|A[\frac{t^{j_0}}{g}]|_x\leq 1$ because $|A|_x\leq 1$ and $|\cdot|_x$ is of rank $1$. 
Thus, since any $a\in A^{j_0\circ}$ is almost integral over $A[\frac{t^{j_0}}{g}]$ and $|\cdot|_x$ is of rank 1, we have 
$|A^{j_0\circ}|_{x, j_0}\leq 1$. Hence $|\cdot|_{x, j_0}$ gives the desired point $x_{j_0}\in [X_{j_0}]$. 
\end{proof}

\subsubsection{Witt-perfect Riemann's extension theorem}
Next we shall investigate the Riemann's extension problem in the context of Witt-perfect rings by transporting the situation to the case of perfectoid algebras, in which case Riemann's extension theorem has been studied by Andr\'e, Bhatt and Scholze and known to experts. 
Let us start setting up some notation.
\\

$\bf{Notation}$: Fix a prime number $p>0$ and a $p$-torsion free ring $A$ that admits a compatible system of $p$-power roots $g^{\frac{1}{p^n}} \in A$ for a regular element $g \in A$ for $n>0$. Moreover, assume the following premises: 
\begin{enumerate}
\item
$A$ is an algebra over a $p$-adically separated $p$-torsion free Witt-perfect valuation domain $V$ of rank $1$ such that $p^{\frac{1}{p^n}} \in V$ for $n>0$.

\item
$A$ is a $(pg)^{\frac{1}{p^\infty}}$-almost Witt-perfect ring and $A$ is completely integrally closed in $A[\frac{1}{p}]$.

\item
$(p,g)$ is a $(p)^{\frac{1}{p^\infty}}$-almost regular sequence on $A$ (while retaining that $p,g$ are regular elements). 
\end{enumerate}

In this situation, we use the following notation.

\begin{definition}\label{DefAj}
Let $B$ be a $pg$-torsion free $A$-algebra, and let $\widehat{(~)}$ denote the $p$-adic completion. 
\begin{enumerate}
\item
We define $K:=\widehat{V}[\frac{1}{p}]$ and equip it with the norm $||\cdot||_{\widehat{V}, (p), p}$ in the sense of 
\cite[Definition 2.26]{NS19}. 
\item
We define $\mathcal{B}:=\widehat{B}[\frac{1}{p}]$ and equip it with the norm $||\cdot||_{\widehat{B}, (p), p}$  in the sense of \cite[Definition 2.26]{NS19}. 
\item
For every $j>0$, 
we define $B^{j}$ as the Tate ring associated to $(B[\frac{p^{j}}{g}], (p))$  (cf.\ Definition \ref{TateRingDef} (2)). 
\item
Let $\mathcal{C}$ be a Banach $\mathcal{A}$-algebra, and let $||\cdot||_{\mathcal{C}}$ be the norm on it. 
Put 
$$
\mathcal{C}_{0}:=\{c\in \mathcal{C}\ |\ ||c||_{\mathcal{C}}\leq 1\}. 
$$
For every $j>0$, we define $\mathcal{C}^{j}=\widehat{\mathcal{C}_{0}[\frac{p^{j}}{g}]}[\frac{1}{p}]$ and equip it with the norm 
$||\cdot||_{\widehat{\mathcal{C}_{0}[\frac{p^{j}}{g}]}, (p), p}$ in the sense of \cite[Definition 2.26]{NS19}. 
In particular, $\mathcal{B}^{j}$ is the Banach ring $\widehat{\widehat{B}[\frac{p^{j}}{g}]}[\frac{1}{p}]$ equipped with the norm 
$||\cdot||_{\widehat{\widehat{B}[\frac{p^{j}}{g}]}, (p), p}$. 
\end{enumerate}
\end{definition}

Here we should list several remarks.

\begin{remark}
\label{DefAi2}
The notations are as above. 
\begin{enumerate}
\item
By Proposition \ref{PerfWitt} (and its proof), $K$ is a perfectoid field and $\mathcal{A}$ is a $(pg)^{\frac{1}{p^\infty}}$-almost perfectoid algebra over $K$. 

\item
It follows from Lemma \ref{adiccompletion} that $g \in \mathcal{A}$ is a $(p)^{\frac{1}{p^\infty}}$-almost regular element. 

\item
The natural map $\mathcal{A}^{\natural\circ} \hookrightarrow  \mathcal{A}^{\circ}$ is a $(pg)^{\frac{1}{p^\infty}}$-almost isomorphism, where $\mathcal{A}^{\natural}$ denotes the untilt of the tilt of $\mathcal{A}$ and $\mathcal{A}^\natural$ is a perfectoid $K$-algebra in view of \cite[Proposition 3.5.4]{An1}. 

\item
$\mathcal{B}$ is a Banach $\mathcal{A}$-algebra. 
Moreover, ${\mathcal{B}}^{j}$ is a Banach $\mathcal{B}$-algebra that is viewed as a ring of analytic functions on the rational subset $\big\{x \in X~\big|~|p^j| \le |g(x)|\big\}$ of $X:=\Spa(\mathcal{B},\mathcal{B}^\circ)$ (cf.\ \cite[Proposition 1.3 and 1.6]{Hu94}). 

\item
Equip the polynomial ring $\mathcal{B}[T]$ with the Gauss norm (cf.\ \cite[Definition 1.6]{K13}), and consider the completion $\mathcal{B}\langle T\rangle$ of it. 
Let $(gT-p^{j})^{-}$ denote the closure of the ideal $(gT-p^{j})$ in ${\mathcal{B}}\langle T\rangle$. 
Then ${\mathcal{B}}^{j}$ can be identified with $\mathcal{B}\langle T\rangle/(gT-p^{j})^{-}$ as a ring, and the norm on ${\mathcal{B}}^{j}$ is equivalent to the quotient norm on 
$\mathcal{B}\langle T\rangle/(gT-p^{j})^{-}$. 
\end{enumerate}
\end{remark}

The most important case is when $B=A$. Let us investigate several properties of $A^{j}$ and $\mathcal{A}^{j}$. 
First we describe the relationship of them.

\begin{lemma}\label{Aj-completion}
In the situation of Definition \ref{DefAj}, the following assertions hold for every $j>0$. 
\begin{enumerate}
\item
The natural $A$-algebra map: 
\begin{equation}\label{natAalgmap2}
\widehat{A\big[\big(\frac{p^j}{g}\big)^{\frac{1}{p^\infty}}\big]} \rightarrow \widehat{\widehat{A}\big[\big(\frac{p^j}{g}\big)^{\frac{1}{p^\infty}}\big]}
\end{equation}
(where the completions are $p$-adic) is a $(p)^{\frac{1}{p^\infty}}$-almost isomorphism. 
\item
Let $\widehat{A^{j}}$ be the separated completion of $A^{j}$. 
Applying the functor $(~)\otimes_{A}A[\frac{1}{p}]$ to (\ref{natAalgmap2}) yields an isomorphism of topological rings: 
$$
\widehat{A^{j}}\xrightarrow{\cong} \mathcal{A}^{j}. 
$$
\item
Let $\psi_{j}$ denote the composite map $A[\frac{1}{pg}]=A^{j}\to\widehat{A^{j}}\xrightarrow{\cong} \mathcal{A}^{j}$. 
Then we have the following identification of rings: 
\begin{equation}
\label{InterWitt}
A^{j\circ}=A[\frac{1}{pg}] \times_{\mathcal{A}^j} \mathcal{A}^{j\circ}=\Big\{a \in A[\frac{1}{pg}]~\Big|~\psi_{j}(a) \in \mathcal{A}^{j\circ}\Big\}. 
\end{equation}
\end{enumerate}
\end{lemma}
\begin{proof}
(1):  By assumption, $(p,g)$ is a $(p)^{\frac{1}{p^\infty}}$-almost regular sequence on $A$. 
Hence by Proposition \ref{RegularRatLoc}, the natural map $\widehat{A\big[\big(\frac{p^j}{g}\big)^{\frac{1}{p^n}}\big]} \to \widehat{\widehat{A}\big[\big(\frac{p^j}{g}\big)^{\frac{1}{p^n}}\big]}$ is a $(p)^{\frac{1}{p^\infty}}$-almost isomorphism for every $n\geq 0$. Moreover, we have the following commutative diagram: 
\[\xymatrix{
\varinjlim_{n\geq 0}\widehat{A\big[\big(\frac{p^j}{g}\big)^{\frac{1}{p^n}}\big]} \ar[r]^\approx\ar[d]& \varinjlim_{n\geq 0}\widehat{\widehat{A}\big[\big(\frac{p^j}{g}\big)^{\frac{1}{p^n}}\big]}
\ar[d]\\
\widehat{A\big[\big(\frac{p^j}{g}\big)^{\frac{1}{p^\infty}}\big]} \ar[r]^{(\ref{natAalgmap2})}& \widehat{\widehat{A}\big[\big(\frac{p^j}{g}\big)^{\frac{1}{p^\infty}}\big]}.
}\]
The vertical maps become isomorphisms after applying $(~)\otimes_{A}A/p^{m}A$. Hence they become isomorphisms after $p$-adic completion. 
Hence by Lemma \ref{almostcomp}, the lower map (\ref{natAalgmap2}) is also a $(p)^{\frac{1}{p^\infty}}$-almost isomorphism. 

(2): Since $pgA[(\frac{p^j}{g}\big)^{\frac{1}{p^\infty}}]\subset A[\frac{p^{j}}{g}]$ and $p^{j+1}=\frac{p^{j}}{g}\cdot pg$, we have $p^{j+1}A[(\frac{p^j}{g}\big)^{\frac{1}{p^\infty}}]\subset A[\frac{p^{j}}{g}]$. Similarly, $p^{j+1}\widehat{A}[(\frac{p^j}{g}\big)^{\frac{1}{p^\infty}}]\subset \widehat{A}[\frac{p^{j}}{g}]$. 
Hence by \cite[Lemma 2.5]{NS19}, the ring maps $\widehat{A[\frac{p^{j}}{g}]}\to \widehat{A[(\frac{p^j}{g}\big)^{\frac{1}{p^\infty}}]}$ and $\widehat{\widehat{A}[\frac{p^{j}}{g}]}\to \widehat{\widehat{A}[(\frac{p^j}{g}\big)^{\frac{1}{p^\infty}}]}$ are injective, and their cokernels are annihilated by $p^{j+1}$. 
In particular, they become isomorphisms after inverting $p$. Hence one can regard $\widehat{A[(\frac{p^j}{g}\big)^{\frac{1}{p^\infty}}]}$ and $\widehat{\widehat{A}[(\frac{p^j}{g}\big)^{\frac{1}{p^\infty}}]}$ as rings of definition of $\widehat{A^{j}}$ and $\mathcal{A}^{j}$, respectively.

Since (\ref{natAalgmap2}) is $(p)^{\frac{1}{p^\infty}}$-almost bijective, it becomes an isomorphism after inverting $p$. 
Thus, applying the functor $(~)\otimes_{A}A[\frac{1}{p}]$ to (\ref{natAalgmap2}) yields a canonical isomorphism of rings 
\begin{equation}\label{AjmathcalAj}
\widehat{A^{j}}\xrightarrow{\cong}\mathcal{A}^{j}, 
\end{equation}
which restricts to (\ref{natAalgmap2}). Moreover, (\ref{natAalgmap2}) is a $(p)^{\frac{1}{p^\infty}}$-almost surjective embedding from a ring of definition of  $\widehat{A^{j}}$ into that of $\mathcal{A}^{j}$. 
In particular, $p(\widehat{\widehat{A}[(\frac{p^j}{g}\big)^{\frac{1}{p^\infty}}]})$ is contained in the image of $\widehat{A[(\frac{p^j}{g}\big)^{\frac{1}{p^\infty}}]}$ via (\ref{AjmathcalAj}).  Therefore, (\ref{AjmathcalAj}) is also a homeomorphism, as desired.

(3): Since $A^{j}=A[\frac{1}{pg}]$ as rings, this assertion immediately follows from the assertion (2) and Corollary \ref{VariantBL}. 
\end{proof}

Next we discuss preuniformity (resp.\ uniformity) of $A^{j}$ (resp.\ $\mathcal{A}^{j}$).

\begin{proposition}
\label{pg-approx}
In the situation of Definition \ref{DefAj}, the following assertions hold for every $j>0$. 
\begin{enumerate}
\item
$\mathcal{A}^{j}$ is a perfectoid $K$-algebra. In particular, $\mathcal{A}^{j}$ is uniform. 
\item
There is an $A[\frac{p^j}{g}]$-algebra homomorphism: 
\begin{equation}
\label{Scholzeapprox}
\widehat{A\big[\big(\frac{p^j}{g}\big)^{\frac{1}{p^\infty}}\big]} \xrightarrow{\approx}\mathcal{A}^{j \circ}
\end{equation}
that is $(p)^{\frac{1}{p^\infty}}$-almost bijective. 
\item
The inclusion map: 
\begin{equation}
\label{Scholzeapprox2}
A\big[\big(\frac{p^j}{g}\big)^{\frac{1}{p^\infty}}\big] \hookrightarrow A^{j \circ}
\end{equation}
 is a $(p)^{\frac{1}{p^\infty}}$-almost isomorphism. In particular, $A^j$ is preuniform.

\item
There is an $A[\frac{p^j}{g}]$-algebra isomorphism:
$$
\widehat{A^{j\circ}} \xrightarrow{\cong} \mathcal{A}^{j\circ}.
$$
Moreover, $A^{j\circ}$ is Witt-perfect.
\end{enumerate}
\end{proposition}

\begin{proof}
(1): Let $\mathcal{A}^{\natural}$ denote the untilt of the tilt of $\mathcal{A}$. Then $\mathcal{A}^\natural$ is a perfectoid $K$-algebra in view of \cite[Proposition 3.5.4]{An1}. Since the natural map $\mathcal{A}^{\natural\circ} \hookrightarrow  \mathcal{A}^{\circ}$ is a $(pg)^{\frac{1}{p^\infty}}$-almost isomorphism, the induced map $\mathcal{A}^{\natural\circ}[\frac{p^{j}}{g}] \hookrightarrow  \mathcal{A}^{\circ}[\frac{p^{j}}{g}]$ is a $(p)^{\frac{1}{p^\infty}}$-almost isomorphism. 
Therefore, the induced map $\widehat{\mathcal{A}^{\natural\circ}[\frac{p^{j}}{g}]} \rightarrow  \widehat{\mathcal{A}^{\circ}[\frac{p^{j}}{g}]}$ is also a $(p)^{\frac{1}{p^\infty}}$-almost isomorphism by Lemma \ref{almostcomp}.  Hence inverting $p$ yields an isomorphism of topological rings $(\mathcal{A}^{\natural})^{j}\to \mathcal{A}^{j}$. 
On the other hand, $(\mathcal{A}^\natural)^j$ is a perfectoid $K$-algebra by \cite[Theorem 6.3 (ii)]{Sch12}. 
Thus, $\mathcal{A}^j$ is a perfectoid $K$-algebra. 

(2): By the above, $(\mathcal{A}^{\natural})^{j}\to \mathcal{A}^{j}$ restricts to an isomorphism of rings $(\mathcal{A}^{\natural})^{j\circ}\xrightarrow{\cong}\mathcal{A}^{j\circ}$. 
Moreover, $\mathcal{A}^{\natural\circ} \hookrightarrow  \mathcal{A}^{\circ}$ induces the $(p)^{\frac{1}{p^\infty}}$-almost isomorphism 
$\widehat{\mathcal{A}^{\natural\circ}\big[\big(\frac{p^j}{g}\big)^{\frac{1}{p^\infty}}\big]}\xrightarrow{\approx}\widehat{\mathcal{A}^{\circ}\big[\big(\frac{p^j}{g}\big)^{\frac{1}{p^\infty}}\big]}$ by Lemma \ref{almostcomp}, and $\mathcal{A}^{\circ}=\widehat{A}$. 
Thus we have the following commutative diagram: 
\[\xymatrix{
\widehat{\mathcal{A}^{\natural\circ}\big[\big(\frac{p^j}{g}\big)^{\frac{1}{p^\infty}}\big]}\ar[d]\ar[r]^{\approx}&\ar[d]\widehat{\widehat{A}\big[\big(\frac{p^j}{g}\big)^{\frac{1}{p^\infty}}\big]}\\
(\mathcal{A}^\natural)^{j\circ}\ar[r]^{\cong}&\mathcal{A}^{j\circ}
}\]
where the left vertical map is a $(p)^{\frac{1}{p^\infty}}$-almost isomorphism in view of Scholze's result \cite[Lemma 6.4]{Sch12}. 
Hence the right vertical map:
\begin{equation}\label{pgalmostisom}
\widehat{\widehat{A}\big[\big(\frac{p^j}{g}\big)^{\frac{1}{p^\infty}}\big]} \xrightarrow{\approx} \mathcal{A}^{j \circ}
\end{equation} 
is a $(p)^{\frac{1}{p^\infty}}$-almost isomorphism. 
By considering the composition of (\ref{natAalgmap2}) and (\ref{pgalmostisom}), we obtain the desired $(p)^{\frac{1}{p^\infty}}$-almost  isomorphism
\begin{equation}
\label{Scholzeapprox}
\widehat{A\big[\big(\frac{p^j}{g}\big)^{\frac{1}{p^\infty}}\big]} \xrightarrow{\approx} \mathcal{A}^{j \circ}.
\end{equation}

(3): First recall that $A^j=A\big[\big(\frac{p^j}{g}\big)^{\frac{1}{p^\infty}}\big][\frac{1}{p}]$ as rings (see the proof of Lemma \ref{Aj-completion} (2)). By Lemma \ref{Beauville-Laszlo}, we have $A\big[\big(\frac{p^j}{g}\big)^{\frac{1}{p^\infty}}\big]=A^j \times_{\widehat{A\big[\big(\frac{p^j}{g}\big)^{\frac{1}{p^\infty}}\big] }[\frac{1}{p}]}\widehat{A\big[\big(\frac{p^j}{g}\big)^{\frac{1}{p^\infty}}\big] }$. 
Moreover, the  $(p)^{\frac{1}{p^\infty}}$-almost isomorphism (\ref{Scholzeapprox}) extends to an isomorphism of rings 
$\widehat{A\big[\big(\frac{p^j}{g}\big)^{\frac{1}{p^\infty}}\big]}[\frac{1}{p}] \xrightarrow{\cong} \mathcal{A}^{j}$. 
Hence we have $A\big[\big(\frac{p^j}{g}\big)^{\frac{1}{p^\infty}}\big]=A[\frac{1}{pg}] \times_{\mathcal{A}^j}\widehat{A\big[\big(\frac{p^j}{g}\big)^{\frac{1}{p^\infty}}\big] }$. 
On the other hand, by Lemma \ref{Aj-completion} (3), $A^{j\circ}=A[\frac{1}{pg}] \times_{\mathcal{A}^j}\mathcal{A}^{j \circ}$. 
By construction, the diagram of rings: 
\[\xymatrix{
A\big[\big(\frac{p^j}{g}\big)^{\frac{1}{p^\infty}}\big]\ar@^{(->}[rr]\ar[d]&&A^{j\circ}\ar[d]\\
\widehat{A\big[\big(\frac{p^j}{g}\big)^{\frac{1}{p^\infty}}\big]}\ar[rr]_{(\ref{Scholzeapprox})}&&\mathcal{A}^{j\circ}
}\]
is commutative. Thus, the ring map 
\begin{equation}\label{bscgAjmcalAj}
A[\frac{1}{pg}] \times_{\mathcal{A}^j}\widehat{A\big[\big(\frac{p^j}{g}\big)^{\frac{1}{p^\infty}}\big] }\to A[\frac{1}{pg}] \times_{\mathcal{A}^j}\mathcal{A}^{j \circ}
\end{equation}
induced by (\ref{Scholzeapprox}) is isomorphic to $A\big[\big(\frac{p^j}{g}\big)^{\frac{1}{p^\infty}}\big]\hookrightarrow A^{j\circ}$. 
Moreover, by Lemma \ref{PullbackRings}, (\ref{bscgAjmcalAj})  
is a $(p)^{\frac{1}{p^\infty}}$-almost isomorphism. 
Hence the first assertion follows. 
In particular, we have $p^{j+2}A^{j\circ}\subset p^{j+1}A[\big(\frac{p^j}{g}\big)^{\frac{1}{p^\infty}}]\subset A[\frac{p^{j}}{g}]$. Therefore, $A^{j}$ is preuniform.

(4): 
Since $A^{j}$ is preuniform, one can apply \cite[Proposition 2.4 (1)-(b)]{NS19} 
to the inclusion map $A[\frac{p^{j}}{g}]\hookrightarrow A^{j\circ}$ and deduce that the induced map $\widehat{A[\frac{p^{j}}{g}]}\rightarrow \widehat{A^{j\circ}}$ extends to an isomorphism of rings 
\begin{equation}\label{widehatcircle}
(\widehat{A^{j}})^{\circ}\xrightarrow{\cong}\widehat{A^{j\circ}}. 
\end{equation}
On the other hand, Lemma \ref{Aj-completion} (2) allows us to extend (\ref{natAalgmap2}) to a ring isomorphism: 
\begin{equation}\label{AjhatmcalAj}
(\widehat{A^{j}})^{\circ}\xrightarrow{\cong} \mathcal{A}^{j\circ}. 
\end{equation} 
By composing the inverse map of (\ref{widehatcircle}) and (\ref{AjhatmcalAj}), we obtain the desired isomorphism. 
In particular, $A^{j\circ}/(p)\cong\mathcal{A}^{j\circ}/(p)$ and $A^{j\circ}/(p^{2})\cong\mathcal{A}^{j\circ}/(p^{2})$. 
Since $\mathcal{A}^{j}$ is perfectoid, $\mathcal{A}^{j\circ}$ is Witt-perfect by Proposition \ref{PerfWitt}. 
Hence $A^{j\circ}$ is also Witt-perfect. 
\end{proof}

The set of $\mathcal{A}$-algebras $\{\mathcal{A}^j\}_{j>0}$ forms an inverse system, where $\mathcal{A}^{j+1} \to \mathcal{A}^{j}$ is the natural inclusion defined by 
\begin{equation}
\label{transition}
\frac{p^{j+1}}{g}\mapsto p\cdot\frac{p^{j}}{g}. 
\end{equation}
Then $\mathcal{A}^{j+1} \to \mathcal{A}^{j}$ is a continuous map between Banach $K$-algebras, so that it induces $\mathcal{A}^{j+1\circ} \to \mathcal{A}^{j\circ}$. Recall that we already defined an inverse system $\{A^{j}\}_{j> 0}$ in a similar way; see $(\ref{Aj-inverse})$. After the preparations we have made above, we will establish \textit{Witt-perfect Riemann's Extension Theorem} (see Theorem \ref{RiemannExt} below). 
Notice that it is independent of Zariskian Riemann's extension theorem (Theorem \ref{RiemannExtadic}).

\begin{theorem}[Witt-perfect Riemann's extension theorem]
\label{RiemannExt}
Let $A$ be a $p$-torsion free algebra over a $p$-adically separated $p$-torsion free Witt-perfect valuation domain $V$ of rank $1$ admitting a compatible system of $p$-power roots $p^{\frac{1}{p^n}} \in V$, together with a regular element $g \in A$ admitting a compatible system of $p$-power roots $g^{\frac{1}{p^n}} \in A$. Denote by $\widehat{(~)}$ the $p$-adic completion and suppose that the following conditions hold:
\begin{enumerate}
\item
$A$ is a $(pg)^{\frac{1}{p^\infty}}$-almost Witt-perfect ring and $A$ is completely integrally closed in $A[\frac{1}{p}]$. 

\item
$(p,g)$ is a $(p)^{\frac{1}{p^\infty}}$-almost regular sequence on $A$ (which merely says that $g$ is a $(p)^{\frac{1}{p^\infty}}$-almost regular element on
$A/(p)$). 
\\

Then the following assertions hold. 
\begin{enumerate}
\item
We have the following identification of rings: 
\begin{equation}
\label{fiberproduct1}
\varprojlim_{j>0}A^{j\circ}=A[\frac{1}{pg}]\times_{\mathcal{A}[\frac{1}{g}]}g^{-\frac{1}{p^{\infty}}}\mathcal{A}^{\circ}. 
\end{equation}
\item
There is an injective $A$-algebra map:
$$
\widehat{\varprojlim_{j>0} A^{j\circ}} \hookrightarrow \varprojlim_{j>0} \widehat{A^{j\circ}},
$$
whose cokernel is $(g)^{\frac{1}{p^\infty}}$-almost zero.
\end{enumerate}
\end{enumerate}
\end{theorem}

\begin{proof}
(a): By Lemma \ref{Aj-completion} (3), we have canonical isomorphisms: 
\begin{eqnarray}
\label{fiberproduct2}
\varprojlim_{j>0} A^{j\circ} \cong\varprojlim_{j>0}\Big(A[\frac{1}{pg}] \times_{\mathcal{A}^j} \mathcal{A}^{j\circ}\Big) \cong A[\frac{1}{pg}] \times_{\varprojlim_{j>0}\mathcal{A}^j} \varprojlim_{j>0}\mathcal{A}^{j\circ}.  \nonumber
\end{eqnarray}
On the other hand, it follows from Riemann's extension theorem for (almost) perfectoid $K$-algebras \cite[Th\'eor\`eme 4.2.2]{An1} (see Theorem \ref{Hebbarkeits1} for the detailed proof) that there is an $\mathcal{A}^{\circ}$-algebra isomorphism:
\begin{equation}
\label{completionflat1}
g^{-\frac{1}{p^\infty}}\mathcal{A}^\circ \xrightarrow{\cong} \varprojlim_{j>0} \mathcal{A}^{j\circ}.
\end{equation}
Moreover, we have the commutative diagram of rings: 
\[\xymatrix{
A\ar@{^(->}[dd]\ar[rr]\ar[dr]&&g^{-\frac{1}{p^\infty}}\mathcal{A}^\circ\ar[dl]^{(\ref{completionflat1})}\ar@^{(->}[dd]\\
&\varprojlim_{j>0}\mathcal{A}^{j\circ}\ar@^{(->}[dd]&\\
A[\frac{1}{pg}]\ar[rr]\ar[rd]_\alpha&&\ar[ld]^\beta (g^{-\frac{1}{p^\infty}}\mathcal{A}^\circ)[\frac{1}{pg}]\\
&\varprojlim_{j>0}\mathcal{A}^{j}&
}\]
where $\alpha$ and $\beta$ are the unique maps induced by the universal property of localization. 
Since the composite map $g^{-\frac{1}{p^\infty}}\mathcal{A}^\circ\xrightarrow{\cong} \varprojlim_{j>0}\mathcal{A}^{j\circ}\hookrightarrow \varprojlim_{j>0}\mathcal{A}^{j}$ is injective, $\beta$ is also injective. Thus we have 

$$
A[\frac{1}{pg}] \times_{\varprojlim_{j>0}\mathcal{A}^j} \varprojlim_{j>0}\mathcal{A}^{j\circ} \cong 
A[\frac{1}{pg}] \times_{\varprojlim_{j>0}\mathcal{A}^j} g^{-\frac{1}{p^\infty}}\mathcal{A}^\circ\cong 
A[\frac{1}{pg}] \times_{(g^{-\frac{1}{p^\infty}}\mathcal{A}^\circ)[\frac{1}{pg}]} g^{-\frac{1}{p^\infty}}\mathcal{A}^\circ. 
$$
Since $(g^{-\frac{1}{p^\infty}}\mathcal{A}^\circ)[\frac{1}{pg}]=\mathcal{A}[\frac{1}{g}]$, the assertion follows. 

(b): 
Since the natural maps $A\to A^{j\circ}$ $(j> 0)$ are compatible with the inverse system $\{A^{j\circ}\}_{j> 0}$, 
one can define a canonical structures as an $A$-algebra on each one of the rings 
\begin{equation}\label{313MonN}
\widehat{\varprojlim_{j> 0}A^{j\circ}}=\varprojlim_{n>0}\Big(\big(\varprojlim_{j>0} A^{j\circ}\big)/(p^n)\Big)\ \ \textnormal{and}\ \ 
\varprojlim_{j>0} \widehat{A^{j\circ}}=\varprojlim_{j>0}\varprojlim_{n>0}\big(A^{j\circ}/(p^n)\big). 
\end{equation}
For a fixed $n>0$, consider the exact sequence of inverse systems of $A$-algebras: $0 \to \{A^{j\circ}\}_{j>0} \xrightarrow{p^n} \{A^{j\circ}\}_{j>0} \to \{A^{j\circ}/(p^n)\}_{j>0} \to 0$. Then this induces an injective ring map
\begin{equation}\label{3131813N}
\big(\varprojlim_{j>0} A^{j\circ}\big)/(p^n) \hookrightarrow  \varprojlim_{j>0}\big(A^{j\circ}/(p^n)\big). 
\end{equation}
In view of (\ref{313MonN}), taking the inverse limit with respect to $n>0$ yields the composite ring map
\begin{equation}
\label{8291230}
\widehat{\varprojlim_{j>0} A^{j\circ}} \hookrightarrow  \varprojlim_{n>0}\varprojlim_{j>0}\big(A^{j\circ}/(p^n)\big) \xrightarrow{\cong} \varprojlim_{j>0} \widehat{A^{j\circ}},
\end{equation}
which gives an injective $A$-algebra map. Now we want to prove that $(\ref{8291230})$ is $(g)^{\frac{1}{p^\infty}}$-almost surjective. 
Note that the map $A\to \widehat{\varprojlim_{j>0} A^{j\circ}}$ extends to $\widehat{A}\to \widehat{\varprojlim_{j>0} A^{j\circ}}$ by the universality of completion (see \cite[Proposition 7.1.9 in Chapter 0]{FK18}). Hence $(\ref{8291230})$ yields the composite map of $A$-algebras
\begin{equation}
\label{82912302}
\widehat{A} \to \widehat{\varprojlim_{j>0} A^{j\circ}} \xrightarrow{(\ref{8291230})} \varprojlim_{j>0} \widehat{A^{j\circ}}. 
\end{equation}
On the other hand, by the universality of completion again, we have the commutative squares: 
\[\xymatrix{
A\ar[r]\ar[d]&A^{j\circ}\ar[d]\\
\widehat{A}\ar[r]&\widehat{A^{j\circ}}
}\]
($j> 0$) of which the bottom arrows are compatible with the inverse system $\{\widehat{A^{j\circ}}\}_{j> 0}$. Hence we obtain an $A$-algebra map 
\begin{equation}
\label{naturalringmapA}
\widehat{A}\to \varprojlim_{j\geq 0}\widehat{A^{j\circ}}, 
\end{equation}
which extends to the isomorphism $(\ref{completionflat1})$ as described in the proof of Theorem \ref{Hebbarkeits1}.
In particular, $(\ref{naturalringmapA})$ is $(g)^{\frac{1}{p^\infty}}$-almost surjective. 
Here, since $\varprojlim_{j>0} \widehat{A^{j\circ}}$ is $p$-adically separated, an extension of the map $A\to \varprojlim_{j>0} \widehat{A^{j\circ}}$ along the completion $A\to \widehat{A}$ is unique. 
Therefore, (\ref{82912302}) is identified with (\ref{naturalringmapA}). 
Thus we conclude that  $(\ref{82912302})$ and hence $(\ref{8291230})$ are $(g)^{\frac{1}{p^\infty}}$-almost surjective. 
This completes the proof of the assertion. 
\end{proof}

By combining Theorem \ref{RiemannExt} with Theorem \ref{RiemannExtadic}, we obtain the following corollary.

\begin{corollary}
\label{corWittZar}
Keep the notations and the hypotheses as in Theorem \ref{RiemannExt}. Suppose further that $A$ is $p$-adically Zariskian and integral over a Noetherian ring. 
Then we have the equality:
\begin{equation}
\label{fiberproduct1}
A_{A[\frac{1}{pg}]}^+=\Big\{x \in A[\frac{1}{pg}]~\Big|~\widetilde{\psi}(x) \in g^{-\frac{1}{p^\infty}}\mathcal{A}^\circ\Big\},
\end{equation}
where $\widetilde{\psi}: A[\frac{1}{pg}]\to \mathcal{A}[\frac{1}{g}]$ is the natural map. 
\end{corollary}
\begin{proof}
By Theorem \ref{RiemannExtadic} and Theorem \ref{RiemannExt} (a), we have the commutative diagram of rings: 
\[\xymatrix{
A^{+}_{A[\frac{1}{pg}]}\ar[r]^{\cong\ \ \ }&\varprojlim_{j>0}A^{j\circ}\ar[r]^{\cong\ \ \ \ \ \ \ }&A[\frac{1}{pg}]\times_{\mathcal{A}[\frac{1}{g}]}g^{-1/p^{\infty}}\mathcal{A}^{\circ}\ar[d]^{\pi_{1}}\\
A\ar@^{(->}[u]\ar@^{(->}[rr]&&A[\frac{1}{pg}]
}\]
where $\pi_{1}$ is the projection map. 
Since $\pi_{1}$ is injective and $\im (\pi_{1})=\{x \in A[\frac{1}{pg}]~\Big|~\widetilde{\psi}(x) \in g^{-\frac{1}{p^\infty}}\mathcal{A}^\circ\}$, 
the assertion follows. 
\end{proof}

\begin{Discussion}
Here is an alternative way to deduce Corollary \ref{corWittZar}. 
Since $\mathcal{A}^{j\circ}$ is completely integrally closed in $\mathcal{A}^{j\circ}[\frac{1}{pg}]$, it follows that the right-hand side of $(\ref{completionflat1})$ is completely integrally closed after inverting $pg$ by Lemma \ref{lem07032}. This implies that $g^{-\frac{1}{p^\infty}}\mathcal{A}^\circ$ is completely integrally closed after inverting $pg$. Thus, $A_{A[\frac{1}{pg}]}^+$ is contained in the right-hand side of $(\ref{fiberproduct1})$, and it remains to prove the other inclusion. Note that $A \to g^{-\frac{1}{p^\infty}}A$ is almost integral and $A \subset g^{-\frac{1}{p^\infty}}A \subset A_{A[\frac{1}{pg}]}^+$ by Proposition \ref{propIC=CIC}. So Corollary \ref{Cont2} gives us
\begin{equation}
\label{fiberproduct3}
A_{A[\frac{1}{pg}]}^+=\Big\{x \in C~\Big|~|x| \le 1;~\forall~|\cdot| \in \Val(C,D)_{|p|<1}\Big\}
\end{equation}
by setting $(C,D):=(g^{-\frac{1}{p^\infty}}A[\frac{1}{pg}],g^{-\frac{1}{p^\infty}}A)$, where $(g^{-\frac{1}{p^\infty}}A)[\frac{1}{p}]$ is equipped with the canonical structure as a Tate ring by declaring that $g^{-\frac{1}{p^\infty}}A$ is a ring of definition and the topology is $p$-adic. A result of Huber \cite[Proposition 3.9]{Hu93} asserts that\footnote{Notice that $(g^{-\frac{1}{p^\infty}}A)[\frac{1}{p}]$ may differ from $g^{-\frac{1}{p^\infty}}(A[\frac{1}{p}])$. But the former is contained in the latter and Lemma \ref{Lem0724} applies to claim that $g^{-\frac{1}{p^\infty}}A$ is an integrally closed subring of
$(g^{-\frac{1}{p^\infty}}A)[\frac{1}{p}]$.}
\begin{eqnarray}
\Val(C,D)_{|p|<1} \hookrightarrow \Spa\Big((g^{-\frac{1}{p^\infty}}A)[\frac{1}{p}],g^{-\frac{1}{p^\infty}}A\Big) \nonumber \\
\cong \Spa\Big((g^{-\frac{1}{p^\infty}}\widehat{A})[\frac{1}{p}],g^{-\frac{1}{p^\infty}}\widehat{A}\Big) \cong \Spa\Big((g^{-\frac{1}{p^\infty}}\mathcal{A^\circ})[\frac{1}{p}],g^{-\frac{1}{p^\infty}}\mathcal{A}^\circ\Big) \nonumber,
\end{eqnarray}
which shows that any $|\cdot| \in \Val(C,D)_{|p|<1}$ extends to an element $|\cdot| \in \Spa\Big((g^{-\frac{1}{p^\infty}}\mathcal{A^\circ})[\frac{1}{p}],g^{-\frac{1}{p^\infty}}\mathcal{A}^\circ\Big)$ for which we know $|x| \le 1$ for $x \in g^{-\frac{1}{p^\infty}}\mathcal{A}^\circ$. This fact combined with $(\ref{fiberproduct3})$ yields the following:
$$
A_{A[\frac{1}{pg}]}^+ \subset \Big\{x \in A[\frac{1}{pg}]~\Big|~\widetilde{\psi}(x) \in g^{-\frac{1}{p^\infty}}\mathcal{A}^\circ\Big\} \subset \Big\{x \in C~\Big|~|x| \le 1;~\forall~|\cdot| \in \Val(C,D)_{|p|<1}\Big\}=A_{A[\frac{1}{pg}]}^+,
$$
so that $(\ref{fiberproduct1})$ has been proved. 
\end{Discussion}

\begin{remark}
\begin{enumerate}
\item
Witt-perfect rings are almost never Noetherian and thus, it is natural to ask whether such algebras could be integral over a Noetherian ring. One way for constructing such an algebra over a Noetherian normal domain $R$ is to take the \textit{maximal \'etale extension} of $R$. The details are found in \cite{Sh18} and \cite{Sh19}. We will apply this method to construct almost Cohen-Macaulay algebras in \S\ref{SecAppWPA}.

\item
As we have seen so far, almost Witt-perfect rings play a centra role. So the next question naturally arises in view of the Scholze's crucial result that the structure presheaf of a perfectoid space is indeed a sheaf. Let $(A,A^+)$ be an affinoid Tate ring such that $A^+$ is almost Witt-perfect and completely integrally closed in $A$. Then is the pair $(A,A^+)$ sheafy, or is it stably uniform? Some relevant results are found in the papers \cite{BV18} and \cite{Mi16}. 
\end{enumerate}
\end{remark}

\subsection{Witt-perfect Abhyankar's lemma}

Now we are prepared to prove the main result, which is a variant of Andr\'e's Perfectoid Abhyankar's Lemma. Here is the statement of the above main theorem.

\begin{theorem}[Witt-perfect Abhyankar's lemma]
\label{PerAbhyankar}
Let $A$ be a $p$-torsion free algebra over a $p$-adically separated $p$-torsion free Witt-perfect valuation domain $V$ of rank $1$ admitting a compatible system of $p$-power roots $p^{\frac{1}{p^n}} \in V$, together with a regular element $g \in A$ admitting a compatible system of $p$-power roots $g^{\frac{1}{p^n}} \in A$. Suppose that the following conditions hold.
\begin{enumerate}
\item
$A$ is a $p$-adically Zariskian and normal ring.

\item
$A$ is a $(pg)^{\frac{1}{p^\infty}}$-almost Witt-perfect ring.

\item
$A$ is torsion free and integral over a Noetherian normal domain $R$ such that $g \in R$ and the height of the ideal $(p,g) \subset R$ is $2$.
\\

Let $A[\frac{1}{pg}] \hookrightarrow B'$ be a finite \'etale extension. 
Denote by $B:=(g^{-\frac{1}{p^{\infty}}}A)_{B'}^+$ the integral closure of $g^{-\frac{1}{p^{\infty}}}A$
in $B'$. Then the following statements hold:
\begin{enumerate}
\item
The Frobenius endomorphism $Frob:B/(p) \to B/(p)$ is $(pg)^{\frac{1}{p^\infty}}$-almost surjective and it induces an injection $B/(p^{\frac{1}{p}}) \hookrightarrow B/(p)$.

\item
The induced map $A/(p^m) \to B/(p^m)$ is $(pg)^{\frac{1}{p^\infty}}$-almost finite \'etale for all $m>0$.
\end{enumerate}
\end{enumerate}
\end{theorem}

We first prove the following preliminary result, which substantially contains the assertion (a) of the theorem.

\begin{proposition}
\label{preprop}
Keep the notation and the assumption as in Theorem \ref{PerAbhyankar}. 
Then the following assertions hold. 
\begin{enumerate}
\item
$B$ is the integral closure of $A$ in $B'$. 
\item
For every $j>0$, the following assertions hold (see Definition \ref{DefAj} for the notation). 
\begin{enumerate}
\item
Equip $B'$ with the canonical structure as a Tate ring that is module-finite over $A^{j}$ (as in Lemma \ref{tatefinex}). 
Then $B'=B^{j}$ as topological rings. 
In particular, the ring extension $A[\frac{1}{pg}]\hookrightarrow B'$ is identified with a continuous ring map
\begin{equation}
\label{AjBjnatural}
A^{j}\to B^{j}.
\end{equation}

\item
Equip $B'\otimes_{A[\frac{1}{pg}]}\mathcal{A}^j$ with the canonical structure as a Tate ring that is module-finite over $\mathcal{A}^j$. 
Then $B'\otimes_{A[\frac{1}{pg}]}\mathcal{A}^{j}\cong \mathcal{B}^{j}$ as topological rings. 
\item
$\mathcal{B}^j$ is a perfectoid $K$-algebra. 
\item
The restriction $A^{j\circ} \to B^{j\circ}$ of (\ref{AjBjnatural}) is integral. Moreover, it is $(p)^{\frac{1}{p^\infty}}$-almost finite \'etale, and $B^{j\circ}$ is a Witt-perfect $V$-algebra.
\end{enumerate}
\item
The natural ring map $B\to\varprojlim_{j>0}B^{j\circ}$ (cf.\ (\ref{Aj-inverse})) is an isomorphism. 
\item
The following assertions hold. In particular, $B$ satisfies the assumptions in Theorem \ref{RiemannExtadic} (with $t=p$), Proposition \ref{pg-approx} and Theorem \ref{RiemannExt}. 
\begin{enumerate}
\item
$B$ is $p$-adically Zariskian and integral over a Noetherian ring. 
\item
$(p,g)$ is a regular sequence on $B$. 
\item
$B$ is completely integrally closed in $B[\frac{1}{pg}]$.
\item
$B$ is a $(pg)^{\frac{1}{p^\infty}}$-almost Witt-perfect ring. 
\end{enumerate}
\item
The $p$-adic completion $\widehat{\varprojlim_{j>0}B^{j\circ}}$ is integral $(pg)^{\frac{1}{p^\infty}}$-almost perfectoid.\footnote{Since $A=g^{-\frac{1}{p^\infty}}A$  (cf.\ the proof of (1))  and $A$ is $(pg)^{\frac{1}{p^\infty}}$-almost Witt-perfect by assumption, the $p$-adic completion $\widehat{g^{-\frac{1}{p^\infty}}A}$ is an integral $(pg)^{\frac{1}{p^\infty}}$-almost perfectoid ring.
In \cite[Question 3.5.1]{An1}, a question is raised as to whether $g^{-\frac{1}{p^\infty}}\widehat{A}$ is integral perfectoid in the case when $A$ is actually Witt-perfect.} 
\end{enumerate}
\end{proposition}

\begin{proof}[Proof of Proposition \ref{preprop}]
(1): Since $A$ is normal and integral over a Noetherian ring, it is completely integrally closed in $A[\frac{1}{g}]$ in view of Proposition \ref{propIC=CIC}. 
Hence $A=g^{-\frac{1}{p^\infty}}A$ by Corollary \ref{lem07031}. Thus the assertion follows. 

(2): By Proposition \ref{pg-approx}, $A^{j}$ is preuniform. Hence the assertion (a) follows from Proposition \ref{RiemannFinEt}. 
In view of Proposition \ref{rationaluniformity}, to deduce the assertion (b), it suffices to show that $(p^j, g)$ forms a regular sequence on $A$ and $B$. 
Since $A$ is $p$-torsion free and $p$-adically Zariskian, any prime number in $A$ is a regular element. 
Thus, if the generic characteristic of $R$ is positive, then the $R$-algebra $A$ and the $A$-algebra $B$ are the zero rings, where $(p^j, g)$ forms a regular sequence. 
If the generic characteristic of $R$ is $0$, then $R$ is N-2 by \cite[Theorem 4.6.10]{Fo17}, and hence the assertion follows from Lemma \ref{PrepForPerAbhy3}, as desired. 
The assertion (c) follows from the assertion (b) and \cite[Theorem 7.9]{Sch12}.  
Finally, let us prove the assertion (d). By assumption, $A[\frac{p^{j}}{g}]$ and $B[\frac{p^{j}}{g}]$ are integral over a Noetherian ring. Moreover, the ring map $A[\frac{p^{j}}{g}]\to B[\frac{p^{j}}{g}]$ is integral. 
Hence by Proposition \ref{propIC=CIC}, 
$$
B^{j\circ}=\biggl(B[\frac{p^{j}}{g}]\biggr)^{+}_{B^{j}}=\biggl(A[\frac{p^{j}}{g}]\biggr)^{+}_{B^{j}}=\biggl(\biggl(A[\frac{p^{j}}{g}]\biggr)^{+}_{A^{j}}\biggr)^{+}_{B^{j}}=(A^{j\circ})^{+}_{B^{j}}. 
$$
Therefore, the first assertion follows. 
Thus, the second assertion follows from the almost purity theorem for Witt-perfect rings \cite[Theorem 5.2]{DK14} or \cite[Theorem 2.9]{DK15} (see \cite[Theorem 5.9]{NS19} for a conceptual proof).

(3): Since $B$ is integral over $A$ and $p\in A$ is contained in the Jacobson radical, it is also contained in the Jacobson radical of $B$. Moreover, $B$ is integral over a Noetherian normal domain $R$, and integrally closed in $B[\frac{1}{pg}]$. Hence we can apply Theorem \ref{RiemannExtadic} to $B$ and obtain the assertion.

(4): 
The assertions (a) and (b) have already been proved above. Since $B$ is integrally closed in $B[\frac{1}{pg}]$, the assertion (c) follows from Proposition \ref{propIC=CIC}. 
Let us show the assertion (d). To prove the almost Witt-perfectness, it suffices to check that the condition (1) in Definition \ref{AlmostWittRing} is satisfied because $B$ contains $p^{\frac{1}{p}}$. By applying \cite[Proposition 4.4.1]{An1} (see Corollary \ref{vanishinglimone} for a self-contained proof), for any fixed $r=\frac{n}{p}$ with $n \in \mathbb{N}$, we get a $(pg)^{\frac{1}{p^\infty}}$-almost isomorphism:
\begin{equation}
\label{Witt-perfect2}
{\varprojlim_j}^1 \big(B^{j\circ}/(p^r)\big) \approx 0.
\end{equation}
After applying $\varprojlim$ to the standard short exact sequence $0 \to B^{j\circ}/(p^{\frac{p-1}{p}}) \to B^{j\circ}/(p) \to B^{j\circ}/(p^{\frac{1}{p}}) \to 0$, the following $(pg)^{\frac{1}{p^\infty}}$-almost surjection follows from $(\ref{Witt-perfect2})$: 
\begin{equation}
\label{Witt-perfect3}
\varprojlim_j \big(B^{j\circ}/(p)\big) \to \varprojlim_j \big(B^{j\circ}/(p^{\frac{1}{p}})\big).
\end{equation}
By Witt-perfectness of $B^{j\circ}$, the Frobenius isomorphism $B^{j\circ}/(p^{\frac{1}{p}}) \cong B^{j\circ}/(p)$ yields a short exact sequence: $0 \to B^{j\circ}/(p^{\frac{1}{p}}) \to B^{j\circ}/(p) \xrightarrow{Frob} B^{j\circ}/(p) \to 0$. Then again by $(\ref{Witt-perfect2})$, we get
\begin{equation}
\label{Witt-perfect4}
\varprojlim_j \big(B^{j\circ}/(p)\big) \xrightarrow{Frob} \varprojlim_j \big(B^{j\circ}/(p)\big)~\mbox{is}~(pg)^{\frac{1}{p^\infty}}\mbox{-almost surjective}. 
\end{equation}
Consider the commutative diagram
$$
\begin{CD}
\varprojlim_j \big(B^{j\circ}/(p)\big) @>Frob>> \varprojlim_j \big(B^{j\circ}/(p)\big) \\
@AAA @AAA \\
\big(\varprojlim_j B^{j\circ}\big)/(p) @>Frob>> \big(\varprojlim_j B^{j\circ}\big)/(p) .\\
\end{CD}
$$
It suffices to show that in view of $(\ref{Witt-perfect4})$ that
\begin{equation}
\label{Witt-perfectinjective}
\big(\varprojlim_j B^{j\circ}\big)/(p) \to \varprojlim_j \big(B^{j\circ}/(p)\big)~\mbox{is a}~(pg)^{\frac{1}{p^\infty}}\mbox{-almost isomorphism}.
\end{equation}
By taking the inverse limits over $j$ to the short exact sequence: $0 \to B^{j\circ} \xrightarrow{p} B^{j\circ} \to B^{j\circ}/(p) \to 0$, we see that the map in $(\ref{Witt-perfectinjective})$ is injective. On the other hand, the above map is $(pg)^{\frac{1}{p^\infty}}$-almost surjective by applying the almost surjectivity of $(\ref{Witt-perfect3})$ to \cite[Proposition 4.3.1 and Remarque 4.3.1]{An1}, which shows that the Frobenius endomorphism on $(\varprojlim_{j}{B^{j\circ}})/(p)$ is $(pg)^{\frac{1}{p^\infty}}$-almost surjective. 
Hence by the assertion (3), we obtain the desired consequence. 

(5): It follows from the assertions (3) and (4). 
\end{proof}

Let us complete the proof of Theorem \ref{PerAbhyankar}.

\begin{proof}[Proof of Theorem \ref{PerAbhyankar}]
The assertion (a) follows from the assertions (c) and (d) of Proposition \ref{preprop} (4). Let us prove the assertion (b). 
We fix the notation as in Proposition \ref{preprop}. 
Let us make a reduction by using Galois theory of rings. By decomposing $A$ into the direct product of rings, we may assume and do that $A[\frac{1}{pg}] \to B'$ is finite \'etale of constant rank (indeed, one can check the conditions $(1) \sim (4)$ remain to hold for each direct factor of the ring $A$). By Lemma \ref{AlmostG-GaloisExt4} applied to the finite \'etale extension $A[\frac{1}{pg}] \hookrightarrow B'=B[\frac{1}{pg}]$, there is the  decomposition
\begin{equation}
\label{Galoisext}
A[\frac{1}{pg}] \hookrightarrow  B'=B[\frac{1}{pg}] \hookrightarrow C',
\end{equation}
where $A[\frac{1}{pg}] \to C'$ and $B'=B[\frac{1}{pg}] \to C'$ are Galois coverings. Let $G$ be the Galois group for $A[\frac{1}{pg}] \to C'$ and let
$H \subset G$ be the Galois subgroup for $B' \to C'$. Notice that $G$ is finite. Let $C$ be the integral closure of $A$ in $C'$. 
Notice that one can apply Proposition \ref{preprop} to $A$ and $C$ by taking $B'=A[\frac{1}{pg}]$ and $B'=C'$, respectively (cf.\ Proposition \ref{preprop} (1)). 
We will use the consequences of this fact without explicit mention in what follows. 

In view of (\ref{Aj-inverse}) and Proposition \ref{preprop} (2)-(a), we obtain the following commutative diagram:
$$
\begin{CD}
A @>>> A^{j+1\circ} @>>> B^{j+1\circ} @>>> B' \\
@| @VVV @VVV @| \\
A @>>> A^{j\circ} @>>> B^{j\circ} @>>> B'.\\
\end{CD}
$$
Taking inverse limits, we have composite map of rings: 
\begin{equation}
\label{composition}
A \cong \widetilde{A}:=\varprojlim_j A^{j\circ} \to \widetilde{B}:=\varprojlim_j B^{j\circ} \to B',
\end{equation}
where the first isomorphism is due to Theorem \ref{RiemannExtadic}. Similarly, we obtain the compositions of ring maps:
\begin{equation}
\label{inverselimitC}
A \cong\widetilde{A} \to \widetilde{C}:=\varprojlim_j C^{j\circ} \to C'.
\end{equation} 
By Proposition \ref{preprop} (3), we find that (\ref{composition}) and (\ref{inverselimitC}) are isomorphic to the integral maps $A\to B$ and $A \to C$, respectively. Following the convention in Definition \ref{DefAj}, we set
$$
\mathcal{B}:=\widehat{B}[\frac{1}{p}]~\mbox{and}~\mathcal{C}:=\widehat{C}[\frac{1}{p}].
$$

Our goal is the following:

\begin{enumerate}
\item[$\bullet$]
Recall that $A[\frac{1}{pg}] \to C'$ is a $G$-Galois covering and $B'=B[\frac{1}{pg}] \to C'$ is a $H$-Galois covering. Fix an integer $m>0$. Then we will establish that $A/(p^m) \to C/(p^m)$ is a $(pg)^{\frac{1}{p^\infty}}$-almost $G$-Galois covering, and $B/(p^m) \to C/(p^m)$ is a $(pg)^{\frac{1}{p^\infty}}$-almost $H$-Galois covering. Using these facts combined with Proposition \ref{AlmostG-GaloisExt3} (3), we deduce that $A/(p^m) \to B/(p^m)$ is $(pg)^{\frac{1}{p^\infty}}$-almost finite \'etale. As we can treat $A \to C$ and $B \to C$ in a complete parallel manner in view of Proposition \ref{preprop}, we consider only the case $A \to C$ in what follows. We use the notation $A \to C$ and $\widetilde{A} \to \widetilde{C}$ interchangeably.
\end{enumerate}

As $A[\frac{1}{pg}] \to C'$ is a $G$-Galois covering, so is $\mathcal{A}^j\to \mathcal{C}^j$ in view of \cite[Lemma 12.2.7]{Fo17}. Let $\widehat{C^{j\circ}}$ be the $p$-adic completion of $C^{j\circ}$. 
Since $C^{j\circ}[\frac{1}{p}]=C'$, there is a natural $\mathcal{A}^j$-algebra map
\begin{equation}\label{NatRingMap1}
\mathcal{C}^j=C' \otimes_{A[\frac{1}{pg}]} \mathcal{A}^j \to \widehat{C^{j\circ}}[\frac{1}{p}]. 
\end{equation}
Since $\widehat{A^{j\circ}}\cong\mathcal{A}^{j\circ}$ by Proposition \ref{pg-approx} (4), the map (\ref{NatRingMap1}) is an isomorphism, which induces $\mathcal{C}^{j\circ} \cong \widehat{C^{j\circ}}$ in view of \cite[Corollary 4.10]{NS19}.
Thus, $G$ acts on $\widehat{C^{j\circ}}$ and
\begin{equation}\label{GaloisCoh2}
(\widehat{C^{j\circ}})^G \cong(\mathcal{C}^{j\circ})^G\cong\mathcal{A}^{j\circ}
\end{equation}
by applying Lemma \ref{GalPwrBdd} or Discussion \ref{Galoisvanishing} (1) below. In particular, $\mathcal{A}^{j\circ} \to \widehat{C^{j\circ}}$ is an integral extension. In summary,
\begin{equation}
\label{completionGalois}
\mathcal{A}^{j\circ} \to \mathcal{C}^{j\circ} \cong \widehat{C^{j\circ}}~\mbox{is}~(p)^{\frac{1}{p^\infty}}\mbox{-almost \'etale and}~\mathcal{A}^j \to \mathcal{C}^j \cong \widehat{C^{j\circ}}[\frac{1}{p}]~\mbox{is a}~G\mbox{-Galois covering}.
\end{equation}

To finish the proof, let us apply the proof of \cite[Proposition 5.2.3]{An1} via Galois theory of commutative rings to $(\ref{completionGalois})$. We refer the reader to \cite[(5.6), (5.7), (5.8), (5.9) and (5.10) of Proposition 5.2.3]{An1} for the following discussions. 
%We have
%$$
%\mathcal{C} \cong (\varprojlim_j \widehat{C^{j\circ}})[\frac{1}{p}]
%$$
%in view of Theorem \ref{RiemannExtadic} and Theorem \ref{RiemannExt}. In particular, we get isomorphisms $\mathcal{C}^\circ \cong \varprojlim_j \widehat{C^{j\circ}} \cong \varprojlim_j \mathcal{C}^{j\circ}$, where the first isomorphism follows from Lemma \ref{lem07032}, together with the fact that for any given perfectoid ring $\mathcal{D}$ with its ring of definition $\mathcal{D}_0$, $\mathcal{D}_0$ is completely integrally closed in $\mathcal{D}$ if and only if $\mathcal{D}_0=\mathcal{D}^\circ$.

After invoking the notation $(\ref{composition})$ and $(\ref{inverselimitC})$, there follows the following $(g)^{\frac{1}{p^\infty}}$-almost isomorphisms by applying Theorem \ref{RiemannExt} (b) to $A^j$ (resp. $C^j$): 
\begin{equation}
\label{inverselimitcompletion}
 \widehat{\widetilde{A}} \approx \mathcal{A}^\circ ~(\mbox{resp}.~
\widehat{\widetilde{C}} \approx \mathcal{C}^\circ).
\end{equation}
Indeed, this is checked by a chain of almost isomorphisms:
$$
\widehat{\widetilde{A}} \cong  \widehat{\varprojlim_j A^{j\circ}} \approx \varprojlim_j \widehat{A^{j\circ}} \cong g^{-\frac{1}{p^\infty}}\mathcal{A}^\circ \approx \mathcal{A}^\circ.
$$
Here, the first isomorphism is the $p$-adic completion of the isomorphism from
Theorem \ref{RiemannExtadic}, and the last second isomorphism is due to Riemann's extension theorem \cite[Th\'eor\`eme 4.2.2]{An1} (see Theorem \ref{Hebbarkeits1} for a self-contained proof). The same reasoning applies to deduce $\widehat{\widetilde{C}} \approx \mathcal{C}^\circ$.

%Hence
%\begin{equation}
%\label{GaloisExtension1}
%(\widehat{\widetilde{C}})^G \cong \big(\widehat{\varprojlim_j C^{j\circ}}\big)^G \cong \big(\varprojlim_j \widehat{C^{j\circ}}\big)^G \cong \varprojlim_j (\mathcal{C}^{j\circ})^G \cong \varprojlim_j \mathcal{A}^{j\circ} \cong \varprojlim_j \widehat{A^{j\circ}} \cong \widehat{\widetilde{A}},
%\end{equation}
%where the third isomorphism follows from the commutativity of inverse limits with taking $G$-invariants and $(\ref{completionGalois})$, and the fourth one from $(\ref{GaloisCoh2})$. The last one follows from Theorem \ref{RiemannExt}.

In view of $(\ref{completionGalois})$ and applying \cite[Proposition 3.3.4]{An1}, the ring map
\begin{equation}
\label{GaloisExtension2}
\mathcal{C}^{j\circ} \widehat{\otimes}_{\mathcal{A}^{j\circ}} \mathcal{C}^{j\circ} \to \prod_G \mathcal{C}^{j\circ}~\mbox{defined by}~b \otimes b' \mapsto \big(\gamma(b)b'\big)_{\gamma \in G}
\end{equation}
is a $(p)^{\frac{1}{p^\infty}}$-almost isomorphism, where the completed tensor product is $p$-adic. By \cite[Proposition 4.4.4]{An1}, we have $\mathcal{C}\{\frac{p^j}{g}\} \cong \mathcal{C}^j$ and $\mathcal{C}$ is an $\mathcal{A}$-algebra. Using this, we obtain
$$
\big(\mathcal{C} \widehat{\otimes}_{\mathcal{A}} \mathcal{C}\big)\{\frac{p^j}{g}\} \cong \mathcal{C} \widehat{\otimes}_{\mathcal{A}} \mathcal{C} \widehat{\otimes}_{\mathcal{A}} \mathcal{A}^j \cong \big(\mathcal{C} \widehat{\otimes}_{\mathcal{A}} \mathcal{A}^j\big) \otimes_{\mathcal{A}^j} \big(\mathcal{C} \widehat{\otimes}_{\mathcal{A}} \mathcal{A}^j\big) 
\cong \mathcal{C}\{\frac{p^j}{g}\} \otimes_{\mathcal{A}^j} \mathcal{C}\{\frac{p^j}{g}\}\cong \mathcal{C}^j \otimes_{\mathcal{A}^j}\mathcal{C}^j.
$$
By Riemann's extension theorem \cite[Th\'eor\`eme 4.2.2]{An1}  (see also Theorem \ref{Hebbarkeits1}) and by \cite[Proposition 3.3.4]{An1},  we have $(pg)^{\frac{1}{p^\infty}}$-almost isomorphisms:
\begin{equation}
\label{GaloisExtension3}
\varprojlim_j \big(\mathcal{C}^{j\circ} \widehat{\otimes}_{\mathcal{A}^{j\circ}}\mathcal{C}^{j\circ}\big) \approx
\varprojlim_j \big(\mathcal{C}^j \otimes_{\mathcal{A}^j}\mathcal{C}^j\big)^{\circ} \cong \varprojlim_j \big(\mathcal{C} \widehat{\otimes}_{\mathcal{A}} \mathcal{C}\big)\{\frac{p^j}{g}\}^{\circ} \approx \big(\mathcal{C} \widehat{\otimes}_{\mathcal{A}}\mathcal{C}\big)^\circ \approx \mathcal{C}^\circ \widehat{\otimes}_{\mathcal{A}^\circ} \mathcal{C}^\circ.
\end{equation}
Putting $(\ref{GaloisExtension2})$ and $(\ref{GaloisExtension3})$ together, we obtain the following $(pg)^{\frac{1}{p^\infty}}$-almost isomorphism:
\begin{equation}
\label{GaloisExtension4}
\mathcal{C}^\circ \widehat{\otimes}_{\mathcal{A}^\circ} \mathcal{C}^\circ\approx \prod_G \mathcal{C}^\circ.
\end{equation}

By Discussion \ref{Galoisvanishing} (2), we know that $\mathcal{A}^\circ/(p^m) \to \big(\mathcal{C}^\circ/(p^m)\big)^G$ is a $(pg)^{\frac{1}{p^\infty}}$-almost isomorphism for any $m>0$. So this fact combined with the $(g)^{\frac{1}{p^\infty}}$-almost isomorphisms $(\ref{inverselimitcompletion})$ and $(\ref{GaloisExtension4})$ modulo $p^m$ yields that the induced map: $\widetilde{A}/(p^m) \to \widetilde{C}/(p^m)$ is a $(pg)^{\frac{1}{p^\infty}}$-almost $G$-Galois covering. This map factors as $\widetilde{A}/(p^m) \to \widetilde{B}/(p^m) \to  \widetilde{C}/(p^m)$. It then follows from Proposition \ref{AlmostG-GaloisExt3} (3) that $\widetilde{A}/(p^m) \to \widetilde{B}/(p^m)$ is $(pg)^{\frac{1}{p^\infty}}$-almost finite \'etale, as desired. This completes the proof of the theorem.
\end{proof}

\begin{Discussion}
\label{Galoisvanishing}
\begin{enumerate}
\item
Here is a way to check the isomorphism: $(\widehat{C^{j\circ}})^G \cong \mathcal{A}^{j\circ}$ that appears in $(\ref{GaloisCoh2})$. Since inverse limits commutes with taking $G$-invariants and $\widehat{A^{j\circ}} \cong \mathcal{A}^{j\circ}$ by Proposition \ref{pg-approx}, we have
\begin{equation}
\label{GaloisCoh}
(\widehat{C^{j\circ}})^G \cong \big(\varprojlim_m C^{j\circ}/(p^m)\big)^G \cong \varprojlim_m \big(C^{j\circ}/(p^m)\big)^G \approx \varprojlim_m \big((C^{j\circ})^G/(p^m)\big) \cong \varprojlim_m A^{j\circ}/(p^m) \cong \mathcal{A}^{j\circ},
\end{equation}
where $\approx$ in the middle denotes a $(p)^{\frac{1}{p^\infty}}$-almost isomorphism and we reason this as follows: Consider the short exact sequence $0 \to C^{j\circ} \xrightarrow{p^m} C^{j\circ} \to C^{j\circ}/(p^m) \to 0$. Applying the Galois cohomology $H^i(G,~)$ to this exact sequence, we get an injection $(C^{j\circ})^G/(p^m) \hookrightarrow \big(C^{j\circ}/(p^m)\big)^G$ whose cokernel embeds into $H^1(G,C^{j\circ})$. By applying \cite[Theorem 2.4]{Fa88} or \cite[Proposition 3.4]{Ol09}, $H^1(G,C^{j\circ})$ is $(p)^{\frac{1}{p^\infty}}$-almost zero. Hence $(\ref{GaloisCoh})$ is proved. $\widehat{C^{j\circ}}$ is completely integrally closed in $\widehat{C^{j\circ}}[\frac{1}{p}]$ by Lemma \ref{p-adicnormal}. Then we have $\widehat{C^{j\circ}} \cong p^{-\frac{1}{p^\infty}}(\widehat{C^{j\circ}})$ and $p^{-\frac{1}{p^\infty}}(\mathcal{A}^{j\circ}) \cong \mathcal{A}^{j\circ}$ by Lemma \ref{lem07031}. Since the functor $p^{-\frac{1}{p^\infty}}(~)$ commutes with taking $G$-invariants, $(\ref{GaloisCoh})$ yields an (honest) isomorphism:
$$
(\widehat{C^{j\circ}})^G \cong \big(p^{-\frac{1}{p^\infty}}(\widehat{C^{j\circ}})\big)^G \cong p^{-\frac{1}{p^\infty}}\big((\widehat{C^{j\circ}})^G\big) \cong
p^{-\frac{1}{p^\infty}}(\mathcal{A}^{j\circ}) \cong \mathcal{A}^{j\circ},
$$
which proves $(\ref{GaloisCoh2})$.

\item
Using $(\ref{GaloisCoh2})$, let us prove that the map
$$
\mathcal{A}^{\circ}/(p^m) \to\big(\mathcal{C}^{\circ}/(p^m)\big)^G
$$
is a $(pg)^{\frac{1}{p^\infty}}$-almost isomorphism for any integer $m>0$. We have already seen the $(pg)^{\frac{1}{p^\infty}}$-almost isomorphisms: $\mathcal{A}^{j\circ}/(p^m) \approx (\mathcal{C}^{j\circ})^G/(p^m) \approx \big(\mathcal{C}^{j\circ}/(p^m)\big)^G$. Taking the inverse limits $j \to \infty$ and using \cite[Proposition 4.2.1]{An1} or the $(g)^{\frac{1}{p^\infty}}$-almost isomorphism $(\ref{inverselimitiso})$, we get $(pg)^{\frac{1}{p^\infty}}$-almost isomorphisms:
$$
\mathcal{A}^\circ/(p^m) \approx \varprojlim_{j>0} \big(\mathcal{C}^{j\circ}/(p^m)\big)^G \cong \big(\varprojlim_{j>0} \mathcal{C}^{j\circ}/(p^m)\big)^G \approx \big(\mathcal{C}^\circ/(p^m)\big)^G,
$$
as wanted.

\end{enumerate}
\end{Discussion}

\begin{Problem}
Does Theorem \ref{PerAbhyankar} hold true under the more general assumption that $A$ is not necessarily integral over a Noetherian ring?
\end{Problem}

This problem is related to a possible generalization of Riemann's extension theorem (see Theorem \ref{RiemannExtadic} and Theorem \ref{RiemannExt}) for Witt-perfect rings of general type.

\section{Applications of Witt-perfect Abhyankar's lemma}\label{SecAppWPA}

\subsection{A construction of almost Cohen-Macaulay algebras}

Before proving the main theorem for this section, we recall the definition of big Cohen-Macaulay algebras, due to Hochster.

\begin{definition}[Big Cohen-Macaulay algebra]
\label{BigMac}
Let $(R,\fm)$ be a Noetherian local ring of dimension $d>0$ and let $T$ be an $R$-algebra. Then $T$ is a \textit{big Cohen-Macaulay $R$-algebra}, if there is a system of parameters $x_1,\ldots,x_d$ such that $x_1,\ldots,x_d$ is a regular sequence on $T$ and $(x_1,\ldots,x_d)T \ne T$. Moreover, we say that a big Cohen-Macaulay algebra is \textit{balanced}, if every system of parameters satisfies the above conditions.
\end{definition}

We also recall the definition of almost Cohen-Macaulay algebras from \cite[Definition 4.1.1]{An2}. Refer the reader to \cite[Proposition 2.5.1]{An181} for a subtle point on this definition.

\begin{definition}[Almost Cohen-Macaulay algebra]
\label{DefAlmost}
Let $(R,\fm)$ be a Noetherian local ring of dimension $d>0$, and let $(T,I)$ be a basic setup equipped with an $R$-algebra structure. Fix a system of parameters $x_1,\ldots,x_d$. We say that $T$ is \textit{$I$-almost Cohen-Macaulay with respect to $x_1,\ldots,x_d$}, if $T/\fm T$ is not $I$-almost zero and
$$
c \cdot \big((x_1,\ldots,x_i):_T x_{i+1}\big) \subset (x_1,\ldots,x_i)T
$$
for any $c \in I$ and $i=0,\ldots,d-1$.
\end{definition}

It is important to keep in mind that the permutation of the sequence $x_1,\ldots,x_d$ in the above definition may fail to form an almost regular sequence. We consider the sequence $p,x_2,\ldots,x_d$ for the main theorem below.
\\

$\textbf{Andr\'e's construction}$: For the applications given below, we take $I$ to be the ideal $\bigcup_{n>0} \varpi^{\frac{1}{p^n}}T$ as the basic setup $(T,I)$ for some regular element $\varpi \in R$. Following \cite{An2}, we introduce some auxiliary algebras. Let $W(k)$ be the ring of Witt vectors for a perfect field $k$ of characteristic $p>0$ and let
$$
A:=W(k)[[x_2,\ldots,x_d]]
$$
be an unramified complete regular local ring and $V_j:=W(k)[p^{\frac{1}{p^j}}]$. Then $V_j$ is a complete discrete valuation ring and set $V_{\infty}:=\varinjlim_j V_j$. Then this is a Witt-perfect valuation domain. For a fixed element $0 \ne g \in A$, we set
$$
B_{jk}:=V_j[[x^{\frac{1}{p^j}}_2,\ldots,x^{\frac{1}{p^j}}_d]][g^{\frac{1}{p^k}}][\frac{1}{p}]:=\Big(V_j[[x^{\frac{1}{p^j}}_2,\ldots,x^{\frac{1}{p^j}}_d]][T]/(T^{p^k}-g)\Big)[\frac{1}{p}]
$$
for any pair of non-negative integers $(j,k)$. For any pairs $(j,k)$ and $(j',k')$ with $j \le j'$ and $k \le k'$, we can define the natural map $B_{jk} \to B_{j'k'}$. Let us define the $A$-algebra $A_{jk}$ to be the integral closure of $A$ in $B_{jk}$. Let us also define
\begin{equation}
\label{bigalgebra1}
A_{\infty\infty}:=\varinjlim_{j,k} A_{jk}~\mbox{and}~A_{\infty g}:=\mbox{the integral closure of}~A_{\infty\infty}~\mbox{in}~A_{\infty\infty}[\frac{1}{pg}].
\end{equation}
For brevity, let us write
\begin{equation}
\label{bigalgebra2}
A_{\infty}:=A_{\infty0}:=\varinjlim_jV_j[[x^{\frac{1}{p^j}}_2,\ldots,x^{\frac{1}{p^j}}_d]].
\end{equation}
Then we have a tower of integral ring maps:
$$
A \to A_{\infty} \to A_{\infty\infty} \to A_{\infty g}.
$$

\begin{lemma}
\label{separatedlemma}
Let $R$ be a Noetherian domain with a proper ideal $I$ and let $T$ be a normal ring that is a torsion free integral extension of $R$. Assume that  $\varpi \in I$ is a nonzero element such that $T$ admits a compatible system of $p$-power roots
$\varpi^{\frac{1}{p^n}}$. Then $T/IT$ is not $(\varpi^{\frac{1}{p^\infty}})$-almost zero.
\end{lemma}

\begin{proof}
In order to prove that $T/IT$ is not $(\varpi)^{\frac{1}{p^\infty}}$-almost zero, it suffices to prove that $T_{\fm}/IT_{\fm}$ is not $(\varpi)^{\frac{1}{p^\infty}}$-almost zero, where $\fm$ is any maximal ideal of $T$ containing $IT$, since $T_{\fm}/IT_{\fm}$ is the localization of $T/IT$.
Then $T_{\fm}$ is a normal domain that is an integral extension over the Noetherian domain $R_{\fm \cap R}$, in which $I$ is a proper ideal. To derive a contradiction, we suppose that $T_\fm/I T_\fm$ is $(\varpi^{\frac{1}{p^\infty}})$-almost zero. Notice that $T_\fm$ is contained in the absolute integral closure $(R_{\fm \cap R})^+$. In particular, it implies that
$$
(\varpi)^{\frac{1}{p^n}} \in  IT_\fm~\mbox{for all}~n>0.
$$
Raising $p^n$-th power on both sides, we get by \cite[Lemma 4.2]{Sh11};
$$
\varpi \in \bigcap_{n>0} I^{p^n}T_\fm=0,
$$
which is a contradiction.
\end{proof}

The big rings $A_{\infty\infty}$ and $A_{\infty g}$ enjoy the following desirable properties.

\begin{proposition}
\label{AlmostPerf}
Let the notation be as in $(\ref{bigalgebra1})$ and $(\ref{bigalgebra2})$. Then the following assertions hold:
\begin{enumerate}
\item
$A_{\infty}$ is completely integrally closed in its field of fractions. It is an integral and faithfully flat extension over $A$. Moreover, the localization map $A_{\infty}[\frac{1}{pg}] \to A_{\infty\infty}[\frac{1}{pg}]$ is ind-\'etale. Finally, $A_{\infty\infty}$ is a $(p)^{\frac{1}{p^\infty}}$-almost Cohen-Macaulay and Witt-perfect algebra.

\item
$A_{\infty g}$ is a $(g)^{\frac{1}{p^\infty}}$-almost Witt-perfect algebra over the Witt-perfect valuation domain $V_{\infty}$ such that $p^{\frac{1}{p^n}} \in V_{\infty}$ and $g^{\frac{1}{p^n}} \in A_{\infty g}$. Moreover, $A_{\infty g}$ is a $(pg)^{\frac{1}{p^\infty}}$-almost Cohen-Macaulay normal ring that is completely integrally closed in $A_{\infty g}[\frac{1}{pg}]$. In particular, the localization of $A_{\infty g}$ at any maximal ideal is a $(pg)^{\frac{1}{p^\infty}}$-almost Cohen-Macaulay normal domain.
\end{enumerate}
\end{proposition}

\begin{proof}
$(1)$: It is clear that $A \to A_{\infty}$ is integral by construction. Since $A_{\infty}$ is a filtered colimit of regular local subrings with module-finite transition maps, one readily checks that the transition map is flat and thus, $A \to A_{\infty}$ is faithfully flat. By Lemma \ref{completenormal}, $A_{\infty}$ is a completely integrally closed domain in its field of fractions. By looking at the discriminant, it is easy to check that $A_{\infty}[\frac{1}{pg}] \to A_{\infty\infty}[\frac{1}{pg}]$ is ind-\'etale. We claim that $A_{\infty\infty}$ is a $(p)^{\frac{1}{p^\infty}}$-almost Cohen-Macaulay and Witt-perfect algebra. To show that it is Witt-perfect, it suffices to show that the $p$-adic completion $\widehat{A_{\infty\infty}}$ is integral perfectoid. But this follows from the description appearing in \cite[Lemme 2.5.1]{An2} and a comment in \cite[Definition 2.2]{Bh18} together with an application of \cite[Corollary 3.6]{NS19} (where we should set $g=1$). In order to show that $A_{\infty\infty}$ is $(p)^{\frac{1}{p^\infty}}$-almost Cohen-Macaulay, since $A_{\infty\infty}$ is $p$-torsion free, it suffices to show that $x_2,\ldots,x_d$ is a $(p)^{\frac{1}{p^\infty}}$-almost regular sequence on $A_{\infty\infty}/(p)$ in view of Lemma \ref{separatedlemma}. By Andr\'e's crucial result \cite[Th\'eor\`eme 2.5.2]{An2},\footnote{This is known as \textit{Andr\'e's Flatness Lemma}. A similar construction also appears in \cite[Theorem 16.9.17]{GR18}, where they apply $p$-integral closure instead of integral closure. This makes it possible to get rid of ``$(p)^{\frac{1}{p^\infty}}$-almost" from the statement.} the induced map $A_\infty/(p) \to A_{\infty\infty}/(p)$ is $(p)^{\frac{1}{p^\infty}}$-almost faithfully flat. Hence the desired claim follows from this fact together with the fact that $p,x_2,\ldots,x_d$ is a regular sequence on $A_\infty$.

$(2)$: By the assertion $(1)$, $(p,g)$ is a $(p)^{\frac{1}{p^\infty}}$-almost regular sequence on $A_{\infty \infty}$. Next we study $A_{\infty g}$ and consider $\widetilde{A_{\infty\infty}}:=\varprojlim_j A^{j\circ}_{\infty\infty}$ attached to $A_{\infty\infty}$ as defined in $(\ref{Aj-inverse})$ (see also Theorem \ref{RiemannExt}). Then we claim that
\begin{equation}
\label{RiemannEq}
A_{\infty g} \cong \widetilde{A_{\infty\infty}}.
\end{equation}
Since $A_{\infty g}$ is integrally closed in $A_{\infty\infty}[\frac{1}{pg}]=A_{\infty g}[\frac{1}{pg}]$, it follows from Proposition \ref{propIC=CIC} that $A_{\infty g}$ is completely integrally closed in $A_{\infty\infty}[\frac{1}{pg}]$. Now by applying Theorem \ref{RiemannExtadic} to $A_{\infty\infty}$, the isomorphism $(\ref{RiemannEq})$ follows. By the construction $(\ref{bigalgebra1})$, $A_{\infty\infty}$ is integrally closed in $A_{\infty\infty}[\frac{1}{p}]$. It follows that $A_{\infty\infty}$ is completely integrally closed in $A_{\infty\infty}[\frac{1}{p}]$. Then Lemma \ref{p-adicnormal} applies to show that the $p$-adic completion $\widehat{A_{\infty\infty}}$ is completely integrally closed in $\widehat{A_{\infty\infty}}[\frac{1}{p}]$. Now we can apply Riemann's extension theorem \cite[Th\'eor\`eme 4.2.2]{An1} (see Theorem \ref{Hebbarkeits1} for a self-contained proof), together with $(\ref{RiemannEq})$, to get that
$$
g^{-\frac{1}{p^\infty}}\widehat{A_{\infty \infty}}  \cong \varprojlim_j \widehat{A^{j\circ}_{\infty\infty}} \approx \widehat{\varprojlim_j A^{j\circ}_{\infty\infty}} \cong \widehat{A_{\infty g}},
$$
where the middle map is a $(g)^{\frac{1}{p^\infty}}$-almost isomorphism due to Theorem \ref{RiemannExt}. In particular, $\widehat{A_{\infty \infty}} \to \widehat{A_{\infty g}}$ is a $(g)^{\frac{1}{p^\infty}}$-almost isomorphism. 

From the property of $A_{\infty \infty}$ mentioned in $(1)$, one finds that $\widehat{A_{\infty g}}$ is an integral $(g)^{\frac{1}{p^\infty}}$-almost perfectoid and $(pg)^{\frac{1}{p^\infty}}$-almost Cohen-Macaulay algebra. By the fact that $A_{\infty}[\frac{1}{pg}]$ is a normal domain and $A_{\infty}[\frac{1}{pg}] \to A_{\infty\infty}[\frac{1}{pg}]$ is obtained as a filtered colimit of finite \'etale $A_{\infty}[\frac{1}{pg}]$-algebras, we see that $A_{\infty\infty}[\frac{1}{pg}]$ is a normal ring;  the localization at any maximal ideal is an integrally closed domain by Lemma \ref{normallocaldomain}. Since $A_{\infty g}$ is integrally closed in $A_{\infty\infty}[\frac{1}{pg}]$, it follows that $A_{\infty g}$ is also normal.
\end{proof}

As a corollary, we obtain the following theorem.

\begin{theorem}
\label{AlmostCMalg}
Let $(R,\fm)$ be a complete Noetherian local domain of mixed characteristic $p>0$ with perfect residue field $k$. Let $p,x_2,\ldots,x_d$ be a system of parameters and let $R^+$ be the absolute integral closure of $R$. Then there exists an $R$-algebra $T$ together with a nonzero element $g \in R$ such that the following hold:
\begin{enumerate}
\item
$T$ admits compatible systems of $p$-power roots $p^{\frac{1}{p^n}}, g^{\frac{1}{p^n}} \in T$ for all $n>0$.

\item
The Frobenius endomorphism $Frob:T/(p) \to T/(p)$ is surjective.

\item
$T$ is a $(pg)^{\frac{1}{p^\infty}}$-almost Cohen-Macaulay normal domain with respect to $p,x_2,\ldots,x_d$ and $R \subset T \subset R^+$.

\item
The $p$-adic completion $\widehat{T}$ is integral perfectoid.

\item
$R[\frac{1}{pg}] \to T[\frac{1}{pg}]$ is an ind-\'etale extension. In other words, $T[\frac{1}{pg}]$ is a filtered colimit of finite \'etale $R[\frac{1}{pg}]$-algebras contained in $T[\frac{1}{pg}]$.
\end{enumerate}
\end{theorem}

\begin{proof}
In the following, we may assume $\dim R \ge 2$ without loss of generality. By Cohen's structure theorem, there is a module-finite extension 
$$
A:=W(k)[[x_2,\ldots,x_d]] \hookrightarrow R.
$$
As the induced field extension $\Frac(A) \to \Frac(R)$ is separable, there is an element $g \in A \setminus pA$ such that $A[\frac{1}{pg}] \to R[\frac{1}{pg}]$ is \'etale. As in Proposition \ref{AlmostPerf}, we set
$$
A_{\infty}:=\bigcup_{n>0} W(k)[p^{\frac{1}{p^n}}][[x_2^{\frac{1}{p^n}},\ldots,x_d^{\frac{1}{p^n}}]].
$$
Now consider the integral extensions $A \to A_{\infty} \to A_{\infty\infty} \to A_{\infty g}$ as in Proposition \ref{AlmostPerf}. Let $\fn$ be a maximal ideal of $A_{\infty g}$. Then the localization $(A_{\infty g})_{\fn}$ is a normal domain that is an integral extension over $A$ and enjoys the same properties as $A_{\infty g}$. Since $(p,g)$ forms part of a system of parameters of $A$ and $(A_{\infty g})_{\fn}$ is a filtered colimit of module-finite normal $A$-algebras, it follows that $(p,g)$ is a regular sequence on $(A_{\infty g})_{\fn}$ by Serre's normality criterion.\footnote{In what follows, if necessary, we repeat the same argument for deriving the regularity of $(p,g)$ in order to apply Theorem \ref{PerAbhyankar}.} By base change, the map
\begin{equation}
\label{colimit}
(A_{\infty g})_{\fn}[\frac{1}{pg}] \to R \otimes_A (A_{\infty g})_{\fn}[\frac{1}{pg}]
\end{equation}
is finite \'etale. It follows from \cite[Tag 033C]{Stacks} together with the normality of $(A_{\infty g})_{\fn}[\frac{1}{pg}]$ that $R \otimes_A (A_{\infty g})_{\fn}[\frac{1}{pg}]$ is a normal ring. Letting the notation be as in $(\ref{colimit})$, set
$$
B:=\mbox{the integral closure of}~R~\mbox{in}~R \otimes_A (A_{\infty g})_{\fn}[\frac{1}{pg}].
$$
Let $K$ be the field of fractions of $(A_{\infty g})_{\fn}[\frac{1}{pg}]$ and let $K \to L$ be the corresponding base change of $(\ref{colimit})$. Then the finite \'etaleness of $(\ref{colimit})$ implies that $L$ is the total ring of fractions of $R \otimes_A (A_{\infty g})_{\fn}[\frac{1}{pg}]$, and $L$ is a finite direct product of fields that are finite separable over $K$. Since $R \otimes_A (A_{\infty g})_{\fn}[\frac{1}{pg}]$ is normal, it is integrally closed in $L$. Then $B$ is identified with the integral closure of $R$ in $L$. By Lemma \ref{normallocaldomain}, it follows that $B$ is a normal ring that fits into the commutative diagram:
$$
\begin{CD}
(A_{\infty g})_{\fn} @>>> B \\
@AAA @AAA \\ 
A @>>> R \\
\end{CD}
$$
in which every map is injective and integral. Let $\fn'$ be any maximal ideal of $B$. Since $A$ is a local domain and $A \to B$ is a torsion free integral extension, one finds that $A \cap \fn'$ is the unique maximal ideal of $A$ and the induced localization map $A \to B_{\fn'}$ is an injective integral extension between normal domains. By setting $A:=(A_{\infty g})_{\fn}$ in the notation of Theorem \ref{PerAbhyankar} and applying Lemma \ref{separatedlemma}, it follows that $B$ is a $(pg)^{\frac{1}{p^\infty}}$-almost Cohen-Macaulay normal ring with respect to $p,x_2,\ldots,x_d$ and $(pg)^{\frac{1}{p^\infty}}$-almost Witt-perfect. Since these properties are preserved under localization with respect to any maximal ideal, it follows that the normal domain $B_{\fn'}$ enjoys the same properties. 

To finish the proof, let us put $C:=B_{\fn'}$ for brevity of notation. Set
$$
T:=\mbox{the integral closure of}~C~\mbox{in}~C[\frac{1}{p}]^{\rm{\acute{e}t}},
$$
where $C[\frac{1}{p}]^{\rm{\acute{e}t}}$ is the maximal \'etale extension of $C[\frac{1}{p}]$ contained in the absolute integral closure $C[\frac{1}{p}]^+$. By using a direct limit argument combined with \cite[Tag 033C]{Stacks} and the fact that $C[\frac{1}{p}]^{\rm{\acute{e}t}}$ is a filtered colimit of finite \'etale $C[\frac{1}{p}]$-algebras,
we find that the normality of $C$ implies the normality of $C[\frac{1}{p}]^{\rm{\acute{e}t}}$. Since $T$ is integrally closed in $C[\frac{1}{p}]^{\rm{\acute{e}t}}$, $T$ is a normal domain. Thus, $T$ is a Witt-perfect normal domain in view of \cite[Lemma 5.1]{Sh18} or \cite[Lemma 10.1]{Sh19}. Using this, it can be seen that the $p$-adic completion $\widehat{T}$ is integral perfectoid. In other words, the Frobenius map on $\widehat{T}/(p)$ induces an isomorphism: $\widehat{T}/(p^\frac{1}{p}) \cong \widehat{T}/(p)$. Therefore, it remains to establish that $T$ is $(pg)^{\frac{1}{p^\infty}}$-almost Cohen-Macaulay with respect to $p,x_2,\ldots,x_d$. Let us note that the composite map
$$
(A_{\infty g})_{\fn}[\frac{1}{pg}] \to C[\frac{1}{pg}] \to T[\frac{1}{pg}]
$$
satisfies that $T[\frac{1}{pg}]$ is the filtered colimit of finite \'etale $(A_{\infty g})_{\fn}[\frac{1}{pg}]$-algebras. As $T$ is integrally closed in its field of fractions, the integral closure of $(A_{\infty g})_{\fn}$ in $T[\frac{1}{pg}]$ is the same as $T$. In the proof of Proposition \ref{AlmostPerf}, we showed that $\widehat{A}_{\infty \infty} \to \widehat{A}_{\infty g}$ is a $(g)^{\frac{1}{p^\infty}}$-almost isomorphism. Thus, $A_{\infty \infty}/(p) \to A_{\infty g}/(p)$ is a $(g)^{\frac{1}{p^\infty}}$-almost isomorphism. Summing up, we conclude from Theorem \ref{PerAbhyankar} applied to $A:=(A_{\infty g})_{\fn}$ that $T/(p)$ is the filtered colimit of $(pg)^{\frac{1}{p^\infty}}$-almost finite \'etale $A_{\infty\infty}/(p)$-algebras. By Lemma \ref{separatedlemma} and Proposition \ref{AlmostPerf}, $T$ is $(pg)^{\frac{1}{p^\infty}}$-almost Cohen-Macaulay.
\end{proof}

As a corollary, we obtain the following result, which is the strengthened version of the main results in \cite{HM18}. The proof uses standard results from the theory of local cohomology. For a Noetherian local ring $(R,\fm)$, let $H^i_\fm(M)$ be the $i$-th local cohomology module of an $R$-module $M$ with support at the maximal ideal $\fm$ of $R$.

\begin{corollary}
Let the notation and hypotheses be as in Theorem \ref{AlmostCMalg}. Then the local cohomology modules $H^i_{\fm}(T)$ are $(pg)^{\frac{1}{p^\infty}}$-almost zero in the range $0 \le i  \le \dim R-1$. In particular, the image of the map $H_\fm^i(T) \to H_\fm^i(R^+)$ induced by $T \to R^+$ is $(pg)^{\frac{1}{p^\infty}}$-almost zero.
\end{corollary}

\begin{proof}
Letting $p,x_2,\ldots,x_d$ be a system of parameters of $R$, if one inspects the structure of the proof of Theorem \ref{AlmostCMalg} and Theorem \ref{PerAbhyankar},  it follows that $x_2^m,\ldots,x_d^m$ forms a $(pg)^{\frac{1}{p^\infty}}$-almost regular sequence on $T/(p^m)$ for all integers $m > 0$. As in the proof of \cite[Theorem 3.17]{HM18}, the Koszul cohomology modules $H^i(p^m,x_2^m,\ldots,x_d^m;T)$ and hence $H_\fm^i(T)$ are $(pg)^{\frac{1}{p^\infty}}$-almost zero for $i<\dim R$.  
\end{proof}

It is reasonable to study the following problem, which we credit to Heitmann in the $3$-dimensional case thanks to his proof of the direct summand conjecture; see \cite{H05}.

\begin{Problem}
Let $(R,\fm)$ be a complete Noetherian local domain of arbitrary characteristic with its absolute integral closure $R^+$ and the unique maximal ideal $\fm_{R^+}$. Fix a system of parameters $x_1,\ldots,x_d$ of $R$. Then does it hold true that
$$
c \cdot \big((x_1,\ldots,x_i):_{R^+} x_{i+1}\big) \subset (x_1,\ldots,x_i)R^+
$$
for any $c \in \fm_{R^+}$ and $i=0,\ldots,d-1$?
\end{Problem}

Bhatt gave an even stronger answer to the above problem in mixed characteristic in \cite{Bh20} by taking $x_1=p^n$, using prismatic cohomology and mod-$p^n$ Riemann-Hilbert correspondence. Namely, $p,x_2,\ldots,x_d$ is a regular sequence on $R^+$. In the equal prime characteristic case, Hochster and Huneke already gave a complete answer in \cite{HH92}. However, almost nothing is known in the equal characteristic zero case. Even in the mixed characteristic case, the above problem is not known to hold true if one starts with an \textit{arbitrary} system of parameters.

\begin{Problem}
Let $T$ be a big Cohen-Macaulay algebra over a Noetherian local domain $(R,\fm)$.
\begin{enumerate}
\item[$\bullet$]
Assume that $R$ has mixed characteristic. Then does $T$ map to an integral perfectoid big Cohen-Macaulay $R$-algebra? 

\item[$\bullet$]
Assume that $R$ has an arbitrary characteristic. Then does $R$ (or $T$) map to a coherent big Cohen-Macaulay $R$-algebra?
\end{enumerate}
\end{Problem}

For the coherence of absolutely integrally closed domains, see \cite{P22}. Here we mention a few related results.

\begin{proposition}
\label{freealgebra}
Assume that $T$ is a big Cohen-Macaulay algebra over a Noetherian local domain $(R,\fm)$ of any characteristic. Then $T$ maps to an $R$-algebra $B$ such that the following hold:
\begin{enumerate}
\item
$B$ is free over $T$. In particular, $B$ is a big Cohen-Macaulay $R$-algebra.

\item
$B$ is absolutely integrally closed. In other words, every nonzero monic polynomial in $B[X]$ has a root in $B$.
\end{enumerate}
\end{proposition}

\begin{proof}
Just apply \cite[Tag 0DCR]{Stacks}.
\end{proof}

In relation to Proposition \ref{freealgebra} and some observations on $p$-integral closure as discussed in \cite{CKS21}, we prove the following fact, which shows that flatness can be destroyed under taking $p$-integral closure. We refer the reader to \cite[2.1.7]{CKS21} for details on $p$-integral closure.

\begin{proposition}
Let $(R,\fm)$ be a non-regular local domain of mixed characteristic $p>0$. Then there exists a faithfully flat $R$-algebra $T$ such that $p^{\frac{1}{p^n}} \in T$ for $n>0$ and the Frobenius map induces a surjection $T/(p^{\frac{1}{p}}) \to T/(p)$. Moreover, let $T$ be any $R$-algebra with the aforementioned properties, and let $\widetilde{T}$ be the $p$-integral closure of $T$ in $T[\frac{1}{p}]$. Then the $p$-adic completion of $\widetilde{T}$ is integral perfectoid, 
but $\widetilde{T}$ is never flat over $R$.
\end{proposition}

\begin{proof}
The first assertion is due to Proposition \ref{freealgebra}. It follows from \cite[Proposition 2.1.8]{CKS21} that the $p$-adic completion $\widehat{\widetilde{T}}$ is integral perfectoid. Assume that $\widetilde{T}$ is flat over $R$. Since $T \to \widetilde{T}$ is an integral extension, it follows that $\widetilde{T}$ is faithfully flat over $R$. Moreover, $\widehat{\widetilde{T}}$ is faithfully flat over $R$ by \cite[Theorem 0.1]{Ye18}. But the main result of \cite{BIM18} forces $R$ to be regular, which contradicts the hypothesis that $R$ is not regular.
\end{proof}

\begin{theorem}[Gabber-Ramero]
Let $(R,\fm)$ be a complete local domain of mixed characteristic. Then any integral perfectoid big Cohen-Macaulay $R$-algebra $B$ admits an $R$-algebra map $B \to C$ such that $C$ is an integral perfectoid big Cohen-Macaulay $R$-algebra and $C$ is an absolutely integrally closed quasi-local domain.
\end{theorem}

\begin{proof}
See \cite[Theorem 17.5.96]{GR18}.
\end{proof}

\begin{Problem}
Let $(R,\fm)$ be a complete Noetherian local domain of mixed characteristic. Then can one construct a big Cohen-Macaulay $R$-algebra $T$ such that $T$ has bounded $p$-power roots of $p$ or equivalently, the radical ideal $\sqrt{pT}$ is finitely generated?
\end{Problem}

So far, big Cohen-Macaulay algebras constructed using perfectoids necessarily admit $p$-power roots of $p$ and we do not know if the construction as stated in the problem is possible.

\begin{Problem}
Let the notation be the same as that of Theorem \ref{RiemannExt}. Then under what condition is it true that $
\widehat{\varprojlim_{j>0} A^{j\circ}} \hookrightarrow \varprojlim_{j>0} \widehat{A^{j\circ}}$ is an isomorphism?
\end{Problem}

This is connected with the Mittag-Leffler condition. See \cite[Corollary 8.2.16]{GR18} for a relevant result. The paper \cite{GR18} also introduces and discusses the almost variant of the Mittag-Leffler condition. Bhatt and Scholze introduced the ``perfectoidization functor" over semiperfectoid rings. While our present paper is independent from \cite{BS22}, we are left to investigate the relationship between $\widehat{T}$ from Theorem \ref{AlmostCMalg} and $(A_{\infty g} \otimes_A R)_{\rm{perfd}}$.

\subsection{A construction of big Cohen-Macaulay modules}

We demonstrate a method of constructing a big Cohen-Macaulay module by using the $R$-algebra $T$ from Theorem \ref{AlmostCMalg}.

\begin{corollary}
\label{BigCMModule}
Let the notation be as in Theorem \ref{AlmostCMalg}. Set $M:=(pg)^{\frac{1}{p^\infty}}T$. Then $M$ is an ideal of $T$ that is $(pg)^{\frac{1}{p^\infty}}$-almost isomorphic to $T$, and $M$ is a big Cohen-Macaulay $R$-module. In other words, $H^i_\fm(M)=0$ for all $0 \le i \le \dim R-1$. 
\end{corollary}

\begin{proof}
Notice that $T$ is $pg$-torsion free and there is an isomorphism as $T$-modules: $T \cong (pg)^{\frac{1}{p^n}}T$. Consider the commutative diagram:
$$
\begin{CD}
T @>\times (pg)^{1-\frac{1}{p}}>> T @>\times (pg)^{\frac{1}{p}-\frac{1}{p^2}}>> T @>>> \cdots\\
@V\times pgVV@V\times (pg)^{\frac{1}{p}}VV @V\times (pg)^{\frac{1}{p^2}}VV \\
(pg)T @>>> (pg)^{\frac{1}{p}}T @>>> (pg)^{\frac{1}{p^2}}T @>>> \cdots \\
\end{CD}
$$
where the horizontal arrows in the bottom are natural injections, and the vertical arrows are bijections. Fix any $i<\dim R$. Applying the local cohomology to this commutative diagram, the bottom horizontal sequence becomes:
\begin{equation}
\label{BigCMModule1}
\varinjlim_{n>0} H_{\fm}^i((pg)^{\frac{1}{p^n}}T) \cong H_{\fm}^i(\varinjlim_{n>0}(pg)^{\frac{1}{p^n}}T) \cong H_{\fm}^i(M),
\end{equation}
where the first isomorphism uses the commutativity of cohomological functor with direct limit. The horizontal upper sequence becomes:
\begin{equation}
\label{BigCMModule2}
\varinjlim_{n} \Big\{H_{\fm}^i(T) \xrightarrow{\times(pg)^{1-\frac{1}{p}}} H_{\fm}^i(T) \xrightarrow{\times (pg)^{\frac{1}{p}-\frac{1}{p^2}}} H_{\fm}^i(T) \to \cdots  \Big\} \cong 0,
\end{equation}
because the local cohomology modules $H_{\fm}^i(T)$ are annihilated by $(pg)^{\frac{1}{p^n}}$ for any $n>0$ and $i<\dim R$. As $(\ref{BigCMModule1})$ and $(\ref{BigCMModule2})$
yield the isomorphic modules, we have the desired vanishing cohomology.

Since $R$ is a Noetherian local domain, there is a discrete valuation $v:R \to \mathbb{Z}_{\ge 0} \cup \{\infty\}$ with center on the maximal ideal. Then one extends $v$ as a $\mathbb{Q}$-valued valuation on $T$. One can use this valuation to deduce that $M \ne \fm M$ and the details are left as an exercise; see also \cite[Lemma 3.15]{MS12}.
\end{proof}

\section{Appendix A: Almost integrality under finiteness conditions}\label{AppendixA}
\label{non-adicBanachcompletion}

In this appendix, our aim is to give a proof to the following result (see Proposition \ref{propIC=CIC}). 
As in Corollary \ref{Cont3}, for a topological space $X$, we denote by $[X]$ the maximal separated quotient of $X$, thus defining the natural epimorphism $X \to [X]$.

\begin{proposition}
\label{propIC=CIC}
Let $A_0$ be a ring that is integral over a Noetherian ring, and let $t\in A_0$ be a regular element. Then an element $a\in A_0[\frac{1}{t}]$ is integral over $A_0$ if and only if it is almost integral over $A_0$. More precisely, for the Tate ring $A$ associated to $(A_0, (t))$, we have 
$$
(A_0)^+_A=A^\circ=\Big\{a \in A~\Big|~|a|_x\leq 1~\mbox{for any}~x \in [\textnormal{Spa}(A, (A_0)^+_A)]\Big\}. 
$$
\end{proposition}

The idea of our proof is to reduce the assertion to the situation of Corollary \ref{Cont3}, using \emph{Zariskization}. 
Let us recall its definition below (see also \cite[Chapter 0, 7.3(b)]{FK18} or \cite[Definition 3.1]{T18}).

\begin{definition}
Let $A$ be a ring with an ideal $I\subset A$. Then we denote by $A^{Zar}_I$ the localization $(1+I)^{-1}A$, and call it the \emph{$I$-adic Zariskization} of $A$. 
\end{definition}

We will utilize the following properties of Zariskization.

\begin{lemma}
\label{20200102}
Let $A\subset B$ be an integral ring extension and let $I\subset A$ be an ideal. Then the following assertions hold. 
\begin{enumerate}
\item
The induced ring map $A^{Zar}_I\to B^{Zar}_{IB}$ is also integral. 
\item
Let $\{A_\lambda\}_{\lambda\in\Lambda}$ be the filtered system of all module-finite $A$-subalgebras of $B$. 
Then we have a canonical isomorphism of rings $\varinjlim_\lambda (A_\lambda)^{Zar}_{IA_\lambda}\xrightarrow{\cong}B^{Zar}_{IB}$. 
\end{enumerate}
\end{lemma}

\begin{proof}
$(1)$: Set $B'=B\otimes_{A}A^{Zar}_I$. Then the map $A^{Zar}_I\to B^{Zar}_{IB}$ is given as the composite of the integral map $A^{Zar}_I\to B'$ and the canonical $B$-algebra map $B'\to B^{Zar}_{IB}$. Moreover, since $B'$ is $IB'$-adically Zariskian, we have the $B$-algebra map $B^{Zar}_{IB}\to B'$. Since the composite $B^{Zar}_{IB}\to B'\to B^{Zar}_{IB}$ is the identity map by the universal property, the map $B'\to B^{Zar}_{IB}$ is surjective. Hence the assertion follows. 

$(2)$: Since $B$ is integral over $A$, we have $\varinjlim_\lambda A_\lambda=B$. For each $\lambda\in\Lambda$, the map $A_\lambda\hookrightarrow B$ induces the $A_\lambda$-algebra map $\varphi_\lambda: (A_\lambda)^{Zar}_{IA_\lambda}\to B^{Zar}_{IB}$. Hence we have the $B$-algebra map $\varphi: \varinjlim_\lambda(A_\lambda)^{Zar}_{IA_\lambda}\to B^{Zar}_{IB}$. 
Now for any $x\in IB$, there exists some $\lambda\in \Lambda$ such that $1+x\in 1+IA_\lambda$. Hence $\varphi$ is injective. 
Set $C:=\varinjlim_\lambda(A_\lambda)^{Zar}_{IA_\lambda}$. Since $A^{Zar}_I\to C$ is integral by the assertion (1), $C$ is $IC$-adically Zariskian. 
Hence we obtain the $B$-algebra map $\psi: B^{Zar}_{IB}\to C^{Zar}_{IC}$, and the composite $\varphi\circ \psi$ is the identity map by the universal property. Therefore $\varphi$ is surjective. Thus the assertion follows. 
\end{proof}

\begin{corollary}
\label{12290101}
Let $A_0$ be a ring with a regular element $t \in A_0$. Put $A:=A_0[\frac{1}{t}]$ and $A':=({A_0})^{Zar}_{(t)}[\frac{1}{t}]$. Then the inclusion $A_0\hookrightarrow (A_0)^+_{A}$ induces an isomorphism $(({A_0})^{Zar}_{(t)})^+_{A'}\xrightarrow{\cong}((A_0)^+_A)^{Zar}_{(t)}$. 
\end{corollary}

\begin{proof}
Since integrality of a ring extension is preserved under localization, it suffices to show that $((A_0)^+_A)^{Zar}_{(t)}\cong (A_0)^+_A\otimes_{A_0}(A_0)^{Zar}_{(t)}$. 
First, we have an isomorphism $\varinjlim_{\lambda}(A_\lambda)^{Zar}_{(t)}\xrightarrow{\cong} ((A_0)^+_A)^{Zar}_{(t)}$ by Lemma \ref{20200102} (2). 
Moreover for each $\lambda\in\Lambda$, there exists some $m>0$ for which $t^mA_\lambda\subset A_0$. Then, since $1+t^{m+1}A_\lambda\subset 1+tA_0$, we have $(A_\lambda)^{Zar}_{(t)} \cong (A_\lambda)^{Zar}_{(t^{m+1})}\cong A_\lambda\otimes_{A_0}(A_0)^{Zar}_{(t)}$. Thus the assertion follows. 
\end{proof}

Now we can complete the proof of Proposition \ref{propIC=CIC}.

\begin{proof}[Proof of Proposition \ref{propIC=CIC}]
Set $X=\Spa(A, (A_0)^+_A)$. Since we know that 
$$
(A_0)^+_A\subset A^\circ\subset \Big\{a \in A~\Big|~|a|_x\leq 1~\mbox{for any}~x \in [X]\Big\},
$$
it suffices to show the reverse inclusion. 
Pick $c\in A$ such that $|c|_x\leq 1$ for any $x\in [X]$. 
By assumption, there exists a Noetherian subring $R\subset A_0$ such that $t\in R$ and the filtered system $\{R_\lambda\}_{\lambda\in \Lambda}$ of all module-finite $R$-subalgebras in $A_0$ satisfies $A_0=\varinjlim_{\lambda}R_\lambda$. 
Then by Lemma \ref{20200102},  $A'_0:=\varinjlim_\lambda (R_\lambda)^{Zar}_{(t)}$ is integral over a Noetherian ring $R^{Zar}_{(t)}$. Let $A'$ be the Tate ring associated to $(A'_0, (t))$, and $X'=\Spa(A', (A'_0)^+_{A'})$. 
Then Corollary \ref{Cont3} implies that  
$$
(A'_0)^+_{A'}=(A')^\circ=\Big\{a \in A'~\Big|~|a|_{x'}\leq 1~\mbox{for any}~x' \in [X']\Big\}. 
$$
Moreover, for the continuous ring map $\psi: A\to A'$, we have $|\psi(c)|_{x'}\leq 1$ for any $x'\in X'$ by assumption. Thus we find that $\psi(c)\in (A'_0)^+_{A'}$. On the other hand, $A'_0\cong (A_0)^{Zar}_{(t)}$ by Lemma \ref{20200102} and hence we have 
$$
((A_0)^+_A)^{Zar}_{(t)}\cong (A'_0)^+_{A'}
$$
by Lemma \ref{12290101}. Since the map $(A_0)^+_A\to ((A_0)^+_A)^{Zar}_{(t)}$ becomes an isomorphism after $t$-adic completion, 
one can deduce from Beauville-Laszlo's lemma (Lemma \ref{Beauville-Laszlo}) that the diagram of ring maps 
\[\xymatrix{
(A_0)^+_A\ar[r]\ar[d]&((A_0)^+_A)^{Zar}_{(t)}\ar[d]\\
A\ar[r]_{\psi}&A'
}\]
is cartesian. Thus we obtain $c\in (A_0)^+_A$, as wanted. 
\end{proof}

\section{Appendix B: Almost vanishing of derived limits}
\label{AppendixB}

We give a self-contained account of the proof of the Riemann's extension theorem, as proved by Andr\'e and its consequence on the almost vanishing of the derived limits of a certain tower of perfectoid algebras. This appendix is also meant to help the reader understand Andr\'e's papers \cite{An1} and \cite{An2}. See \cite[Th\'eor\`eme 4.2.2]{An1} and \cite[Proposition 4.4.1]{An1} for the following results, respectively.

\begin{theorem}[Riemann's extension theorem for perfectoid algebras]
\label{Hebbarkeits1}
Fix a perfectoid $K$-algebra $\mathcal{A}$, where $K$ is a perfectoid field with a nonzero element $\varpi \in K^\circ$ admitting all $p$-power roots, and let $g \in \mathcal{A}^\circ$ be an element that admits a compatible system of $p$-power roots $\{g^{\frac{1}{p^m}}\}_{m>0}$, such that $g$ is a $(\varpi)^{\frac{1}{p^\infty}}$-almost regular element of $\mathcal{A}^\circ/(\varpi^r)$ for any fixed $r \in \mathbb{N}[\frac{1}{p}]$.
Then there is an isomorphism:
$$
\varprojlim_{j>0} \mathcal{A} \big\{\frac{\varpi^j}{g}\big\}^\circ \cong g^{-\frac{1}{p^\infty}} \mathcal{A}^\circ.
$$
\end{theorem}

\begin{proof}
Throughout the proof, we fix $r \in \mathbb{N}[\frac{1}{p}]$ and for a given $j \in \mathbb{N}$, we set
$$
A^j_0:=\mathcal{A}^\circ/(\varpi^r)\big[\big(\frac{\varpi^j}{g}\big)^{\frac{1}{p^\infty}}\big].
$$
Then the natural ring map $\eta_j:A^{j+1}_0 \to A^{j}_0$ is defined by $\frac{\varpi^{j+1}}{g} \mapsto \varpi \cdot \frac{\varpi^j}{g}$ and hence $\{A^j_0\}_{j \in \mathbb{N}}$ forms an inverse system.\footnote{Notice that our standing hypothesis on the sequence $\varpi,g$ shows that $A_0^j$ is $(\varpi)^{\frac{1}{p^\infty}}$-almost isomorphic to $(\mathcal{A}^{\circ})^{[j]}$ as defined in \cite[(4.3)]{An1} in view of Lemma \ref{blowupring1}.} Now we claim that the natural map
\begin{equation}
\label{inverselimitiso}
f:\mathcal{A}^\circ/(\varpi^r) \to \varprojlim_j A^j_0
\end{equation}
is a $(g)^{\frac{1}{p^\infty}}$-almost isomorphism. First we show that $f$ is $(g)^{\frac{1}{\infty}}$-almost injective. But we remark that the localization map $\mathcal{A}^\circ/(\varpi^r) \to \mathcal{A}^\circ/(\varpi)^r[\frac{1}{g}]$ factors as $\mathcal{A}^\circ/(\varpi^r) \to A^j_0 \to \mathcal{A}^\circ/(\varpi)^r[\frac{1}{g}]$ which is $(g)^{\frac{1}{p^\infty}}$-almost injective by our hypothesis. So it suffices to treat the cokernel of $f$. Consider the commutative diagram
$$
\begin{CD}
\mathcal{A}^\circ/(\varpi^r) @>f_{j+c}>> A^{j+c}_0 \\
@VVV @VV\eta_j\circ \cdots \circ \eta_{j+c-1}V \\
\mathcal{A}^\circ/(\varpi^r) @>f_j >> A^j_0.\\
\end{CD}
$$
For any fixed $m>0$, we must show that
\begin{equation}
\label{inverselimitiso2}
\eta_j\circ \cdots \circ \eta_{j+c-1}\Big(g^{\frac{1}{p^m}} \cdot \Big(\frac{\varpi^{j+c}}{g}\Big)^{\frac{i}{p^n}}\Big) \in \im(f_j)~\mbox{in}~A^j_0~\mbox{for}~\forall c\ge p^mr~\mbox{and}~i,n>0.
\end{equation}

First assume that $p^n \le ip^m$. Then as $c\ge p^mr$, we have $\varpi^r \mid \varpi^{\frac{c}{p^m}}$. Hence $\varpi^r \mid \varpi^{\frac{ci}{p^n}}$. This shows that
$$
\Big(\frac{\varpi^{j+c}}{g}\Big)^{\frac{i}{p^n}}=\varpi^{\frac{ci}{p^n}} \cdot \Big(\frac{\varpi^{j}}{g}\Big)^{\frac{i}{p^n}}=0,
$$
which gives $(\ref{inverselimitiso2})$ under the assumption $p^n \le ip^m$. Second assume that $p^n > ip^m$. In this case, we get the following
$$
g^{\frac{1}{p^m}} \cdot \Big(\frac{\varpi^{j+c}}{g}\Big)^{\frac{i}{p^n}}=g^{\frac{1}{p^m}-\frac{i}{p^n}} \cdot \varpi^{\frac{ji+ci}{p^n}} \in \im(f_r),
$$
which gives $(\ref{inverselimitiso2})$ under the assumption $p^n > ip^m$. So we proved that for any given $m>0$, if we choose $c$ such that $c \ge p^mr$, then it follows that
\begin{equation}
\label{cokernelmap}
\im\Big(\coker(f_{j+c}) \to \coker(f_j)\Big)~\mbox{is annihilated by}~g^{\frac{1}{p^m}}.
\end{equation}
So combining $(\ref{inverselimitiso})$ and \cite[Lemma 6.4]{Sch12} together, we find that 
$$
\mathcal{A}^\circ/(\varpi^r) \to 
\varprojlim_{j>0} \Big(\mathcal{A} \big\{\frac{\varpi^j}{g}\big\}^\circ/(\varpi^r)\Big) 
$$
is a $(\varpi g)^{\frac{1}{p^\infty}}$-almost isomorphism. After taking the inverse limits, $$
\mathcal{A}^\circ \to 
\varprojlim_{j>0} \mathcal{A} \big\{\frac{\varpi^j}{g}\big\}^\circ
$$
is a $(\varpi g)^{\frac{1}{p^\infty}}$-almost isomorphism. Applying the functor of almost elements $(\varpi g)^{-\frac{1}{p^\infty}}(~)$ on both sides, we obtain the desired isomorphism.
\end{proof}

\begin{corollary}[Almost vanishing of $\varprojlim^1$]
\label{vanishinglimone}
Let the notation and hypotheses be the same as in Theorem \ref{Hebbarkeits1} and fix $r \in \mathbb{N}[\frac{1}{p}]$. Set
$$
A^j_0:=\mathcal{A}^\circ/(\varpi^r)\big[\big(\frac{\varpi^j}{g}\big)^{\frac{1}{p^\infty}}\big]~\mbox{for}~j \in \mathbb{N}.
$$
Then the inverse system $\{A^j_0\}_{j \in \mathbb{N}}$ gives a $(g)^{\frac{1}{p^\infty}}$-almost vanishing
$$
{\varprojlim_{j \in \mathbb{N}}}^1 A^j_0 \approx 0.
$$
\end{corollary}

\begin{proof}
Without loss of generality, we may assume that $g$ is a regular element of $\mathcal{A}^\circ/(\varpi^r)$ for any $r \in \mathbb{N}$. Keep the notation as in Theorem \ref{Hebbarkeits1}. Then we have a commutative diagram:
$$
\begin{CD}
0 @>>> \mathcal{A}^\circ/(\varpi^r) @>f_{j+c}>> A^{j+c}_0 @>>> \coker(f_{j+c}) @>>> 0\\
@.@VVV @VV\eta_j\circ \cdots \circ \eta_{j+c-1}V @VVf_{j+c,j}V \\
0 @>>> \mathcal{A}^\circ/(\varpi^r) @>f_j >> A^j_0 @>>> \coker(f_j) @>>> 0\\
\end{CD}
$$
where each horizontal sequence is exact. 

Fix an integer $m>0$. Let us put $N_j:=\coker(f_j)$. Then choose $c(m) \in \mathbb{N}$ according to $(\ref{cokernelmap})$ such that the image of the map $N_{j+c(m)} \to N_j$ is annihilated by $g^{\frac{1}{p^m}}$. Put $k(m,n):=1+n \cdot c(m)$. Then for any fixed $m>0$, $\{A_0^{k(m,n)}\}_{n \in \mathbb{N}}$ forms a cofinal subsystem of the inverse system $\{A^j_0\}_{j \in \mathbb{N}}$. In other words,
$$
{\varprojlim_{j \in \mathbb{N}}}^1 A^j_0 \cong {\varprojlim_{n \in \mathbb{N}}}^1 A^{k(m,n)}_0.
$$
Moreover, $\{\mathcal{A}^\circ/(\varpi^r)\}_{n \in \mathbb{N}}$ is a constant inverse system. So we have ${\varprojlim_{j \in \mathbb{N}}}^1 \mathcal{A}^\circ/(\varpi^r)=0$ and by applying the derived limits to the exact sequence: $
0 \to \{\mathcal{A}^\circ/(\varpi^r)\}_{n \in \mathbb{N}} \to \{A_0^{k(m,n)}\}_{n \in \mathbb{N}} \to \{N_{k(m,n)}\}_{n \in \mathbb{N}} \to 0$, we get an isomorphism:
$$
{\varprojlim_{j \in \mathbb{N}}}^1 A^j_0 \cong {\varprojlim_{n \in \mathbb{N}}}^1N_{k(m,n)}.
$$
In particular, the right-hand side does not depend on the choice of $m \in \mathbb{N}$. Now we claim that
\begin{equation}
\label{derivedinverse1}
g^{\frac{1}{p^m}} \cdot \big({\varprojlim_{n \in \mathbb{N}}}^1N_{k(m,n)}\big)=0.
\end{equation} 
To prove this, we may replace the system $\{N_j\}_{j \in \mathbb{N}}$ with $\{N_{k(m,n)}\}_{n \in \mathbb{N}}$ to simplify the notation. 

Choose any element $(\beta_i)_{i \in \mathbb{N}} \in \prod_{j \in \mathbb{N}}N_{j}$ and set $\gamma_i:=g^{\frac{1}{p^m}} \beta_i$. By using $(\ref{cokernelmap})$, the infinite sum
$$
\alpha_k:=\sum_{j=k}^\infty f_{j,k}(\gamma_j)
$$
makes sense, where $f_{j,k}:N_j \to N_k$ is the map given above. Then $(\alpha_k)_{k \in \mathbb{N}}$ maps to $(\gamma_k)_{k \in \mathbb{N}}$ under the mapping:
$$
\Delta:\prod_{j \in \mathbb{N}}N_{j} \to \prod_{j \in \mathbb{N}}N_{j}
$$
defined by the formula $\Delta((x_i)_{i \in \mathbb{N}}):=\big(x_i-f_{i+1,i}(x_{i+1})\big)_{i \in \mathbb{N}}$ for $(x_i)_{i \in \mathbb{N}} \in \prod_{j \in \mathbb{N}} N_j$, which gives $(\ref{derivedinverse1})$. See \cite[Proposition 3.5.7]{Wei94} for the Mittag-Leffler condition and its relation to the derived inverse limits and \cite[Lemma 2.4.2]{GR03} for its almost variants. Since $m$ was arbitrarily chosen, we conclude that ${\varprojlim_{j \in \mathbb{N}}}^1 A^j_0$ is $(g)^{\frac{1}{p^\infty}}$-almost zero, as desired.
\end{proof}

\section{Appendix C: Almost Galois extensions}
\label{AlmostG-GaloisExt}

We make use of Galois theory of commutative rings in making reductions in steps of proofs of some results in the present paper. Let $A \to B$ be a ring extension and let $G$ be a finite group acting on $B$ as ring automorphisms.

\begin{definition}
We say that $B$ is a \textit{G-Galois extension} of $A$, if 
$A=B^G$ and the natural ring map
$$
B \otimes_A B \to \prod_G B;~b \otimes b' \mapsto (\gamma(b)b')_{\gamma \in G}
$$
is an isomorphism. 
\end{definition}

Some fundamental results about Galois extensions are found in \cite{An1} or \cite{Fo17}. A source of the definition of \textit{almost G-Galois extension} is \cite{An1}. Here we cite some related results for the sake of readers.

\begin{definition}
\label{AlmostG-GaloisExt2}
Let $(A,I)$ be a basic setup and let $B \to C$ be an $A$-algebra map with $G$ acting on $C$ as ring automorphisms. Then we say that $B \to C$ is an \textit{I-almost G-Galois extension} if the natural map $B \to C$ induces an $I$-almost isomorphism $B \xrightarrow{\approx} C^G$ and
$$
C \otimes_B C \to \prod_G C;~c \otimes c' \mapsto (\gamma(c)c')_{\gamma \in G}
$$
is an $I$-almost isomorphism.
\end{definition}

\begin{proposition}
\label{AlmostG-GaloisExt3}
Let $(A,I)$ be a basic setup and let $B \to C$ be an $A$-algebra map with a finite group $G$ acting on $C$ such that $B \to C$ factors as $B \to C^G \to C$. Then the following assertions hold.
\begin{enumerate}
\item
If $B \to C$ is $I$-almost $G$-Galois, then it is $I$-almost finite \'etale of constant rank $|G|$.

\item
Assume that $B \to D$ is an $A$-algebra map. Let $G$ act on the base change $D \otimes_B C$ through the second factor. If $B \to C$ is $I$-almost $G$-Galois, then so is $D \to D \otimes_B C$. Conversely, if $D$ is faithfully flat over $B$ and $D \to D \otimes_B C$ is $I$-almost $G$-Galois, then so is $B \to C$.

\item
Assume that $B \to C$ is $I$-almost $G$-Galois and assume that $B \to D \to C$ is a factorization of rings such that $D \to C$ is $I$-almost $H$-Galois for a subgroup $H \subset G$. Then the canonical map
$$
D \otimes_B C \to \prod_{G/H} C
$$
is an $I$-almost isomorphism and $B \to C$ is $I$-almost finite \'etale of constant rank $|G/H|$.
\end{enumerate}
\end{proposition}

\begin{proof}
See \cite[Proposition 1.9.1]{An1} for the proof.
\end{proof}

In the next lemma, $S_n$ will denote the group of permutations of $n$ elements.

\begin{lemma}
\label{AlmostG-GaloisExt4}
The following assertions hold.
\begin{enumerate}
\item
Assume that $B \to C$ is \'etale of constant rank $r$. Then there is an $S_r$-Galois extension $B \to D$ which factors as $B \to C \to D$ such that $C \to D$ is an $S_{r-1}$-Galois extension.

\item
Let $C$ be a ring on which a finite group $G$ acts as ring automorphisms. Set $B:=C^G$. Suppose that there is a factorization $B \to D \to C$ such that $G(D)=D$ and $B \to D$ is a $G$-Galois extension. Then $D \to C$ is bijective.
\end{enumerate}
\end{lemma}

\begin{proof}
See \cite[Lemme 1.9.2]{An1} and \cite[Lemme 1.9.3]{An1} for the proof.
\end{proof}

\begin{acknowledgement}
The authors are grateful to Professor K. Fujiwara for encouragement and comments on this paper. 
They also thank S. Ishiro for helpful discussions. Finally, they thank the referee for a meticulous reading of the paper. The first author was partially supported by JSPS Grant-in-Aid for Early-Career Scientists 23K12952. The second author was partially supported by JSPS Grant-in-Aid for Scientific Research(C) 23K03077.
\end{acknowledgement}

\end{document}